\documentclass[12pt]{article}
\usepackage{fontspec}
\usepackage{amsmath,amssymb,amsthm}
\usepackage{xcolor}

\usepackage{hyperref}

\newcommand{\conc}{\mathcal{Q}}

\newtheorem{theorem}{Theorem}[section]
\newtheorem{lemma}[theorem]{Lemma}
\newtheorem{prop}[theorem]{Proposition}
\newtheorem{corollary}[theorem]{Corollary}
\newtheorem{conjecture}[theorem]{Conjecture}

\theoremstyle{definition}
\newtheorem{definition}[theorem]{Definition}

\renewenvironment{proof}{\noindent{\bf Proof}}{\hspace*{\fill}$\Box$}

\newenvironment{proofof}[1]{%
\noindent {\bf Proof of #1}}%
{\hspace*{\fill}$\Box$}

\newcommand{\pr}{\mathbb P}
\newcommand{\E}{\mathbb E}
\newcommand{\I}{\mathbb I}
\def\Var{{\mathrm{Var}}\,}
\newcommand{\SE}{\mathrm{SE}}

\newcommand{\epsgrid}{\epsilon_{grid}}
\newcommand{\tse}{{t_{SE}}}
\newcommand{\tbal}{{t_{SE}^{bal}}}
\newcommand{\topt}{t}
\newcommand{\supp}{\mathrm{supp}}
\newcommand{\lesspeaked}{\preccurlyeq}
\newcommand{\lesspeakedeps}{\lesspeaked_\epsilon}
\newcommand{\LPMSigma}{\Sigma_1}
\newcommand{\round}[1]{\left[ #1 \right]}

\newcommand{\Clogconcvariancenewcol}{2^{\frac{2}{3}}}
\newcommand{\Clogconcdomination}{3 \cdot 2^{\frac{4}{9}}}
\newcommand{\Cbalancedcontinuitylargea}{4 \cdot 2^{\frac{2}{3}} \cdot 3^{\frac{1}{2}}}

\newcommand{\Cbalancedcontinuitylargebr}{12}
\newcommand{\Cbalancedcontinuitylarge}{14}
\newcommand{\Cpeakednesslone}{\frac{2^{\frac{1}{3}}}{4}}
\newcommand{\Cpeakednessloneb}{\frac{65536 \cdot 2^{\frac{2}{3}}}{\epsilon^{4}}}
\newcommand{\Cpeakednessltwo}{\frac{320 \cdot 2^{\frac{1}{3}}}{\epsilon^{2}}}
\newcommand{\Clogconcdominationdelta}{\frac{2^{\frac{2}{3}} \delta^{3}}{108}}

\newcommand{\Cfewdroppedsimpler}{\frac{70}{\delta^{3}}}
\newcommand{\Cpeakednesslonevardelta}{\frac{16777216 \cdot 2^{\frac{2}{3}}}{\delta^{4}}}
\newcommand{\Cbalancedapproxoptimalfinala}{\frac{2684354560000 \cdot 2^{\frac{2}{3}}}{\delta^{4}}}
\newcommand{\Cbalancedapproxoptimalfinalb}{\frac{128000 \cdot 2^{\frac{1}{3}}}{\delta^{2}}}
\newcommand{\Cbalancedapproxoptimalfdeltasolved}{\frac{\delta^{3}}{\delta^{3} + 3456 \cdot 2^{\frac{1}{3}}}}
\newcommand{\Clesspeakedmedium}{72957}
\newcommand{\Ctsealphaapprox}{\frac{140 K^{4}}{\delta^{3}}}
\newcommand{\Cbalancedcontinuous}{8}
\newcommand{\Cbalancedcontinuoustwice}{16}
\newcommand{\Cbalancedcontinuousfourtimes}{32}
\newcommand{\Cmediumcontnew}{17^{\frac{1}{2}} \cdot 2^{\frac{2}{3}}}

\newcommand{\Cmidsizealphacontinuity}{39}
\newcommand{\Cbalancedcontinuitylargeapp}{28}
\newcommand{\Ctwobalanced}{99}
\newcommand{\CtwobalancedKbound}{1.49 \cdot 10^{16}}
\newcommand{\Ctse}{\frac{6.89 \cdot 10^{66}}{\delta^{27}}}
\newcommand{\Cnineteen}{7}
\newcommand{\Clpmhip}{\left(d + 1\right)^{d}}
\newcommand{\Clpmbee}{16 \cdot 2^{\frac{1}{2}}}

\title{An asymptotically optimal bound for the concentration function of a sum of independent
integer
random variables}

\author{Valentas Kurauskas\footnote{Faculty of Mathematics and Informatics, Vilnius University, Naugarduko 24, LT-03225 Vilnius, Lithuania. Email: valentas.kurauskas@mif.vu.lt.}}

\date{\today}

\begin{document}

\maketitle

\begin{abstract}
    For a 
    random
    variable $X$ define $\conc(X) = \sup_{x \in \mathbb{R}} \pr(X=x)$.

    Let $X_1, \dots, X_n$ be independent integer random variables.
    Suppose $\conc(X_i) \le \alpha_i \in (0,1]$ for each $i \in \{1, \dots, n\}$.
    Juškevičius (2023) conjectured that 
    $\conc(X_1 + \dots +X_n) \le \conc(Y_1 + \dots+ Y_n)$
    where $Y_1, \dots, Y_n$ are independent and
    $Y_i$ is a random integer variable with $\conc(Y_i) =\alpha_i$ that has the smallest variance, i.e. the distribution of $Y_i$ has probabilities $\alpha_i, \dots, \alpha_i, \beta_i$ or probabilities $\beta_i, \alpha_i, \dots, \alpha_i$ on some interval of integers, where $0 \le \beta_i < \alpha_i$.

    We prove this conjecture asymptotically:
    i.e., we show that for each $\delta > 0$ there is $V_0 = V_0(\delta)$ such that
    if $\Var (\sum Y_i) \ge V_0$ then $\conc(\sum X_i) \le (1+\delta) \conc(\sum Y_i)$.

    This implies an analogous asymptotically optimal inequality for concentration at a point
    when $X_1$, $\dots$, $X_n$ take values in a 
    separable Hilbert space.

    Our long and technical argument relies on several non-trivial previous
    results including an inverse Littlewood--Offord theorem and
    an approximation in total variation distance of sums of multivariate
    lattice random vectors by a discretized Gaussian distribution.

%
\end{abstract}

\section{Introduction}
\label{sec.introduction}


The \emph{concentration function} 
of
a random variable $X$ at value $t > 0$ will be defined as $\conc(X, t) = \sup_{x \in \mathbb{R}} \pr(X \in (x, x+t))$.
We will denote the maximum concentration at a point by $\conc(X) = \sup_{x \in \mathbb{R}} \pr(X=x)$.
If 
$X$ is an integer random variable, then
simply
$\conc(X,1) = \conc(X) = \sup_{x \in \mathbb{Z}} \pr(X = x)$. 

For any random variable and a positive integer $k$ we define $\conc_k(X)= \sup_{A: |A| = k} \pr(X \in A)$, that is $\conc_k(X)$ is the sum of $k$ largest probability masses of (the law of) $X$.
If the law of $X$ is $\mu$, we will use the symbols $X$ and $\mu$ interchangeably, so $\conc(\mu)$ is the same as $\conc(X)$ and similarly for other functionals in this paper.

Let $\alpha_1, \dots, \alpha_n \in (0,1]$. What is the maximum possible
value of $\conc(S_n)$ (or $\conc(S_n, L)$) if $S_n = X_1 + \dots + X_n$ and $X_1, \dots, X_n$ 
are independent random variables such that for each $i$ we have $\conc(X_i) \le \alpha_i$ (or $\conc(X_i, \lambda_i) \le \alpha_i)$)? 
These questions and their variants have been the subject of quite a lot of research 
over the past 
90 years.
They originated in the works of L\'evy, Doeblin and Kolmogorov, also in those of Littlewood and Offord, and were further studied by Esseen, Kesten, Rogozin and others, see, for example, the paper \cite{miroshnikovrogozin} and the references therein.
Nevertheless, 
for a general sequence of constraints,
inequalities obtained using classical characteristic function methods, e.g. \cite{miroshnikovrogozin,ushakov}
can be
suboptimal by at least a constant factor.

A subtler type of results are related to determining sequences
of random variables that satisfy the same conditions as $(X_i)$
and maximize the concentration function value of the sum.
If
the extremizing sequence has a simple structure, 
we can easily get sharp quantitative bounds.
For the interval version
(in the case when $\lambda_1 = \dots \lambda_n = L$)
optimal results of this kind have been obtained in~\cite{tj2024} (single-dimensional case)
and \cite{jk2024} (asymptotically, high-dimensional case).
%
The question of optimal bounds for the pointwise
version of the problem has remained open.
While the 
two versions
of the concentration function problem
look
similar, 
they are not equivalent and the extremizers are different in general: they are symmetric in the interval case, see \cite{tj2024, jk2024},
and non-symmetric in the pointwise case, see \cite{jk2021} and the definition of $\nu_\alpha$ below.
This makes the pointwise problem more technically challenging and 
we need to use
completely different ideas
in this work.

%
It is known, see Ushakov~\cite{ushakov}, that the general problem of 
maximizing
the pointwise concentration $\conc(S_n)$ given pointwise bounds for the summands can be reduced to
the case of \emph{integer} random variables, see Corollary~\ref{cor.general} below. 
To our knowledge, the best previous results for the concentration function problem for integer random variables, or its equivalent formulation in the language of sums, have been obtained in \cite{HLP,lev1998,madimanwangwoo2018}. 
For example, building on Gabriel's inequality, see Hardy, Littlewood and P\'{o}lya~\cite{HLP}, Lev~\cite{lev1998} showed that if each $X_i$ is a \emph{uniform random variable} on a set of size $k_i$ then $\conc(S_n) \le \conc(Y_1 + \dots + Y_n)$
where $Y_1, \dots, Y_n$ are independent and $Y_i$ is distributed uniformly on the interval $\{0, \dots, k_i - 1\}$. In other words, in the uniform case, removing ``gaps'' in the supports of all the random summands can only increase the 
maximum probability of the sum. Madiman, Wang and Woo~\cite{madimanwangwoo2018} obtained a similar result that works for a more general family of distributions of integers that includes all two-point distributions, see Theorem~\ref{thm.madimanwangwoo} below.
Given $\alpha \in (0,1]$ let $\nu_\alpha$ be the probability measure
which satisfies 
\[
\nu_\alpha(\{i\}) = \alpha \mbox{ for } i \in \{0, \dots, \lfloor \alpha^{-1} \rfloor - 1\}
\]
and 
\[
 \nu_\alpha(\{\lfloor \alpha^{-1} \rfloor\}) = 1 - \alpha \lfloor \alpha^{-1} \rfloor.
\]
For example, when $\alpha \ge \frac 1 2$, $\nu_\alpha$ is 
the Bernoulli distribution with parameter $1-\alpha$ and when $\alpha = \frac 1 k$ for an integer $k$, it is the uniform distribution on $\{0,\dots,k-1\}$. 

It will be important that $\nu_\alpha$ has minimum 
variance among probability distributions
$\mu$ on integers with $\conc(\mu) \le \alpha$, this can be seen using a standard argument and the main result of \cite{ak2025}.

In this work we prove the following theorem.
\begin{theorem}\label{thm.tse}
    For any $\delta > 0$ there exists a number $V_0(\delta)$ such that the following holds.

    Let $X_1, \dots, X_n$ be a sequence of independent integer random variables such that $\conc(X_i) \le \alpha_i$,
    $\alpha_i \in (0,1]$ and let 
    $S_n = \sum_{i=1}^n X_i$. 
    Let $Y_1, \dots, Y_n$ be independent, $Y_i \sim \nu_{\alpha_i}$.

    If \begin{equation}\label{eq.tse.var}
    \Var (\sum_{i=1}^n Y_i) 
    \ge
    V_0(\delta)
    \end{equation} then 
    \begin{equation}\label{eq.tse}
        \conc(S_n) \le \conc(\sum_{i=1}^n a_i Y_i ) (1 + \delta)
    \end{equation}
    for some $a_1, \dots, a_n \in \{-1, 1\}$. 
\end{theorem}
As 
the sequence $(a_i Y_i)$ itself satisfies the conditions of the theorem, (\ref{eq.tse}) is optimal up to the factor $(1+\delta)$.

For a random element $X$ in a Hilbert space $\mathbb{H}$ we 
can extend the notation by defining $Q(X) = \sup_{x \in \mathbb{H}} \pr(X = x)$.
The following corollary follows immediately by applying
the reduction of~\cite{ushakov}.

\begin{corollary}\label{cor.general}
Theorem~\ref{thm.tse} extends verbatim to the case 
when $X_1, \dots, X_n$ are independent random elements 
from a separable Hilbert space $\mathbb{H}$ over the reals.
In particular, the constant $V_0(\delta)$ does not depend on $\mathbb{H}$.
\end{corollary}
Since we can embed the random variables $Y_i$
into $\mathbb{H}$ by multiplying each by a fixed vector
from $\mathbb{H}$, the corollary is also optimal up to
the constant factor $(1+\delta)$.

Prior to this work, 
similar asymptotically tight 
upper bounds 
were available in the ``uniform'' case
$\alpha_1 = \dots = \alpha_n$ \cite{jk2021, ushakov},
and exact ones could be deduced,
via the reduction of~\cite{ushakov}, 
in the case $\alpha_i^{-1} \in [1,2] \cup \mathbb{Z}$ for all $i \in \{1,\dots,n\}$ from
the results of \cite{lev1998,madimanwangwoo2018} mentioned above.
The latter results could also be used to show bounds similar to (\ref{eq.tse}) for more general sequences $(\alpha_i)$,
however, with an extra constant factor of up to $\frac {2 \sqrt 6} 3$.
Alternatively, they can be derived from a very general framework of Madiman, Melbourne and Xu~\cite{madimanmelbournexu2017}.

This work has been motivated by a conjecture that an exact version of (\ref{eq.tse}) holds, see also \cite{bobkovmarsigliettimelbourne,jk2021,madimanwangwoo2018}:
\begin{conjecture}(Juškevičius~\cite{tjp})\label{conj.main}
  Let $X_1, \dots, X_n$ and $Y_1, \dots, Y_n$ be as in Theorem~\ref{thm.tse}.
  Then there are $a_1, \dots, a_n \in \{-1,1\}$
  such that $\conc(\sum_{i=1}^n X_i) \le \conc(\sum_{i=1}^n a_i Y_i)$.
\end{conjecture}

Let us explain 
the need of a
general sequence $(a_i) \in \{-1,1\}^n$ in the statements of both the theorem and the conjecture. 
Consider the special case $\alpha_1 = \dots = \alpha_n = \alpha$. When $n$ is even, the balanced choice, for example $a_i = (-1)^i, i \in \{1,\dots,n\}$
maximizes $\conc(\sum a_i Y_i)$ \cite{jk2021, singhal}. 
However, when $n$ is odd and, for example, $\alpha \in (\frac 1 2,1)$, the optimal number of -1s and 1s can depend
on $\alpha$ in a non-trivial way \cite{singhal}. On the other hand, it seems plausible that
\emph{asymptotically}, in the general setting of Theorem~\ref{thm.tse}, the choice of these signs has only a second order effect.
%

Theorem~\ref{thm.tse} shows that 
Conjecture~\ref{conj.main}
is 
close to being true 
in the limit; however, the constant $V_0(\delta)$ that arises
in our proof is huge and not explicit. 
If 
the conjecture 
is indeed true, its proof 
appears to be 
beyond the
reach of the methods
we are 
currently
aware of.

The rest of the paper consists of two 
major 
parts: Part I (Section~\ref{sec.main_proof}) and 
 Part II (Section~\ref{sec.nontrivial}). 
The final proof of Theorem~\ref{thm.tse} is given in Section~\ref{subsec.entry_point} of Part I. 
The most difficult result needed for the proof,
Lemma~\ref{lem.less_peaked_medium}, is stated in Part I,
but the content of the whole Part II is devoted to its proof.



A reader who wishes to focus on the less trivial Part II 
may 
 proceed straight to
Section~\ref{sec.nontrivial} after reading just the definitions in Section~\ref{sec.definitions} and
Section~\ref{subsec.epsilon_less_peaked}, and the statement
of  Lemma~\ref{lem.less_peaked_medium}.

Intuitively, Lemma~\ref{lem.less_peaked_medium} that we prove in Part II can be seen as a special case of
Theorem~\ref{thm.tse} where for $i\in\{1,\dots,n-1\}$ the random variables $X_i$ are distributed according to some ``rearrangement'' of the distribution $\nu_{\alpha_i}$ (i.e. each $X_i$ has $\lfloor \alpha_i^{-1} \rfloor$ atoms with probability mass $\alpha_i$ and, if there is any, an atom with the remaining mass), for $i \in \{1,\dots, n-1\}$ the values $\alpha_i$ are bounded away from zero, ``quantized'' and ``balanced'' (Definition~\ref{def.balanced}),
while the last random variable $X_n$ is a uniform random variable on an arbitrary finite set of integers.

\section {Part I: Main proof}
\label{sec.main_proof}

In this part we state the necessary lemmas and prove the main result, Theorem~\ref{thm.tse}.

\subsection{Definitions}
\label{sec.definitions}


A measure $\mu$ on integers is called \emph{unimodal} if there is $m \in \mathbb{Z}$ such that $\mu(\{m\}) \ge \mu(\{m+1\}) \ge \dots$ and $\mu(\{m\}) \ge \mu(\{m-1\}) \ge \dots$. Any such $m$ is called a \emph{mode} of $\mu$.

In their work on rearrangement inequalities, Hardy, Littlewood and P\'{o}lya \cite{HLP} introduced the notation $\mu^+$, ${}^+\mu$ and $\mu^*$ 
which will be defined
as follows.
Let $\mu$ be a finite measure on 
integers. We call a measure $\mu'$ on $\mathbb{Z}$
a rearrangement of $\mu$ if there exists a bijection $f: \mathbb{Z} \to \mathbb{Z}$
such that $\mu'(\{x\}) = \mu(\{f(x)\})$ for each $x \in \mathbb{Z}$.
$\mu^+$ is the rearrangement of $\mu$ that satisfies
 $\mu^+(\{0\}) \ge \mu^+(\{1\}) \ge \mu^+(\{-1\}) \ge \mu^+(\{2\}) \ge \mu^+(\{-2\}) \ge \dots$. Similarly ${}^+\mu$ is the rearrangement 
 of $\mu$ that satisfies  ${}^+\mu(\{0\}) \ge {}^+\mu(\{-1\}) \ge {}^+\mu(\{1\}) \ge {}^+\mu(\{-2\}) \ge {}^+\mu(\{2\}) \ge \dots$. If $\mu^+ = {}^+\mu$ then we say that $\mu$ has a \emph{symmetric decreasing} rearrangement and denote this symmetric measure by $\mu^*$.
 Note that for a symmetric unimodal measure on integers 
 ${}^+\mu = \mu^+ = \mu^* = \mu$. When the law of an integer random variable $X$ is $\mu$, we let $X^+$, $^+X$ and $X^*$ denote random variables distributed
 according to $\mu^+$, ${}^+\mu$ and $\mu^*$ respectively, whenever these distributions exist.

\begin{definition}\label{def.extremal}
    Let $\alpha \in (0,1]$. Let $k =  \lfloor \alpha^{-1} \rfloor$.
    We call a (probability) measure $\mu$ \emph{extremal} for $\alpha$ if there is a set $S \subset \mathbb{Z}$, 
    $|S|=k$ and a number $b \in \mathbb{Z} \setminus S$ such that 
    $\mu(i)=\alpha$ for all $i \in S$ and $\mu(b) = 1-k\alpha$.

\end{definition}

Recall that an extremal measure for $\alpha$ obtained by taking $S=\{0, \dots, k-1\}$ and $b=k$ is denoted by $\nu_\alpha$.

\begin{definition}\label{def.standard_extremal} 
    Let $\alpha \in (0,1]$.
    We call a (probability) measure $\mu$ \emph{standard extremal} for $\alpha$ if it is the distribution of $X + i$ where $i$ is an integer and 
    $X \sim \nu_\alpha$ or $-X \sim \nu_\alpha$,
    i. e. the support of $\mu$ is an interval and the smallest atom, whenever it is unique, is on one of the endpoints of the interval. We denote the set of standard extremal measures for $\alpha$ by $\SE(\alpha)$.
    Any measure which is 
    standard extremal for some $\alpha$ will be called
    \emph{standard extremal}.
\end{definition}

\begin{definition}\label{def.optimal_se}
    Let $\alpha=(\alpha_1, \dots, \alpha_n) \in (0,1]^n$.
    A sequence of
    standard extremal measures $(\mu_i, i \in \{1, \dots,n\})$ will be called a 
    \emph{standard extremal sequence (of measures) for $\alpha$} 
    if
    for each $i$ $\mu_i$ is 
    standard extremal for $\alpha_i$.
    It will be called \emph{optimal} 
    if 
    its convolution maximizes the concentration over the standard extremal sequences for $\alpha$:
    \[
        (\mu_i) \in \arg \max_{(\mu_i): \forall i\, \mu_i \in \SE(\alpha_i)} \conc(*_{i=1}^n \mu_i).
    \]
    We denote the corresponding maximum concentration function value by $\tse(\alpha_1, \dots, \alpha_n)$. 
\end{definition}

We will use the definitions above and the next one interchangeably both for sequences $(\mu_i)$ of measures and for sequences $(X_i)$
of corresponding independent random variables.

\begin{definition}\label{def.balanced}
    Let $(\alpha_1, \dots, \alpha_n) \in (0,1]^n$. A standard extremal sequence $(X_i)$ for $(\alpha_i)$ will be
    called 
    \emph{balanced}
    if each index $i \in \{1,\dots,n\}$ 
    can be mapped 
    bijectively to
    $j_i \in \{1, \dots, n\}$ 
    so that $X_{j_i} \sim - X_i$.


    We call $(\alpha_1, \dots, \alpha_n) \in (0,1]^n$ \emph{balanced} if
    a balanced standard extremal sequence exists for $(\alpha_i)$, equivalently, if for each $\alpha \in (0,1]$ either $\alpha^{-1}$ is an odd integer or $|\{i: \alpha_i=\alpha\}|$ is even. We call $(\alpha_1, \dots, \alpha_n)$ \emph{strongly balanced} if $|\{i: \alpha_i=\alpha\}|$ is even for each $\alpha \in (0,1]$.

    $\pr(\sum X_i = 0)$ equals $\conc(\sum X_i)$ and is the same for each standard extremal balanced sequence $(X_i)$ for $(\alpha_1, \dots, \alpha_n)$.
    We denote this probability by $\tbal(\alpha_1, \dots, \alpha_n)$.
\end{definition}

Finally, the ultimate quantity of our interest will be the maximum concentration
a sum of random variables with given concentration bounds can achieve:
\begin{definition} Given $\alpha = (\alpha_1, \dots, \alpha_n) \in (0,1]^n$ define
    \[
        \topt(\alpha_1, \dots, \alpha_n) := \sup_{(\mu_i): \forall i \, \conc(\mu_i) \le \alpha_i} \conc(*_{i=1}^n \mu_i).  
    \]
\end{definition}

A probability measure $\mu$ on $\mathbb{Z}$ is called \emph{strongly unimodal} if its convolution with any unimodal probability measure
on $\mathbb{Z}$ is unimodal. It is called log-concave if for all $i \in \mathbb{Z}$
\[
    \mu(\{i\})^2 \ge \mu(\{i-1\}) \mu(\{i+1\}).
\]
Keilson and Gerber~\cite{keilsongerber} proved that (a) a probability measure $\mu$ on $\mathbb{Z}$ is log-concave if and only if it is strongly unimodal (their Theorem~3) and (b) the class of strongly unimodal measures is closed under convolution (their Theorem~2). 

As our measures $\nu_\alpha$ are log-concave, they are strongly unimodal and
the convolution of each standard extremal sequence of measures is also log-concave and strongly unimodal.
This property will be of fundamental importance in our proofs.

\subsection{Useful lemmas}

In this section we present a few simple but essential lemmas.

\begin{lemma}\label{lem.variance}
    Let $Y_\alpha \sim \nu_{\alpha}$. Then
    \[
       \Var Y_\alpha 
= \frac 1 {12} \left(-3 \alpha^2 k^2 (k+1)^2 + 2\alpha k (k+1) (2k+1) \right), \quad  k = \lfloor \alpha^{-1} \rfloor.
    \]
    In particular, if $\alpha = \frac 1 k$ for a positive integer $k$ then
    \[
         \Var Y_\alpha = \frac {k^2 - 1} {12}.
    \]
    If $\alpha \ge \frac 1 2$ then
    \[
        \Var Y_\alpha = \alpha (1-\alpha).
    \]
    Furthermore, let $f$ be defined by $f(\alpha) = \Var Y_\alpha$. 
    Then $f$ is continuous and non-increasing for $\alpha \in (0,1)$.
    If $k$ is a positive integer, then in the interval $(\frac 1 {k+1}, \frac 1 k)$
    $f$ is differentiable and for $\alpha$ in this interval $f''(\alpha)<0$ and
    \[
        -\frac 1 6 k (k+1)(k +2)  \le  f'(\alpha) \le - \frac 1 6  k (k^2-1) \le 0.
    \]
    Finally, if $k \in \mathbb{Z}$, $k \ge 2$ then
    \[
        \partial_+ f(\frac 1 {k+1}) = 
    \sup_{\alpha \in (\frac 1 {k+1}, \frac 1 {k})} f'(\alpha)
=
 - \frac 1 6  k (k^2-1) =
    \inf_{\alpha \in (\frac 1 {k}, \frac 1 {k-1})} f'(\alpha)
     = \partial_- f(\frac 1 {k-1}), 
    \]
    where $\partial_+ f(\alpha)$, $\partial_- f(\alpha)$ denote the right and left derivatives of $f$ at $\alpha$ respectively.
\end{lemma}

Thus at points $\frac 1 k$, $k \in \mathbb{Z}$, $f$ is described by a convex function,
however in-between those points $f$ is concave. Also its derivative within the intervals $(\frac 1 {k+1}, \frac 1 k)$ it is decreasing, however the ranges of $f'$ in the consequtive intervals have an opposite, ``increasing'' order.
\medskip

\begin{proof}
    The second and the third equalities are just the variance of a uniform random variable on $k$ points
    an the variance of a Bernoulli random variable.
    To prove the first equality, let $p = p(\alpha) = \frac {\alpha - \frac 1 {k+1}} {\frac 1 k - \frac 1 {k+1}} = 1 - (k+1) (1-k\alpha)$.
    Then
    \[
        \nu_\alpha = p \nu_{\frac 1 k} + (1-p) \nu_{\frac 1 {k+1}}.
    \]
    By the law of total variance and the case $\alpha = \frac 1 k$ of this lemma and by a simple calculation:
    \begin{align*}
        &\Var Y_\alpha = \frac {p(k^2 - 1)} {12} +  \frac { (1-p) ((k+1)^2 - 1)} {12} + p (1-p)/4
        \\ & = \frac 1 {12} \left(-3 \alpha^2 k^2 (k+1)^2 + 2\alpha k (k+1) (2k+1) \right).
    \end{align*}
    It is easy to check from this that $\lim_{x \to 1-}f(x) = 0$ and $\lim_{x \to \frac 1 k -}f(x) = \lim_{x \to \frac 1 k +}f(x)$ for $k$ integer, $k \ge 2$.

    For $k$ a positive integer and $\alpha \in (\frac 1 {k+1}, \frac 1 k)$
    we have $f'(\alpha) = 12^{-1}(-6 \alpha k^2 (k+1)^2 + 2 k (k+1) (2k+1))$,
    so $f''(\alpha)<0$ in this interval and
    \begin{align*}
        &-\frac 1 6 k(k+1)(k+2) = -\frac 1 {12} (-6 k (k+1)^2 + 2k(k+1) (2k+1)) 
        \\ & =  \lim_{x \to \frac 1 {k} -} f'(x) \le f'(\alpha) \le  \lim_{x \to \frac 1 {k + 1}+} f'(x) 
        \\ &= \frac 1 {12}(-6 k^2 (k+1) + 2k(k+1)(2k+1)) = -\frac 1 {6} k(k^2 - 1) \le 0.
    \end{align*}
    Since for a fixed integer $k$,
    $12^{-1}(-6 x k^2 (k+1)^2 + 2 k (k+1) (2k+1))$ is differentiable for all $x$,
    $\lim_{x \to \frac 1 {k} -} f'(x) = \partial_- f(x)$ and
    $\lim_{x \to \frac 1 {k} +} f'(x) = \partial_+ f(x)$.

    
\end{proof}

It will be important for us that
for a log-concave random variable
the value of the
concentration function
is determined by its variance up to a constant factor.
The next lemma follows from Theorem~1.1 of Bobkov, Marsiglietti and Melbourne~\cite{bobkovmarsigliettimelbourne} and
an improvement of the upper bound by Aravinda~\cite{aravinda}.
\begin{lemma}[\cite{aravinda,bobkovmarsigliettimelbourne}]\label{lem.logconcvariancenew}
    Let $X$ be an integer-valued log-concave random variable
    Then
    \[
        \frac 1 {\sqrt{1 + 12 \Var X}} \le \conc(X) \le \frac 1 {\sqrt{1 + \Var X}} \le \frac 1 {\sqrt{\Var X}}.
    \]
\end{lemma}
\medskip

Our next lemma says that the distribution of a log-concave integer random variable $X$ must be
almost flat in an interval of length around $o(\sqrt \Var X)$ that includes its mode. 
It will be used multiple times in this paper. 

\begin{lemma}\label{lem.logconcmode}
    Let $X$ be a random integer variable with a log-concave distribution.
    Assume $\pr(X=x_0) = \max_x \pr(X=x)$. Let $i$ be a positive integer. 
    If, for $\gamma \in [0,1)$ we have
    $\Var X \ge \frac {2 (\gamma+1) i^3  \pr(X=x_0) } {(1-\gamma)^3}$
    then
    \begin{equation}\label{eq.logconcmode}
        \max(\pr(X=x_0 - i), \pr(X=x_0+i)) \ge \pr(X=x_0) \gamma.
    \end{equation}
    In particular, (\ref{eq.logconcmode}) always holds for
    \[
        \gamma = 1 - \left(\frac {4 \pr(X=x_0)} {\Var X}\right)^{\frac 1 3} i \ge 1 - \frac {2^{\frac 2 3} i} {\sqrt{\Var X}}.
    \]
\end{lemma}
\begin{proof}
    Write $p_0 := \pr(X=x_0)$.
    Suppose that (\ref{eq.logconcmode}) does not hold. Then
    \[
        \pr(X=x_0+i) < p_0 \gamma \quad \mbox{and} \quad \pr(X=x_0-i) < p_0 \gamma.
    \]
    Since $X$ is log-concave we get for non-zero integers $k$:
    \[
        \pr(X=x_0 + ki) < p_0 \gamma^k.
    \]
    Using unimodality:
    \[
        \E (X-x_0)^2 \mathbb{I}_{X-x_0 \ge 0} < p_0 \sum_{k=1}^{\infty} i (ki)^2 \gamma^{k-1} = \frac {p_0 i^3 (\gamma+1)} {(1-\gamma)^3}.
    \]
    and similarly for $\E (X-x_0)^2 \mathbb{I}_{X-x_0 < 0}$.
    Thus
    \[
        \E (X-x_0)^2 < \frac {2 (\gamma+1) i^3 p_0} {(1-\gamma)^3}.
    \]
    So
    \begin{align*}
        &  \Var X = \E (X- x_0 + (x_0 - \E X))^2 
        \\ & = \E (X- x_0)^2 - (x_0 - \E X)^2 
         <
        \frac {2 (\gamma+1) i^3 p_0} { (1-\gamma)^3},
    \end{align*}
    which is a contradiction.
    
    Using 
    the already proved part we have
    that (\ref{eq.logconcmode}) holds if
    \begin{align*}
        \Var X 
        \ge \frac {4 i^3 p_0} {(1-\gamma)^3}
        \ge \frac {2  (\gamma + 1) i^3 p_0} {(1-\gamma)^3},
    \end{align*}
    or
    \[
        \gamma \le 1 - \left(\frac {4 p_0} {\Var X}\right)^{\frac 1 3} i.
    \]
    Finally, the last inequality of (\ref{eq.logconcmode}) follows by using the upper bound from Lemma~\ref{lem.logconcvariancenew}.
\end{proof}

\medskip

A simple corollary of the last lemma shows that the terms of a sum of independent log-concave random variables which
don't contribute much to the variance of the sum do not affect the concentration much.

\begin{lemma}\label{lem.logconcdomination}
    Let $\epsilon > 0$. Let $X$ be a log-concave integer random variable,
    and let $Y$ be any integer random variable.
    If $\Var Y \le \epsilon \Var X$ then
    \[
        \conc(X+Y) \ge (1 - \Clogconcdomination \epsilon^{\frac 1 3}) \conc(X).
    \]
\end{lemma}
\begin{proof}
    By Chebyshev's theorem for any $a > 0$
    \[
        \pr(|Y-\E Y| \ge \frac a 2 \sqrt{\Var X}) \le \frac {4 \Var Y} {a^2 \Var X} \le \frac {4\epsilon} {a^2}.
    \]
    Let $m$ be a mode of $X$. By Lemma~\ref{lem.logconcmode} 
    either for all integers $i$ with $m \le i \le m+a \sqrt{ \Var X}$
    or for all integers $i$ with $m-a \sqrt{ \Var X} \le i \le m$ 
    we have
    \begin{align*}
        & \pr(X = i) 
        \ge \pr(X=m) \left(1 - \frac {\Clogconcvariancenewcol a \sqrt{\Var X}}{\sqrt{\Var X}}\right)
        = \pr(X=m) (1 - 
        \Clogconcvariancenewcol
        a).
    \end{align*}
    Without loss of generality we can assume that the inequality holds for all integer $i \in I_1$, $I_1 =[m, m+a \sqrt {\Var X}]$.
    A closed interval of length $l$ contains either $\lfloor l \rfloor$ or
    $\lfloor l \rfloor + 1$ integer points, it contains exactly $\lfloor l \rfloor+1$ integer points when it starts with an integer. 
    Thus the interval $I_1$ 
    contains at least as many integer points as the interval $I_2 = [\E Y - \frac a 2 \sqrt {\Var X}, \E Y + \frac a 2 \sqrt {\Var X}]$.
    
    Thus
    \begin{align*}
        & \conc(X+Y) \ge \pr(X + Y = m + \lfloor \E Y + \frac a 2 \sqrt {\Var X} \rfloor)
        \\ & \ge \pr(Y \in I_2) \min_{i \in I_1} \pr(X = i)
        \\ & \ge \pr(X=m) (1 - \frac {4\epsilon} {a^2}) (1 - \Clogconcvariancenewcol a) \\
        & \ge \pr(X=m) (1 - \frac {4\epsilon} {a^2} - \Clogconcvariancenewcol a).
    \end{align*}
    As the minimum of $f(x) = \frac b {x^2} + c x$ for $b, c > 0$ is $3 \cdot2^{-\frac 2 3} b^{\frac 1 3} c^{\frac 2 3}$,
    maximizing the right side of the last inequality over $a$
    we get that
    \begin{align*}
        & \conc(X+Y) \ge \pr(X=m) \left(1 - 3 \cdot 2^{-\frac 2 3} (4\epsilon)^{\frac 1 3} 
        (\Clogconcvariancenewcol)^{\frac 2 3}
        \right)
        \\ & \ge \pr(X=m) (1 - \Clogconcdomination \epsilon^{\frac 1 3}). 
    \end{align*}
\end{proof}

A simple technical corollary of the last lemma is:

\begin{lemma}\label{lem.few_dropped}
   Let $k$ and $K$ be positive integers and let $\delta \in (0,1)$.
   Let $\alpha \in (0,1]^n$.
   Let $(Y_1, \dots, Y_n)$ be a standard extremal sequence for $\alpha$,
   and suppose that $\alpha_i \ge K^{-1}$ for $i \in \{1, \dots, k\}$.
   Write $S_n = \sum_{i=1}^n Y_i$.
   If 
   \[
        \Var S_n \ge V_{k, K, \delta} := \Cfewdroppedsimpler k K^2
   \]
   then
   \begin{equation}\label{eq.few_dropped}
        \conc(Y_1 + \dots + Y_n) \ge (1-\delta)\conc(Y_{k+1} + \dots + Y_n)
   \end{equation}
\end{lemma}


\medskip

\begin{proof}
    We have
    $\Var (\sum_{i=1}^k Y_i) \le k  K^2$.
    Let $\epsilon = \Clogconcdominationdelta$.
    Since  $(1+\epsilon^{-1}) \le \Cfewdroppedsimpler$
    we have that $\Var (\sum_{i=1}^n Y_i) \ge V_{k, K, \delta}$
    implies $\Var (\sum_{i=1}^n Y_i) \ge (1+\epsilon^{-1}) k K^2$ and
    \[
        \Var(\sum_{i=1}^k Y_i) \le \epsilon  \Var(\sum_{i=k+1}^n Y_i).
    \]
    So (\ref{eq.few_dropped}) follows by Lemma~\ref{lem.logconcdomination}.
    \end{proof}

\medskip

We will also need some properties of the optimal probability.

\begin{lemma}\label{lem.topt_nondec}
    $\topt(\alpha_1, \dots, \alpha_n)$ is non-decreasing in each of $\alpha_i \in (0, 1]$.
\end{lemma}
\begin{proof}
    We omit the details as the proof, based on Krein-Milman theorem, see, e.g., proof of Theorem~1 in \cite{jk2021} or Theorem~1.5 of \cite{madimanmelbournexu2017}, is standard: for any sequence 
    of measures $\mu_1, \dots, \mu_n$ with $\conc(\mu_i) = \alpha_i$, $i \in \{1,\dots,n\}$ we
    can represent the $i$th measure as a mixture of extremal measures for $\alpha_i'$ (which are extreme points of 
    the convex set of measures with concentration function at most $\alpha_i'$), 
    where $\alpha_i'$ is any number such that $1 \ge \alpha_i' \ge \alpha_i$. Then taking the expectation over
    all random variables except the $i$-th one we pick the component of the mixture that maximizes the concentration function,
    and doing this for each $i$ obtain an sequence $\mu_1', \dots, \mu_n'$ of (extremal) measures with $\conc(\mu_i') = \alpha_i'$ for all $i$
    and $\conc(\mu_1' * \dots * \mu_n) \ge \conc(\mu_1 * \dots * \mu_n)$. 
\end{proof}

\medskip

The proof of monotonicity of the optimal value for 
standard extremal sequences is also simple, but with a different argument.
\begin{lemma}\label{lem.tse_nondec}
    $\tse(\alpha_1, \dots, \alpha_n)$ is non-decreasing in each of $\alpha_i \in (0, 1]$.
\end{lemma}
\begin{proof}
    Fix $\alpha_1, \dots, \alpha_{n-1} \in (0,1]$.
    it suffices to show 
    that for each positive integer $k$
    and every $\alpha', \alpha$ such that $\frac 1 {k+1} \le \alpha < \alpha' \le \frac 1 k$ we have $\tse(\alpha_1, \dots, \alpha_{n-1}, \alpha') \ge \tse(\alpha_1, \dots, \alpha_{n-1}, \alpha)$.

    Consider a standard extremal sequence of independent random variables $(Y_1, \dots, Y_{n-1}, Y)$ for $(\alpha_1, \dots, \alpha_{n-1}, \alpha)$
    such that 
    \[
        \conc(Y_1 + \dots + Y_{n-1} + Y) = \tse(\alpha_1, \dots, \alpha_{n-1}, \alpha).
    \]
    We can assume that the maximum atom of the sum is $0$.

    Write $\mu_{n-1}$ and $\mu$ for the laws
    of $Y_1 + \dots + Y_{n-1}$ and $Y$ respectively.

    Since each $Y_i$ is strongly unimodal, the convolution $\mu_{n-1}$ is also strongly unimodal \cite{keilsongerber}.
    Let $I$ be a set of $k$ largest atoms of $\mu_{n-1}$. By the unimodality
    we can assume it is an interval.
    Furthermore we can assume that a set $J$ of $k+1$ largest atoms of $\mu_{n-1}$
    is also an interval and that $|J \setminus I| = 1$.
    It follows that $\mu_{n-1} * \mu$ is maximized over the ``extended set'' of standard extremal measures $\mu$ for $\alpha$ when $\mu(\{i\}) = \alpha$ for $i \in -I$ and $\mu(\{b\}) = 1 - k \alpha$ for $b \in -(J \setminus I)$. Here the ``extended set'' of extremal measures for $\alpha$ consists of the union of $\SE(\alpha)$ and the uniform measures on $k+1$ integers, i.e. we include the case $\mu(\{b\}) = \frac 1 {k+1}$ when $\alpha = \frac 1 {k+1}$.

    Now consider replacing $\mu$ with a standard extremal measure $\mu'$ for $\alpha'$ such that $\mu'(\{i\}) = \alpha'$ for $i \in -I$ and $\mu'(\{b\}) = 1-k\alpha'$ for $b \in -(J \setminus I)$. We can express $\mu = p \nu_{-I} + (1-p) \nu_{-J}$ and $\mu' = p' \nu_{-J} + (1-p') \nu_{-J}$ where $1 \ge p' \ge p \ge 0$ and $\nu_A$ is the uniform measure on $A$.

    So
    \begin{align*}
        &(\mu_{n-1} * \mu)(\{0\})  = p (\mu_{n-1} * \nu_{-I})(\{0\}) + (1-p) (\mu_{n-1} * \nu_{-J})(\{0\})
        \\ & =  p' (\mu_{n-1} * \nu_{-I})(\{0\}) + (1-p') (\mu_{n-1} * \nu_{-J})(\{0\}) + (p - p') \Delta 
        \\ & = (\mu_{n-1} * \mu')(\{0\})  + (p-p') \Delta
    \end{align*}
where 
    \[
    \Delta = (\mu_{n-1} * \nu_{-I})(\{0\}) - (\mu_{n-1} * \nu_{-J})(\{0\}).
    \]
    Note that $\Delta$ is just difference between the average of the $k$ largest atoms of $\mu_{n-1}$ and the average of $k+1$ largest atoms of $\mu_{n-1}$, so it is non-negative. Thus
    \begin{align*}
        & \tse(\alpha_1, \dots, \alpha_{n-1}, \alpha') \ge  (\mu_{n-1} * \mu')(\{0\}) 
        \\ & \ge  (\mu_{n-1} * \mu)(\{0\}) = \tse(\alpha_1, \dots, \alpha_{n-1}, \alpha).
    \end{align*}
\end{proof}

We now proceed to give a few lemmas about the continuity of the optimal concentration function
of sums of independent random variables.
They will be necessary in the final proof of Theorem~\ref{thm.tse}.
We start with a simple lemma for the balanced case that is valid for large $\alpha$.
\begin{lemma}\label{lem.balanced_continuous}
    If $(\alpha_1, \dots, \alpha_{n-2})$ is a balanced sequence, 
    $\alpha, \alpha' \in [\frac 1 2,1]$, and $\alpha' = (1+\epsilon)\alpha$ for $\epsilon > 0$ then 
    \[
        \tbal(\alpha_1, \dots, \alpha_{n-2}, \alpha',\alpha') \le (1+8\alpha\epsilon) \tbal(\alpha_1, \dots, \alpha_{n-2}, \alpha, \alpha). 
    \]
\end{lemma}
\begin{proof}
    Let $\mu$ and $\mu'$ be the laws of the sum of balanced standard extremal sequences for $(\alpha,\alpha)$ and $(\alpha',\alpha')$ respectively.
    Then $\mu$ and $\mu'$ are symmetric and unimodal with support on $\{-1,0,1\}$, $\mu(\{0\}) = \alpha^2 + (1-\alpha)^2$ and $\mu'(\{0\}) = (\alpha')^2 +  (1-\alpha')^2$, so
    \begin{align*}
        \mu'(\{0\}) - \mu(\{0\}) 
       = 4 \alpha \epsilon + 2 \epsilon^2 - 2 \epsilon
\le 4 \alpha \epsilon.
    \end{align*}

    Let $\mu_{n-2}$ be the law of the convolution of a balanced standard extremal sequence for $(\alpha_1, \dots, \alpha_{n-2})$.

    Then using symmetry and $\mu(\{1\}) \ge \mu'(\{1\})$
    \begin{align*}
        & (\mu_{n-2} * \mu')(\{0\}) - (\mu_{n-2} * \mu)(\{0\}) 
       \\ & = (\mu'(\{0\}) - \mu(\{0\})) \mu_{n-2}(\{0\}) +2 (\mu'(\{1\})-\mu(\{1\}) \mu_{n-2}(\{1\})
         \\ & \le 4 \alpha \epsilon \mu_{n-2}(\{0\})  
         \\ & \le 8 \alpha \epsilon \mu * \mu_{n-2}(\{0\}). 
    \end{align*}

    The last line follows since $(\mu * \mu_{n-2})(\{0\}) \ge (\alpha^2 + (1-\alpha)^2) \mu_{n-2}(\{0\}) \ge \frac 1 2 \mu_{n-2}(\{0\})$. 
\end{proof}

\medskip

\medskip

The next couple of lemmas are a bit more technical.
In particular, the next lemma shows that in the strongly balanced case there is a ``continuity for mid-sized'' $\alpha_i$: if we divide the interval $[0, 1]$ into a grid with step size $K^{-2}$, then we can approximate $t(\alpha_1, \dots, \alpha_n)$ by replacing
$\alpha_i$ that are not too large and not too small with the nearest value above $\alpha_i$ on the grid.

\begin{lemma} \label{lem.midsize_alpha_continuity}
    Let $(X_i, i \in \{1, \dots, m\})$ and $(X_i', i \in \{1, \dots, m\})$ be
    balanced standard extremal sequences of random variables
    for strongly balanced sequences
    $(\alpha_i, i \in\{1,\dots,m\})$ and $(\alpha_i', i \in \{1,\dots, m\})$ 
    respectively.
    Assume furthermore, that for each $i \in \{1, \dots, m\}$ there are integers $t_i$ such that
    \[
        \frac 1 K \le \alpha_i' \le  \alpha_i   = \frac {t_i} {K^2}  \le 1 - \frac 1 K, \quad  \alpha_i' \in \left[\frac {t_i-1} {K^2}, \frac {t_i} {K^2}\right].
    \]

    Let $Y$ be a symmetric log-concave integer random variable, independent of both the sequences. 

    Then 
    \[
        \pr(X_1' + \dots + X_m' + Y = 0) \ge \pr(X_1 + \dots + X_m + Y=0)(1-  \Cmidsizealphacontinuity  K^{-\frac 1 6}) 
    \]
\end{lemma}

\begin{proof}
    We may assume $(\alpha_i)$ and $(\alpha_i')$ are non-increasing,
    $(-1)^{i-1} X_i \sim \nu_{\alpha_i}$ and $(-1)^{i-1} X_i' \sim \nu_{\alpha_i'}$.
    As they are strongly balanced, $m$ is even.
    Since the inequality is trivial for $K \le  \Cmidsizealphacontinuity^6$, we can assume
    $K > \Cmidsizealphacontinuity^6$.


    Let $\alpha,\alpha' \in [K^{-1}, 1-K^{-1}]$, $\alpha' \le \alpha = t K^{-2}$ for an integer $t$, $\alpha' \in [(t-1)K^{-2}, tK^{-2}]$. 
    Also write $k=\lfloor \alpha^{-1} \rfloor$ and $k' = \lfloor (\alpha')^{-1} \rfloor$. Since
    \[
        \min_{i, j \in \{1, \dots, K\}, i \ne j} \left|\frac 1 i - \frac 1 j\right| = \frac 1 {K (K-1)} > \frac 1 {K^2},
    \]
    there is at most one $i \in \{1, \dots, K\}$ such that $\frac 1 i \in [(t-1)K^{-2}, tK^{-2}]$ which means that $k' \in \{k, k+1\}$.

    It is easy to check that we can express 
    \begin{align*}
        \nu_{\alpha'} = \nu_{\alpha} \frac {\alpha'} {\alpha} +  \mu_{\alpha, \alpha'}  (1-\frac {\alpha'} \alpha)
    \end{align*}
    where $\mu_{\alpha, \alpha'}$ does not matter, say $\mu_{\alpha', \alpha} = \delta_k$ if $\alpha'=\alpha$, while for $\alpha' < \alpha$
    \begin{align*}
        \mu_{\alpha, \alpha'} =
        \begin{cases}
            \delta_k, & \mbox{if } k' = k \\
            x \delta_k + (1-x) \delta_{k+1}, &\mbox{if } k' = k+1,
        \end{cases}
    \end{align*}
    \[
        x = \frac {(k+1) \alpha' - \frac {\alpha'} \alpha} {1- \frac {\alpha'} \alpha}.
    \]
    When $k'=k+1$, $\alpha' \le \frac 1 {k+1}$, so $x \le 1$. As $\frac 1 {k+1} < \alpha \le \frac 1 k$, we have 
    $(k+1) \alpha' - \frac {\alpha'} \alpha = \alpha' ( (k+1) - \alpha^{-1}) > 0$, so $x > 0$. So $\mu_{\alpha, \alpha'}$ is 
    a probability measure.

    So for $i \in \{1,\dots, m\}$ we can assume a coupling
    \begin{align*}
        X_i' = \mathbb{I}_i X_i + (-1)^{i-1}(1-\mathbb{I}_i)Z_i, \quad \mathbb{I}_i \sim Bin(\frac {\alpha_i'} {\alpha_i}), \quad Z_i \sim \mu_{\alpha_i, \alpha_i'},
    \end{align*}
    where $(\mathbb{I}_i)$, $(Z_i)$ are sequences of independent random variables, independent of $(X_i)$ and $Y$.
    
    Now write $\Delta_i = (-1)^{i-1}(1-\mathbb{I}_i)Z_i$. We will group the sequences of random variables
    into consecutive pairs, so that the distribution of the sum of elements in each pair is symmetric.

    Thus we can write
    \begin{align*}
        & \sum_{i=1}^m X_i = \sum_{j=1}^{\frac m 2} P_j; \quad P_j = X_{2j-1} + X_{2j} \\
        & \sum_{i=1}^m X_i' = \sum_{j=1}^{\frac m 2} \mathbb{I}_j' P_j + (1-\mathbb{I}_j') R_j,
    \end{align*}
    where $\mathbb{I}_j' = \mathbb{I}_{2j-1} \mathbb{I}_{2j}$ and $R_j$ can be seen as a mixture (since $\alpha_{2j-1} = \alpha_{2j}$) 
    \begin{align*}
       R_j \sim \begin{cases}
           X_{2j-1} + \Delta_{2j}, \quad \mbox{with probability } \frac{\frac {\alpha_{2j-1}'} {\alpha_{2j-1}} (1 -\frac {\alpha_{2j-1}'}{\alpha_{2j-1}})} {1-(\frac{\alpha_{2j-1}'} {\alpha_{2j-1}})^2} \\
            \Delta_{2j-1} + \Delta_{2j},  \quad \mbox{with probability } \frac{(1 -\frac {\alpha_{2j-1}'}{\alpha_{2j-1}})^2} {1-(\frac{\alpha_{2j-1}'} {\alpha_{2j-1}})^2} \\
            \Delta_{2j-1} + X_{2j},  \quad \mbox{with probability } \frac{\frac {\alpha_{2j-1}'} {\alpha_{2j-1}} (1 -\frac {\alpha_{2j-1}'}{\alpha_{2j-1}})} {1-(\frac{\alpha_{2j-1}'} {\alpha_{2j-1}})^2}.
        \end{cases}
    \end{align*}
    With this coupling write $S_n = \sum_{i=1}^m X_i + Y$, $S_n' = \sum_{i=1}^m X_i' + Y$, $V_{\alpha_i} = \Var X_i$ and
    \[
        V = \Var S_n = \sum_{i=1}^m V_{\alpha_i} + \Var Y.
    \]
    Using Markov's inequality
    \[
        \pr(\sum_{j=1}^{\frac m 2} (1-\mathbb{I}_j') k_{2j-1}^2 \ge u V) \le \frac {\sum_{j=1}^{\frac m 2}\E (1-\mathbb{I}_{j}')  k_{2j-1}^2} {u V}.
    \]
    and
    \begin{align}
        &\E (1-\mathbb{I}_{j}')k_{2j-1}^2 = k_{2j-1}^2 (1-\left(\frac {\alpha_{2j-1}'} {\alpha_{2j-1}} \right)^2)  \nonumber
        \\ &\le  k_{2j-1}^2 (1-\left(\frac {\alpha_{2j-1} - K^{-2}} {\alpha_{2j-1}} \right)^2)  \le \frac {2 k_{2j-1}^2} {K^2 \alpha_{2j-1}} \le \frac {2 k_{2j-1}^2} {K}. \label{eq.cmrk1}
    \end{align}
     Meanwhile  $\Var (X_{2j-1} + X_{2j}) = 2\Var X_{2j-1}$. In the case $k_{2j-1} \ge 2$ we 
     have using Lemma~\ref{lem.variance} $\Var X_{2j-1} \ge \frac {k_{2j-1}^2 - 1} {12} \ge \frac {k_{2j-1}^2} {16}$,
    so
    \[
        \frac {\E ((1-\mathbb{I}_j') k_{2j-1}^2)} {\Var (X_{2j-1} + X_{2j})} \le \frac {16} K.
    \]
    In the case $k_{2j-1} = 1$ we have $\alpha_{2j-1} \ge \frac 1 2$ and $\Var X_{2j-1} = \alpha_{2j-1} (1-\alpha_{2j-1}) \ge \alpha_{2j-1} K^{-1}$.
    So using the penultimate inequality of (\ref{eq.cmrk1})
    \[
        \frac {\E ((1-\mathbb{I}_j') k_{2j-1}^2 )} {\Var (X_{2j-1} + X_{2j})} \le \frac {2 k_{2j-1}^2} {K^2 \alpha_{2j-1}} \cdot \frac{K} { 2 \alpha_{2j-1}}
        \le \frac 4 K \le \frac {16} K.
    \] 
    A bound that holds for the ratio of $j$th terms for each $j$, holds for the ratio of their sums. Thus for any $u > 0$ 
    \[
        \pr(\sum_{j=1}^{\frac m 2} (1-\mathbb{I}_j') k_j^2 \ge u V) \le \frac {16} {u K}.
    \]
    Let $A$ be the event that $\sum_{j=1}^{\frac m 2} (1-\mathbb{I}_j') k_j^2 < u V$ for $u = K^{-\frac 1 2}$, so that $\pr(\bar{A}) \le 16 K^{-\frac 1 2}$.
    
    Fix a sequence $I' = (I_1', \dots, I_{\frac m 2}')$ such that the event $A_{I'}$ defined by $\mathbb{I}_j' = I_j'$ for each $j$ implies the event $A$.

    Since $R_j$ and $X_{2j-1} + X_{2j}$ are symmetric and 
    \[
        -k_j \le R_j, X_{2j-1} +X_{2j} \le k_j,
    \]
    we have that
    \[
        \Var(\sum_j (1-I_j') R_j) \le \sum_{j=1}^{\frac m 2} (1-I_j')k_j^2 \le 16 K^{-\frac 1 2} \Var S_n
    \]
    and
    \begin{equation}\label{eq.var16}
        \Var \left(Y + \sum_j I_j' (X_{2j-1} + X_{2j}) \right) \ge \Var S_n (1 - 16 K^{-\frac 1 2}).
    \end{equation}
    Denote by $B$ the event that $|\sum (1-I_j') R_j| \le s \sqrt{ \Var S_n}$
    with a small positive constant $s$ to be specified later.
    By Chebyshev's bound $\pr(\bar{B}) \le  16 K^{-\frac 1 2} s^{-2}$.
    Now with $I'$ fixed $\sum_j (1-I_j') R_j$ and $\sum_j I_j' (X_{2j-1} + X_{2j})$ are both independent, and the 
    latter term is log-concave,
    with mode 0, so we get:
    \begin{align*}
        & \pr(S_n' = 0 | \mathbb{I}=I') = \pr(Y + \sum (I_j' (X_{2j-1} + X_{2j}) + (1-I_j') R_j) = 0) \\
        & \ge \pr(Y + \sum I_j' (X_{2j-1} + X_{2j}) = -\sum (1-I_j') R_j | B)  \pr (B)  \\
        & \ge \min_{|t|\le s \sqrt{ \Var S_n}}  \pr(Y + \sum I_j' (X_{2j-1} + X_{2j}) = t) (1 - 16 K^{-\frac 1 2}s^{-2}). 
    \end{align*}
    Write $S_n^{I'} := Y + \sum I_j' (X_{2j-1} + X_{2j})$. 
    Now by (\ref{eq.var16}) 
    if $K \ge 17^2$, which is true by an above assumption, then $\Var S_n^{I'}\ge 17^{-1} {\Var S_n}$,
    therefore by Lemma~\ref{lem.logconcmode}
    for $|t|\le s \sqrt{ \Var S_n}$
    \begin{align*}
        & \pr(S_n^{I'} = t) 
         \ge
        \pr(S_n^{I'}=0) (1 - \frac{\Clogconcvariancenewcol s \sqrt{\Var S_n}} {\sqrt{\Var S_n^{I'}}}) \\
        & \ge \pr(S_n=0) (1 - \Cmediumcontnew s) 
    \end{align*}
    So, combining all the previous bounds and expanding over all $I'$ where $A_{I'}$ holds 
    \begin{align*}
        & \pr(S_n' = 0) \ge \sum_{I'} \pr(S_n'=0 | \mathbb{I}'=I')  (1 - 16 K^{-\frac 1 2}s^{-2}) \pr(\mathbb{I}'=I') 
        \\ &\ge \pr(S_n=0)  (1 - 16 K^{-\frac 1 2}s^{-2}) (1 - \Cmediumcontnew \cdot s) \pr(A) \\
        & \ge \pr(S_n=0) (1 - 16 K^{-\frac 1 2}s^{-2}) (1 - \Cmediumcontnew \cdot s) (1 -  16 K^{-\frac 1 2}) \\
        & \ge \pr(S_n=0) (1 - (16 K^{-\frac 1 2}s^{-2} +  \Cmediumcontnew s +  16 K^{-\frac 1 2})) \\
        &   \ge \pr(S_n=0) (1- \Cmidsizealphacontinuity K^{-\frac 1 6})
    \end{align*}
    with $s = K^{-\frac 1 6}$.

\end{proof}

\medskip

The next lemma shows that $\tbal$ is continuous in the sense that those
 $\alpha_i$ which are of the form $\frac 1 {k_i}$ for $k_i$ a sufficiently large odd integer, can be replaced by $\frac 1 {k_i - 2}$ and the value of $\tbal$ can increase by only a small fraction.
 
\begin{lemma} \label{lem.balanced_continuity_large}
    Let $K$ be a positive integer.
    Let $(X_i', i \in \{1, \dots, n\})$ and $(X_i, i \in \{1, \dots, n\})$ be
    sequences of independent random variables
    where $X_i$ is distributed unifomly on $\{-\frac{k_i-1} 2, \dots, \frac {k_i - 1} 2\}$
    while $X_i'$ is distributed uniformly on $\{-\frac {k_i+1} 2, \dots, \frac {k_i+1} 2\}$
    and $k_i$, $k_i \ge K$, are odd integers.

    Let $Y$ be a symmetric log-concave integer random variable, independent of both the sequences. 

    Then 
    \[
         \pr(Y + \sum_{i=1}^n X_i + Y = 0)  \ge  \pr(Y + \sum_{i=1}^n X_i' = 0)  \ge  
         \pr(Y + \sum_{i=1}^n X_i = 0)  (1 -  \Cbalancedcontinuitylarge 
    K^{-\frac 1 5}).
    \]
\end{lemma}
\begin{proof}
    The first inequality follows by Lemma~\ref{lem.cor3_improved}. 

    We prove the second inequality by using a coupling.
    The second inequality is trivial for $K \le \Cbalancedcontinuitylarge^5$, so we can assume without loss of generality that $K > \Cbalancedcontinuitylarge^5$. 

    Denote $S_n = Y + \sum_{i=1}^n X_i$.
    Let $(\xi_i, i \in \{1, \dots, n\})$ be the sequence of independent Rademacher random variables, independent of $(X_i)$ and $Y$.
    Also let $(\mathbb{I}_i, i\in \{1,\dots,n\})$ be a sequence of independent Bernoulli random variables,
    independent of all other ones, with $\pr(\mathbb{I}_i = 1) = \frac 2 {k_i+2}$.

    Write $\Delta = \sum_{i=1}^n \xi_i \mathbb{I}_i \frac{k_i+1} 2$.
    Then $S_n'$ defined by
    \[
        S_n' = S_n - \sum_{i=1}^n \mathbb{I}_i X_i + \Delta
    \]
    is distributed as 
    $Y + \sum_{i=1}^n X_i'$.

    Fix $\epsilon > 0$ to be specified later.
    By Markov's  inequality
    \[
        \pr(\sum_{i} k_i^2 \mathbb{I}_i \ge \epsilon^2 \sum_{i} k_i^2) \le \frac {\E (\sum_{i} k_i^2 \mathbb{I}_i)} {\epsilon^2 \sum_{i} k_i^2} = \frac {2\sum_{i} k_i} {\epsilon^2 \sum_{i}k_i^2}  \le \frac 2 {\epsilon^2 K}.
    \]
    Let us fix $I_1, \dots, I_n \in \{0,1\}$ such that 
    \begin{equation}\label{eq.condI}
        \sum_{i} k_i^2 I_i < \epsilon^2 \sum_{i} k_i^2.
    \end{equation}

    Write $\Delta_I = \sum_{i} \xi_i I_i \frac{k_i+1} 2$ and $r = \left(\sum_{i} k_i^2\right)^{\frac 1 2}$.
    We have $(k_i + 1)^2/4 \le k_i^2$ for $i \in \{1, \dots, n\}$. Using (\ref{eq.condI}) and Chebyshev's inequality
    \[
        \pr(|\Delta_I| \ge \epsilon' r) \le \frac {\Var \Delta_I} {(\epsilon')^2 \sum_{i} k_i^2}
        = \frac {\sum_{i} {I_i (k_i+1)^2}} {4 (\epsilon')^2 \sum_{i} k_i^2}
        \le \frac {\epsilon^2} {(\epsilon')^2} \le \epsilon
    \]
    for $\epsilon' = \sqrt \epsilon$.

    Denote $S_n^I = S_n - \sum_{i} I_i X_i$.
    We have 
    \begin{align*}
        &\pr(S_n' = 0 | \mathbb{I}_1 = I_1, \dots, \mathbb{I}_n=I_n) = \pr(S_n^I =  - \Delta_I)
        \\ & \ge \pr(S_n^I =  - \Delta_I, |\Delta_I| < \epsilon' r)
        \\ & \ge \min_{|j'| \le \epsilon' r }\pr(S_n^I = - j') \pr(|\Delta_I| < \epsilon' r)
        \\ & \ge \min_{|j'| \le \epsilon' r }\pr(S_n^I =  - j') (1-\epsilon).
    \end{align*} 

    Since the all the independent terms of the sum have a log-concave distribution,  for $|j'| \le \epsilon' r$ we have using Lemma~\ref{lem.logconcmode}
    \begin{align*}
        \pr(S_n^I =  - j') & \ge \pr(S_n^I = 0) \left(1 - \frac{\Clogconcvariancenewcol \epsilon' r }{\sqrt{\Var S_n^I}} \right)
        \\ & 
        \ge \pr(S_n^I=0) \left (1 -  \frac {\Clogconcvariancenewcol \epsilon' r} { \sqrt{(1-2\epsilon^2)\Var S_n}} \right). 
    \end{align*}
    In the last inequality we used (\ref{eq.condI}), Lemma~\ref{lem.variance} and $k_i \ge K \ge 2$ for $i \in \{1,\dots,n\}$ to bound
    \[
        \Var S_n - \Var S_n^I = \sum_{i} I_i \frac{k_i^2 -1} {12} \le \sum_{i} I_i \frac{k_i^2} {12}
        \le 2 \epsilon^2 \Var S_n.
    \]
    So  for $|j'| \le \epsilon' r$ assuming $\epsilon \le \frac 1 2$
    \begin{align*}
        &\pr(S_n^I = - j') \ge \pr(S_n^I = 0) \left(1 - \frac {\Clogconcvariancenewcol \cdot \epsilon' r} { \sqrt{(1-2\epsilon^2) \Var S_n}}\right)
        \\ &\ge \pr(S_n^I = 0) \left(1 - \frac { \Clogconcvariancenewcol \epsilon' \sqrt{24 \Var S_n}} { \sqrt{(1-2\epsilon^2) \Var S_n}} 
 \right)
        \\ & \ge  \pr(S_n^I = 0) (1 - \Cbalancedcontinuitylargea \epsilon')
        \\ & \ge  \pr(S_n = 0) (1 - \Cbalancedcontinuitylargea \epsilon')
    \end{align*}

So 
\begin{align*}
    & \pr(S_n' = 0 | \mathbb{I}_1=I_1,\dots,\mathbb{I}_n=I_n) 
    \\ & \ge  \pr(S_n' = 0, \, \Delta_I < \epsilon' r | \mathbb{I}_1=I_1,\dots,\mathbb{I}_n=I_n)  
    \\ & =  \pr(S_n' = 0 | \mathbb{I}_1=I_1,\dots,\mathbb{I}_n=I_n, \Delta_I < \epsilon' r)  \pr(\Delta_I < \epsilon' r |  \mathbb{I}_1=I_1,\dots,\mathbb{I}_n=I_n)
    \\ & \ge \pr(S_n = 0) (1 - \Cbalancedcontinuitylargea  \epsilon' ) (1-\epsilon) 
    \\ &\ge \pr(S_n=0)(1-  \Cbalancedcontinuitylargea  \epsilon'  - \epsilon)
    \ge \pr(S_n= 0) (1- \Cbalancedcontinuitylargebr \sqrt{\epsilon}).
\end{align*}
The event $A = \{\sum_{i} k_i^2 \mathbb{I}_i < \epsilon^2 \sum_{i} k_i^2\}$ holds with probability at least $1- \frac 2 {\epsilon^2 K}$.
So by the already proved inequality, which holds for each realization of $(\mathbb{I}_1, \dots, \mathbb{I}_n)$ on  $A$
\begin{align*}
    & \pr(S_n' = 0) \ge \pr(S_n' = 0 | A) \pr(A) \ge \pr(S_n= 0) (1-  \Cbalancedcontinuitylargebr
    \sqrt{\epsilon})(1-\frac 2 {\epsilon^2 K}) 
    \\ &
    \ge \pr(S_n= 0) (1- \Cbalancedcontinuitylargebr 
    \sqrt{\epsilon} -\frac 2 {\epsilon^2 K}) \ge \pr(S_n = 0) (1 - \Cbalancedcontinuitylarge 
    K^{-\frac 1 5}),
\end{align*}
if we put $\epsilon = K^{-\frac 2 5}$.

\end{proof}


We will often use the following inequality, which refer to as \emph{the HLP rearrangement inequality}. 
See also a more general
Theorem~\ref{thm.madimanwangwoo} below. 

\begin{theorem}
  \label{thm.HLP}
  Let $X, Y, Z_1, \dots, Z_n$ be integer random variables.
  Suppose the symmetric
  decreasing arrangement $Z_i^*$ exists
  for each $i \in \{1, \dots, n\}$. Then
  \[
     \conc(X + Y + Z_1 + \dots + Z_n) \le \pr(X^+ +  {}^+Y + Z_1^* + \dots + Z_n^*=0).
  \]
\end{theorem}
\begin{proof}
When the random variables have bounded support, this is a restatement of Theorem~374 of \cite{HLP}.
Otherwise, the same result follows by the bounded support case and conditioning.
\end{proof}
\medskip

We end this section with a simple result, similar to the ``balancing lemma'' of~\cite{jk2021} which
shows that Theorem~\ref{thm.tse} holds exactly in the strongly balanced case.

\begin{lemma}\label{lem.strongly_balanced}
  Let $1 \ge \alpha_1 \ge \dots \ge \alpha_n > 0$.
  If $(\alpha_i)$ is strongly balanced then $\topt(\alpha_1, \dots, \alpha_n) = \tbal(\alpha_1, \dots, \alpha_n)$.
\end{lemma}
\begin{proof}
    Using a standard argument (see, e.g., proof of Lemma~\ref{lem.balanced_approx_optimal} below),
    we can assume that there
    exists a sequence of independent random variables $X_1, \dots, X_n$
    such that for each $i \in \{1, \dots, n\}$ the law of $X_i$ is extremal for $\alpha_i$ 
    and $S = \sum_i X_i$ satisfies $\conc(S) = \topt(\alpha_1, \dots, \alpha_n)$.

    We can
    split
    $S$ into $S_1 = \sum_{k=1}^{\frac n 2} X_{2k-1}$ and $S_2 = \sum_{k=1}^{\frac n 2} X_{2k}$. By the simple argument from the proof of Lemma~1 from \cite{jk2021}, 
    \begin{equation}\label{eq.balancing}
        \pr(S=0) \le \max(\pr(S_1 - S_1' = 0), \pr(S_2-S_2'=0))
    \end{equation} 
    where $S_1'$ and $S_2'$ are independent copies of $S_1$ and $S_2$ respectively.
 Without loss of generality we can assume $\pr(S=0) \le \pr(S_1 - S_1' = 0)$.
    Note that $S_1-S_1'$ can be written as $\sum_{i=1}^{\frac n 2} Z_i$ where $Z_i = X_i - X_i'$ and $X_i' \sim X_i$ 
    and $X_1, X_1', \dots, X_{\frac n 2}, X_{\frac n 2}'$ are independent.

    By the rearrangement inequality, i.e., Theorem~\ref{thm.HLP}, 
    $\pr(\sum Z_i = 0) \le \pr (\sum Z_i^*=0)$ where $Z_i^*$ is distributed according to the symmetric decreasing rearrangement of the distribution of $Z_i$. An argument from the proof of Theorem~1 of \cite{jk2021} based on the HLP rearrangement inequality (Theorem~\ref{thm.HLP}) shows that for each integer $k$ $\pr(Z_i^* \in [-k, k]) \le \pr(Y_i - Y_i' \in [-k, k])$ where $Y_i, Y_i' \sim \nu_{\alpha_{2i}}$ are independent.
    Finally, as in Theorem~1 of \cite{jk2021}, we can apply Corollary~3 of \cite{jk2021} (see Lemma~\ref{lem.cor3_improved} below) to get that $\pr(\sum Z_i^* = 0) \le \pr(\sum_{i=1}^{\frac n 2} (Y_i - Y_i') = 0)$. Note that $(Y_1, -Y_1', \dots, Y_{\frac n 2}, -Y_{\frac n 2}')$
    is a standard extremal balanced sequence for $(\alpha_i)$ as required.
\end{proof}

\subsection{\texorpdfstring{$\epsilon$}{ε}-dominated concentration functions}
\label{subsec.epsilon_less_peaked}

Denote $I_j = \{-\lfloor \frac {j-1} 2 \rfloor, \dots, \lceil \frac{j-1} 2 \rceil\}$.
Note that $I_j$ is the support of the uniform distribution $\nu_{\frac 1 j}^+$, since by 
Definition~\ref{def.extremal} $\nu_{\frac 1 j}$ is the uniform distribution on $\{0, \dots, j-1\}$.

\begin{definition}\label{def.lesspeaked}
Let $\mu_1$ and $\mu_2$ be finite measures on integers.
Let $\epsilon \ge 0$.
We write
\[
\mu_1 \lesspeakedeps \mu_2
\]
    and say \emph{the concentration function
    of $\mu_1$ is $\epsilon$-dominated by the concentration function of $\mu_2$}
if for each positive integer $j$
\[
    \conc_j(\mu_1) = \mu_1^+(I_j) \le (1+\epsilon) \conc_j(\mu_2) = (1+\epsilon)\mu_2^+(I_j).
\]
When the measures are probability measures, we extend the notation $\lesspeakedeps$
to random variables distributed according to them.
For exact domination (i.e., the case $\epsilon = 0$) we use notation $\lesspeaked$ 
to mean
$\lesspeaked_{0}$. 
\end{definition}


It is worth noting that the results of \cite{HLP,lev1998,madimanwangwoo2018} (see, for example, Theorem~\ref{thm.HLP} and Theorem~\ref{thm.madimanwangwoo} below) yield not only bounds for the concentration, but 
also statements about the \emph{exact domination} of $\sum_{i=1}^n X_i$ of independent integer random variables by a sum of simpler independent variables $\sum_{i=1}^n Y_i$. This is due to the fact that the bound for $\conc_k(\sum_{i=1}^n X_i)$ follows from the bound for $\conc(U_k + \sum_{i=1}^n X_i)$ where $U_k$ is a new independent random variable distributed uniformly on a set of $k$ integers.

Another exact result that follows from~\cite{jk2021} is the following.
\begin{lemma}\label{lem.cor3_improved}
   Let $\mu_1, \dots, \mu_n$ and $\mu_1', \dots, \mu_n'$ be two sequences of probability measures on $\mathbb{Z}$.
   If each measure is symmetric and unimodal and $\mu_i' \lesspeaked \mu_i$, $i \in \{1,\dots,n\}$ then
   \[
        *_{i=1}^n \mu_i' \lesspeaked *_{i=1}^n \mu_i.
   \]
\end{lemma}
\begin{proof}
Denote  $\mu = *_{i=1}^n \mu_i$, $\mu' = *_{i=1}^n \mu_i'$, $p_i' = \mu'(\{i\})$ and $p_i = \mu(\{i\})$. 
Corollary~3 of~\cite{jk2021} says that for any non-negative integer $k$
\[
       x_k':= \mu'(\{-k, \dots, k\}) \le x_k:=*_{i=1}^k \mu_i(\{-k, \dots, k\}).
\]
It remains to show $x_{k-1}' + p_k' = \mu'(I_{2k}) \le x_{k-1} + p_k = \mu(I_{2k})$ for each positive integer $k$.
But this follows by adding the inqualities $x_k' = x_{k-1}'+2p_k' \le x_k = x_{k-1} + 2p_k$ and $x_{k-1}' \le x_{k-1}$
and dividing by 2.
\end{proof}
\medskip

However, in this paper we will need to deal with situations where one or more
involved measures are non-symmetric,
and we will use approximate domination for that.
   
We call a number $m$ a median of a finite measure $\mu$
if $\int_{-\infty}^m d \mu \ge \frac 1 2 \int_{-\infty}^{\infty} d \mu$ 
and $\int_{m}^\infty d \mu \ge \frac 1 2 \int_{-\infty}^{\infty} d \mu$.
For a random variable $X \sim \mu$ this is a standard definition of median.

Before proving our results about $\epsilon$-dominated concentration functions, we need one nice and simple property of measures $\mu^+$. 

In the next lemma imagine an integer-valued measure on integers as a set of balls,
where the balls are given linear order which agrees with the linear order on $\mathbb{Z}$.
Think of the balls (each representing a mass unit of the measure) as being placed on an interval of integers
and stacked in columns, a ball to the left or below of a given ball is smaller.
We will work with a more fine-grained definition of median: 
we want to capture the ``middle ball'', not only the ``middle atom''.

Formally, let $\mu$ be an integer-valued measure with finite support in $\mathbb{Z}$,
assume $\mu$ is not identically 0.
Let $N = \sum_{x\in \mathbb{Z}} \mu(\{x\})$ and fix a non-decreasing function $f_{\mu^+}: I \to \supp(\mu^+)$
where $I = I_{\mu}$ is an interval in $\mathbb{Z}$ containing $N$ consecutive integers
such that
\begin{align}\label{eq.f_mu+}
    \mu^+(\{j\}) = |\{i \in I: f_{\mu^+}(i) = j\}|.
\end{align}
We take the median of an \emph{interval} of integers $\{a, \dots, b\}$ to be $\frac {a+b} 2$.
As $\{i: f_{\mu^+}(i) \in I_j\}$ is an interval, denote its median by $m_j$, $m_j = m_{j, f_{\mu^+}}$. 

    Let $m = m_{\mu^+}$ denote the median of the interval $I = I_{\mu^+}$, i.e., the range of $f_{\mu^+}$. Then both $f_{\mu^+}(\lfloor m \rfloor)$ and $f_{\mu^+}(\lceil m \rceil)$ are medians of $\mu^+$.
Also $f_{\mu^+}(\lfloor m_j \rfloor)$ and $f_{\mu^+}(\lceil m_j \rceil)$ are medians
of the measure $\mu^+$ restricted to $I_j$.

\begin{lemma}\label{lem.nested_medians}
    Let $\mu$, $f_{\mu^+}$ and $\{m_j\}$ be as above.

    Then 0 is a median of $\mu^+$, $0 = f_{\mu^+}(\lfloor m \rfloor) = f_{\mu^+}(\lfloor m_1 \rfloor) = f_{\mu^+}(\lfloor m_2 \rfloor) = \dots$ and the medians have a nestedness property:
    \[
        m_3 \in [m_1, m_2], \quad m_4 \in [m_3, m_2] \quad m_5 \in [m_3, m_4] \quad m_6 \in [m_4, m_5], \dots
    \]
    equivalently $m_1 \le m_3 \le \dots \le m \le \dots \le m_4 \le m_2$.
\end{lemma}

\medskip

\begin{proof}
    Write  $\{a, \dots, b\}=f^{-1}_{\mu^+}(\{0\})$
    for some integers $a \le b$. 
    By the definition of a median of an interval, $m_1 = \frac {a+b} 2$ and $f_{\mu^+}(\lfloor m_1 \rfloor) = 0$.


    Extending an interval to the right (left) by $l$ elements increases (decreases) its median by $\frac l 2$.
    So for $k\ge2$ we have $m_k = m_{k-1} + \frac {(-1)^{k}} 2 \mu^+(\{(-1)^{k} \lfloor \frac k 2 \rfloor \})$.
    As $m_k - m_1$ is a sum of a sequence with monotone non-increasing absolute values and alternating signs,
    the claim about nestedness follows.

    For $m_2$ we have $m_2 = \frac {a + b} 2  + \frac 1 2 \mu^+(\{1\}) \le b + \frac 1 2$ since by definition $\mu^+(\{1\}) \le \mu^+(\{0\}) = b-a+1$. So $f_{\mu^+}(\lfloor m_2 \rfloor) =0$.

    By the nestedness property and monotonicity of $f$ we have $f(\lfloor m_j \rfloor) = 0$ for all $j$, and since $m_j = m$ for sufficiently large $j$ the same is true for $m$.
\end{proof}

\medskip

Now we will use Lemma~\ref{lem.nested_medians} to prove our main tool in this section. Given
a random variable $X$ and an event $A$ with $\pr(A) > 0$ let $\mathcal{L}(X|A)$
denote the conditional distribution of $X$ given $A$.

\begin{lemma}\label{lem.less_peaked_coupling}
    Let $\mu$ and $\mu'$ be probability measures with rational probabilities
    and finite supports on $\mathbb{Z}$.

    Assume that $\mu'$ is symmetric and unimodal
    and $\mu \lesspeakedeps \mu'$ for some $\epsilon \ge 0$.

    Then there is a coupling of $Z \sim \mu^+$ and $X'\sim\mu'$ and
    an event $A$ with $\pr(A) \ge \frac 1 {1+\epsilon}$ such that
    $\mathcal{L}(Z|A) = \mu^+$ and
    \begin{align*}
            0 \le X' \le Z  \mbox{ or }  Z-1 \le X' \le 0 \quad \mbox{conditioned on $A$} 
    \end{align*}
    Furthermore, if $\pr(\bar{A}) > 0$ then 
    $\mathcal{L}(Z|\bar{A}) = \mu^+$ and
    \begin{align*}
        \mbox{$Z$ and $X'$ are independent} \quad \mbox{conditioned on $\bar{A}$.}
    \end{align*}
\end{lemma}

\begin{proof}
    Since $\mu'$ and $\mu$ have rational probabilities and finite supports, we can assume
    without loss of generality that $\epsilon$ is rational.
    Thus we can express each mass of $\mu$, $(1+\epsilon)^{-1}\mu$ and $\mu'$ 
    as rational numbers $\frac k N$ where $N$ is a fixed positive even integer. 
    Thus $N \int (1+\epsilon)^{-1} d \mu = K$ for some integer $K \le N$, we can also assume $K$ is even.

    Fix a function $f = f_{N (1+\epsilon)^{-1} \mu^+}$ as in (\ref{eq.f_mu+}) from $\{-\frac K 2 + 1, \dots, \frac K 2\}$ to $\mathbb{Z}$.

    Similarly fix a non-decreasing function $f' = f_{N \mu'}$ (note here the notation has nothing to do with derivatives) from $\{-\frac {N} {2} + 1, \dots, \frac {N} {2}\}$ to $\mathbb{Z}$
    such that $N \mu' (\{j\}) = |\{i: f'(i) = j\}|$ for each $j$.

    Let $U$ be uniformly distributed on $\{-\frac N 2+1, \dots, \frac N 2\}$.
    Define $X' :=f'(U)$ and note that $f'(U) \sim \mu'$.
    Let $A$ be the event that $U \in \{-\frac K 2 + 1, \dots, \frac K 2\}$.
    Let $Z$ equal to $f(U)$ if $A$ occurs
    and let $Z$ be a (conditionally) independent random variable (given $U$) distributed as $\mu^+$ if $\bar{A}$ occurs.
    Then $Z \sim \mu^+$ both conditioned on $A$, and on $\bar{A}$ (if $\pr(\bar{A}) > 0$), and therefore also unconditionally. 
    Let us prove that the coupling $(Z, X')$ has the stated ``stochastic domination'' type property.


    We imagine the elements of the domain of $f$ and $f'$ as `balls' placed on $\mathbb{Z}$. The position of $i$ under $f$ will be just $f(i)$ and similarly for $f'$. Let $m, m_1, m_2, \dots$ be the medians as in Lemma~\ref{lem.nested_medians} defined for the function $f$.

    Let $i$, $i \in \{1, \dots, \frac K 2\}$.
    Since $f$ is non-decreasing and $m=\frac 1 2$,
    $f(i) \ge f(\lfloor m \rfloor) = 0$.
    Let the odd integer $j$ be defined by $f(i) = \frac {j-1} 2$. 
    Thus, with respect to $f$, the ball $i$ sits on the right boundary of the symmetric interval $I_j$.

    By Lemma~\ref{lem.nested_medians}, $m_j \le m = \frac 1 2$.
    Thus $f^{-1}(I_j)$ has at least as many balls $i'$ with $i' \le 0$
    than balls $i'$ with $i' \ge 1$.
    So $-i+1 \in f^{-1}(I_j)$.
    Now $(f')^{-1}(I_j)$ is symmetric around $\frac 1 2$ and 
    by the assumption of the lemma $2i \le |f^{-1}(I_j)| = N(1+\epsilon)^{-1}\mu^+(I_j) \le |(f')^{-1}(I_j)|=N\mu'(I_j)$.
    Hence $-i+1, i \in (f')^{-1}(I_j)$ and 
    \[
        0 \le f'(i) \le \frac {j-1} 2 = f(i).
    \]

    Let us now consider a ball $-i+1$ with $i \in \{1, \dots, \frac {K} 2\}$.
    Its position under $f$ is $-\frac {j-1} 2$ for some positive odd $j$.

    Again by Lemma~\ref{lem.nested_medians}, the median $m_j$ of $f^{-1}(I_{j})$
    is at most $\frac 1 2$ while the median $m_{j+1}$ of $f^{-1}(I_{j+1})$ is at least $\frac 1 2$. Hence $-i+1,i \in f^{-1}(I_{j+1}) \subseteq f^{-1}(I_{j+2})$.
    By a similar argument as for a ball with a positive index 
    it follows that $-i+1, i \in (f')^{-1}(I_{j+2})$. So
    \[
        f(-i+1)-1 = -\frac {j - 1} 2 - 1 \le f'(-i) \le 0.
    \]
\end{proof}



We now need two lemmas on using the relation $\lesspeakedeps$ for individual terms
to bound the concentration function of their sum. Some assumptions in them
such as log-concavity and
(\ref{eq.peakednessl1var}) could potentially be relaxed. 

\begin{lemma}\label{lem.peakednessl1}
    Let $m \in \{0,1,\dots\}$ and let $X \sim \mu_X$, $Y_1, \dots, Y_m$, $Y_i \sim \mu_{Y_i}$ and $Z \sim \mu_Z$
    be independent random variables on integers, such that a symmetric decreasing rearrangement
    $Y_i^*$ exists for $1 \le i \le m$.


    Suppose further $Z$ and $Y_1^*, \dots, Y_m^*$ are 
    log-concave and $Z$ is symmetric.
%
%
    If $\mu_X \lesspeakedeps \mu_Z$ for some $\epsilon \in (0,1)$ and
    \begin{equation}\label{eq.peakednessl1var}
        \min(\Var Z,  \Var \sum Y_i^*)   \ge \Cpeakednessloneb.
    \end{equation}
    Then
    \[
        X+Y_1 + \dots + Y_m \lesspeaked_{4\epsilon} Z+Y_1^* + \dots + Y_m^*.
    \]
\end{lemma}

\begin{proof}
    Using standard arguments we may assume, without loss of generality,
    that $X, Z, Y_1, \dots, Y_m$ have rational probabilities and finite supports.

    Let $k$ be any positive integer. Let $U_k^+$ be a uniform random variable on $I_k$, independent from all other
    random variables. By 
    Theorem~\ref{thm.HLP}
    \begin{align*}
        \conc_k(X + Y_1 + \dots + Y_m) &\le k \conc({}^+X + Y_1^* + \dots + Y_m^* + U_k^+) \\
        & = k \pr({}^+X + Y_1^* + \dots + Y_m^* + U_k^+=0).
    \end{align*}
    Write $Y = Y_1^* + \dots + Y_m^*$.
    Using \cite{keilsongerber}, $Y$, $Z$ and $Y+U_k^+$ are all (strongly) unimodal and log-concave.
    Furthermore, 0 is a mode for each of them.

    Consider a coupling between $X^+$ and $Z$ as in Lemma~\ref{lem.less_peaked_coupling}, and
    let $A$ be the event defined in that lemma.

    Write $p_j = \pr(Y + U_k= -j)$.
    Let $j_1, j_2, \dots$ be the increasing sequence of all integer $j \ge 0$ which satisfy  $p_{j+1} \le (1+\epsilon)^{-1} p_{j}$. Let $J_B = \{-j_k\}$.

    According to Lemma~\ref{lem.less_peaked_coupling}, $X^+$ satisfies on the event $A$
    \[
        X^+ -1 \le Z \le 0 \mbox{ or } 0 \le Z \le X^+.
    \]
    Since ${}^+X \sim -X^+$ we can assume $X^+$ is defined so that  ${}^+X = -X^+$, so
    \begin{equation}\label{eq.coupledxz}
        {}^+X \le -Z \le 0 \mbox{ or } 0 \le -Z \le {}^+X + 1.
    \end{equation}
    Whenever $A$ occurs and ${}^+X = j$ with $j < 0$,
    we have $j \le -Z \le 0$. Thus, for any possible realization $j'$ of $-Z$ on this event,
    we have $j \le j' \le 0$, and because of unimodality of $Y + U_k^+$
    \[
        \pr(Y + U_k^+=-j') \ge \pr(Y+U_k^+ = -j).
    \]
    Similarly if $A$ occurs then for any realization $(j, j')$ of ${}^+X$ and $-Z$
    if $j \ge 0$ then $0 \le j' \le j+1$ and 
    \[
        \pr(Y + U_k^+=-j') \ge \pr(Y+U_k^+ = -j-1) \ge (1+\epsilon)^{-1} \pr(Y+U_k^+ = -j) \mathbf{1}_{j \not \in J_B}.
    \]
    Expanding
    over the possible values $(j, j')$ that the coupled pair $({}^+X,Z)$ can take while satisfying (\ref{eq.coupledxz}) and applying our bounds above we get
    \begin{align*}
        & \pr({}^+X + Y + U_k^+=0 \cap A \cap \bar{B})
        \\ & = \sum_{(j,j')} \pr(^+X = j \cap -Z = j' | A) \pr(Y + U_k = -j) \mathbf{1}_{j \not \in J_B} 
        \\ & \le (1+\epsilon) \sum_{(j,j')} \pr(^+X = j \cap -Z = j'|A) \pr(Y + U_k = -j')  \mathbf{1}_{j \not \in J_B} 
        \\& \le (1+\epsilon)\pr(-Z + Y + U_k^+=0 \cap A \cap \bar{B}).
    \end{align*}
    From this we get 
    \begin{align*}
        &\pr({}^+X + Y + U_k^+=0) - \pr({}^+X + Y + U_k^+=0 \cap \bar{A}) 
        \\ &- \pr({}^+X + Y + U_k^+=0 \cap B) 
        \le (1+\epsilon)\pr(- Z + Y+ U_k^+=0) .
    \end{align*}
    What is the probability that the event 
    $B$ 
    occurs? 
    Since
    $Y+U_k$ is unimodal with a mode at $0$ we have $p_{j_k} \le p_{j_1} (1+\epsilon)^{-(k-1)} \le q (1+\epsilon)^{-(k-1)}$ where $q=\conc(Y + U_k^+)$.
    Thus
    \begin{equation}\label{eq.peakednessl1barB}
        \pr(B) \le \frac q {1-(1+\epsilon)^{-1}} = \frac {(1+\epsilon) q} \epsilon \le 2 \epsilon^{-1} q.
    \end{equation}

    So 
    \[
         \pr({}^+X + Y + U_k^+=0 \cap B)  \le \pr({}^+X + Y + U_k^+ = 0 | B) \pr(B) \le 2\epsilon^{-1}\conc({}^+X)q.
    \]
    We also have by Lemma~\ref{lem.less_peaked_coupling} that $\pr(\bar{A}) \le \frac {\epsilon} {1+\epsilon}$
    and conditioned on $\bar{A}$, ${}^+X$ retains the same distribution.
    So
    \[
        \pr({}^+X + Y + U_k^+=0 \cap \bar{A}) \le \frac \epsilon {1+\epsilon} \pr({}^+X + Y + U_k^+=0)  
    \]
    Hence
   \begin{align*}
       &\pr({}^+X + Y + U_k^+=0) (1 - \frac \epsilon {1+\epsilon}) - 2 \epsilon^{-1} \conc({}^+X)\conc(Y+U_k^+)
        \\ &\le (1+\epsilon)\pr(- Z + Y +U_k^+=0)
   \end{align*}
   or
   \begin{align*}
       &\pr({}^+X + Y + U_k^+=0) \le (1+\epsilon)^2 \pr(- Z + Y +U_k^+=0) 
       \\ & + 2(1+\epsilon) \epsilon^{-1} \conc(^+X)\conc(Y+U_k^+)
       \\& \le (1+3\epsilon) \pr(- Z + Y +U_k^+=0)  + 4 \epsilon^{-1} \conc({}^+X)\conc(Y+U_k^+).
   \end{align*}
   Using unimodality 
   and Lemma~\ref{lem.logconcmode}
   for a random variable $R$ where $R = Z$ or $R = Y+U_k^+$ there is an interval $I$ of at least
   $\Cpeakednesslone \sqrt {\Var R}$ integers containing a mode of $R$ such that for any $i \in I$
\begin{align*}
    & \pr(R=i) 
    \ge \conc(R) \left(1 - \frac {\Clogconcvariancenewcol |I|} {\sqrt{\Var R}} \right)
    \ge \frac 1 2 \conc(R) = \frac 1 2 \pr(R=0).
\end{align*}
    Aligning the intervals for each of the random variables $Z$ and $Y + U_k$ and using the condition (\ref{eq.peakednessl1var}) we see that
    \begin{align*}
        &\conc(-Z + Y + U_k^+) 
        \\ & \ge \frac 1 4 (\Cpeakednesslone)^2 \sqrt{\min(\Var(Z), \Var (Y+U_k^+))} \conc(Z) \conc(Y+U_k^+) 
        \\ & \ge  8 \epsilon^{-2} \conc(Z) \conc(Y+U_k^+).
    \end{align*}
    Thus
    \[
        4\epsilon^{-1} \conc({}^+X)\conc(Y+U_k^+) \le  4\epsilon^{-1} (1+\epsilon)\conc(Z)\conc(Y+U_k^+) \le  \epsilon \conc(-Z + Y + U_k^+)
    \]
    and
   \begin{align*}
       &\pr({}^+X + Y + U_k^+=0) \le (1 + 4\epsilon)\pr(- Z + Y + U_k^+=0).
   \end{align*}

%
%
\end{proof}

The next lemma allows to approximate the sum of two random variables by the sum of the respective
$\epsilon$-dominating random variables.

\begin{lemma}\label{lem.peakednessl2}
    Let $\epsilon \in (0, \frac 1 2)$ and define $V_0 = V_0(\epsilon) := \Cpeakednessltwo$. 
    Let $X \sim \mu_X$, $Y \sim \mu_Y$ be indepedent random variables on integers.
    Let $X' \sim \mu_X'$ and $Y' \sim \mu_Y'$ be independent symmetric log-concave integer random variables. 
    If
    \[
        \mu_X \lesspeakedeps \mu_X', \quad \mu_Y \lesspeakedeps \mu_Y', \quad \Var X' \ge V_0 \quad\mbox{and}\quad \Var Y' \ge V_0
    \]
    then
    \[
        \conc(X+Y) \le (1+20\epsilon)\conc(X'+Y').
    \]
\end{lemma}

\begin{proof}
    Using standard arguments we may assume, without loss of generality,
    that $\mu_X, \mu_Y, \mu_X', \mu_Y'$ have rational probabilities and finite supports.

    We again apply 
    the HLP rearrangement inequality,
    Theorem~\ref{thm.HLP},
    and get
    \[
        \conc(X + Y) \le \conc({}^+X + Y^+) = \pr({}^+X + Y^+=0).
    \]
    We apply Lemma~\ref{lem.less_peaked_coupling} to couple $(Y^+, Y')$
    and denote by $A_Y$ the event $A$ from that lemma.
    Note that ${}^+X$ is unimodal
    with a mode at 0 -- we will use this fact multiple times.
    Define the set 
    $\mathcal{B}_X = \{i \in \mathbb{Z}:  i \ge 0,\, \pr({}^+X=i+1) < (1+\epsilon)^{-1}\pr({}^+X=i)\}$
    and 
    the event $B_Y = \{Y^+ \in -\mathcal{B}_X\}$. 

    Let us split the probability space into a few cases. To avoid technical trivialities,
    for two events $C$ and $D$ we define $\pr(C | D) := 0$ whenever $\pr(D) = 0$ in this proof.

    \emph{Case 1: $A_Y$ holds.}
    We define the following events and use Lemma~\ref{lem.less_peaked_coupling} and the definition of $\mathcal{B}_X$ and $B_Y$ to make some observations.
            \begin{itemize}
                \item $A_> := A_Y \cap \{Y^+ > 0\} \cap \{ 0 \le Y' \le Y^+\}$. Note that on this event $Y' \le Y^+$ and so by the unimodality of $^+X$ we have $\pr({}^+X + Y^+=0 |A_>) \le \pr({}^+X + Y'=0 | A_>)$.
                \item $A_\le := A_Y \cap \{Y^+ \le Y' \le 0\}$. Again by the unimodality of ${}^+X$  $\pr({}^+X + Y^+=0 |A_\le) \le \pr({}^+X + Y'=0 | A_\le)$.
                \item $A_{\bar{B}_Y} := A_Y \cap \{Y' < 0\} \cap \{Y' = Y^+ - 1\} \cap \bar{B}_Y$. We have $\pr({}^+X + Y^+=0 |A_{\bar{B}_Y}) \le (1+\epsilon)\pr({}^+X + Y'=0 | A_{\bar{B}_Y})$.
                \item $A_{B_Y}:=A_Y \cap \{Y' < 0\} \cap \{Y' = Y^+ - 1\} \cap B_Y$. We can bound  $\pr(\{{}^+X + Y^+=0\} \cap A_{B_Y}) \le \pr({}^+X + Y^+=0 | {}^+X \in \mathcal{B}_X)\pr({}^+X \in \mathcal{B}_X) \le \conc(Y^+) \pr({}^+X \in \mathcal{B}_X)\le \epsilon^{-1} \conc(Y^+) \conc({}^+X) \le  2\epsilon^{-1} (1+\epsilon)^2 \conc(X') \conc(Y')$. Here to get the second to last inequality we used an analogous argument as in (\ref{eq.peakednessl1barB}) while in the last inequality we used the condition of the lemma.
            \end{itemize}
    Note that these events are disjoint and
    \[
     A_Y = A_> \cup A_\le \cup A_{\bar{B}_Y} \cup A_{B_Y}.
    \]

    \emph{Case 2: $\bar{A}_Y$ holds.} By  Lemma~\ref{lem.less_peaked_coupling} $\pr(\bar{A}_Y) \le 1 - \frac 1 {1+\epsilon} = \frac {\epsilon} {1+\epsilon} \le \epsilon$. Furthermore that lemma says that the event $\bar{A}_Y$ does not affect the distribution of $Y^+$. So 
     $\pr({}^+X + Y^+ = 0 | \bar{A}_Y) \pr(\bar{A}_Y) \le \epsilon \pr({}^+X + Y^+ = 0)$. 

\medskip

For Case 1 we get:
    \begin{align*}
        &\pr(\{{}^+X + Y^+ = 0\} \cap A_Y) = \pr({}^+X + Y^+ = 0 | A_>)\pr(A_>) 
        \\ & + \pr({}^+X + Y^+ = 0 | A_\le)\pr(A_\le) 
      + \pr({}^+X + Y^+ = 0 | A_{\bar{B}_Y})\pr(A_{\bar{B}_Y}) 
        \\ & + \pr(\{{}^+X + Y^+ = 0\} \cap A_{B_Y})
        \\ & \le \pr({}^+X + Y' = 0 | A_>)\pr(A_>) + \pr({}^+X + Y' = 0 | A_\le)\pr(A_\le) 
        \\  & + (1+\epsilon)\pr({}^+X + Y' = 0 | A_{\bar{B}_Y})\pr(A_{\bar{B}_Y}) 
        \\ & + \pr(\{{}^+X + Y^+ = 0\} \cap A_{B_Y})
        \\ & \le (1+\epsilon) \pr(\{{}^+X + Y' = 0\} \cap A_Y) +  2\epsilon^{-1} (1+\epsilon)^2 \conc(X') \conc(Y').
    \end{align*}

Combining the bounds in Case 1 and Case 2 
we get:
\begin{align*}
    &  \pr({}^+X + Y^+ = 0) =  \pr(\{{}^+X + Y^+ = 0\} \cap A_Y) +  \pr(\{{}^+X + Y^+ = 0\} \cap \bar{A}_Y) 
    \\ & \le    (1+\epsilon) \pr(\{{}^+X + Y' = 0\} \cap A_Y) +  2\epsilon^{-1} (1+\epsilon)^2 \conc(Y') \conc(X') 
    \\ & + \epsilon  \pr({}^+X + Y^+ = 0). 
\end{align*}
So
\begin{align*}
    &  (1-\epsilon) \pr({}^+X + Y^+ = 0) \le(1+\epsilon) \pr(\{{}^+X + Y' = 0\} \cap A_Y) 
    \\ & +  2\epsilon^{-1} (1+\epsilon)^2 \conc(Y') \conc(X') 
   \\ &  \le(1+\epsilon) \pr({}^+X + Y' = 0) +  2\epsilon^{-1} (1+\epsilon)^2 \conc(Y') \conc(X'). 
\end{align*}
and using $\epsilon \in (0, \frac 1 2)$
\begin{align*}
    &  \pr({}^+X + Y^+ = 0) \le \frac {1+\epsilon}{1-\epsilon} \pr({}^+X + Y' = 0) +  2\frac{\epsilon^{-1} (1+\epsilon)^2}{1-\epsilon} \conc(X') \conc(Y'). 
    \\ &  \le(1+4\epsilon) \pr({}^+X + Y' = 0) +  10 \epsilon^{-1} \conc(X') \conc(Y'). 
\end{align*}

Let us now bound $\pr({}^+X + Y') = 0$. We could apply Lemma~\ref{lem.peakednessl1} for this. On the other hand ${}^+X \sim -X^+$ so we can use symmetry and apply the first part of this proof to $^+X$ and get
\[
    \pr({}^+X + Y' = 0)  \le  (1+4\epsilon) \pr(X' + Y' = 0) +  10 \epsilon^{-1} \conc(X') \conc(Y'). 
\]
Therefore
\begin{align*}
    &\pr({}^+X + Y^+ = 0)  \le (1+4\epsilon)^2 \pr(X' + Y' = 0) + 40 \epsilon^{-1} \conc(X') \conc(Y') 
    \\ &\le (1+16 \epsilon) \pr(X'+Y'=0)  + 40 \epsilon^{-1} \conc(X') \conc(Y').
\end{align*}

By 
Lemma~\ref{lem.logconcmode} and symmetry
for a random variable $R$ where $R = X'$ or $R = Y'$ 
if $i$ is an integer with $|i| \le \Cpeakednesslone \sqrt {\Var R}$ then
\begin{align*}
    & \pr(R=i) 
    \ge \conc(R) \left(1 -  \frac {\Clogconcvariancenewcol |i|} {\sqrt{\Var R}} \right)
    \ge \frac 1 2 \conc(R) = \frac 1 2 \pr(R=0).
\end{align*}
So if $\Var X', \Var Y' \ge \Cpeakednessltwo$ then 
\begin{align*}
    &  \pr(X' + Y'=0) \ge \frac 1 {2^2} (\Cpeakednesslone)^2 \left(\sqrt {\min(\Var X', \Var Y')}\right)^2 \conc(X') \conc(Y')
    \\ & \ge 10 \epsilon^{-2} \conc(X') \conc(Y')
\end{align*}
and the claim follows.
\end{proof}

\subsection{Combining the proofs for the balanced case}
\label{sec.combining}

To state the next lemma we will use the notation introduced in Definition~\ref{def.balanced}.
The reader should not be scared of an additional parameter $K$: in our proof we only apply this lemma with a particular value of $K$ (which depends on $\delta$).

\begin{lemma}\label{lem.balanced_approx_optimal}
    For each $\delta \in (0,1)$ and each positive integer $K$ 
    there is a constant $V_{\delta, K}$ such that the following holds.

    Let $1 \ge \alpha_1 \ge \dots \ge \alpha_m \ge \frac 1 K > \alpha_{m+1} \ge \dots \ge \alpha_n > 0$.
    
    Assume that 
    \begin{itemize}
        \item $(\alpha_1, \dots, \alpha_m)$ is strongly balanced;
        \item if $\frac 1 K \le \alpha_i \le \frac 1 2$ then $\alpha_i = \frac {j_i} {K^2}$, where $j_i$ is an integer;
        \item if $\alpha_i < \frac 1 K$ (i.e., $i \ge m+1$) then $\alpha_i^{-1}$ is an odd integer, so $(\alpha_1, \dots, \alpha_n)$ is balanced.
    \end{itemize}
    For $\alpha \in (0,1]$, denote by $Y_\alpha$ a random variable distributed according to $\nu_\alpha$. If
    \[
        V:=\sum_i \Var Y_{\alpha_i} > V_{\delta, K}
    \]
    then
    \[
        \topt(\alpha_1, \dots, \alpha_n) \le (1+\delta) \tbal(\alpha_1, \dots, \alpha_n).
    \]
\end{lemma}

Recall that $\tbal(\alpha_1, \dots, \alpha_n) = \pr(Y_1 + \dots + Y_n = 0)$
    where $Y_1, \dots, Y_n$ is a balanced 
    standard extremal sequence of independent random variables for $\alpha_1, \dots, \alpha_n$.
The above lemma is precise up to factor $1+\delta$ because $t(\alpha_1, \dots, \alpha_n) \ge \tbal(\alpha_1, \dots, \alpha_n) = \pr(Y_1 + \dots + Y_n=0)$.

The proof of this lemma will consist of proving $\epsilon$-dominated results for two of the three regimes of $\alpha_i$
and combining these results.

\begin{lemma}\label{lem.less_peaked_large}
    For each $\epsilon \in (0,1)$ there is a constant $V_\epsilon$ such that the following holds.
    If the sequence $\alpha = (\alpha_1, \dots, \alpha_n)$ is balanced,
    $1 \ge \alpha_1  \ge \alpha_2 \ge \dots \ge \alpha_n \ge \frac 1 2$
    and $\sum_i \Var Y_{\alpha_i} > V_\epsilon$ 
    then for any sequence $X_1, \dots, X_n$
    of independent random integer variables where for each $i$ 
    $X_i$ is extremal for $\alpha_i$, we have
    \[
        X_1 + \dots + X_n \lesspeakedeps Y_1 + \dots + Y_n
    \]
    where $(Y_1, \dots, Y_n)$ is a standard extremal balanced sequence for $\alpha$.
\end{lemma}

\begin{lemma}\label{lem.less_peaked_medium}
    For each $\epsilon \in (0,1)$ and each integer $K$ 
    there is a constant $V_{\epsilon,K}$ such that the following holds.
    Suppose the sequence $\alpha = (\alpha_1, \dots, \alpha_n)$ is 
    strongly 
    balanced,
    $\frac 1 2 \ge \alpha_1  \ge \alpha_2 \ge \dots \ge \alpha_n \ge \frac 1 K$,
    and for each $i$ we have $\alpha_i = \frac {j_i} {K^2}$, where $j_i$ is an integer.
    
    If $\sum_i \Var Y_{\alpha_i} > V_{\epsilon,K}$ then for any sequence $X_1, \dots, X_n$
    of independent random integer variables where for each $i$ 
    $X_i$ is extremal for $\alpha_i$, we have
    \[
        X_1 + \dots + X_n \lesspeakedeps Y_1 + \dots + Y_n
    \]
    where $(Y_1, \dots, Y_n)$ is a standard extremal balanced sequence for $\alpha$.
\end{lemma}

Although the two lemmas above look very similar, the proof of 
Lemma~\ref{lem.less_peaked_large}
 is relatively easy using the tools
contained in this paper and is given below in this section,
while the proof of Lemma~\ref{lem.less_peaked_medium} 
takes half of our paper and requires using several major probability theorems.

\medskip

\begin{proofof}{Lemma~\ref{lem.balanced_approx_optimal}}
    Let $X_1, \dots, X_n$ be a sequence of independent integer random
    variables such that $\conc(X_i) \le \alpha_i$. We want to bound $\conc(S)$ for $S = X_1 + \dots + X_n$.
    By a well established argument, see, for example, Lemma~2 and proof of Theorem~1 in
    \cite{jk2021}, we can assume that $X_i$ is extremal for $\alpha_i$ for each $i$.

    As in the proof of Theorem~\ref{thm.tse} we will group the terms of $S$,
    but in this proof we will use slightly simpler splits. Let $A = \{i: \alpha_i \ge \frac 1 2\}=\{1, \dots, a\}$,
    let $B = \{i: \frac 1 2 > \alpha_i \ge \frac 1 K\}=\{a+1, \dots, m\}$ and let $C = \{i: \alpha_i < K^{-1}\} = \{m+1, \dots, n\}$. We write $S_A$, $S_B$ and $S_C$ for sums of the $X_i$
    over each index set $A$, $B$ and $C$, so that $S = S_A + S_B + S_C$.
    Let $(Y_1, \dots, Y_n)$ be a standard extremal balanced sequence for $\alpha$
    and denote $S_A' = \sum_{i \in A} Y_i$,  $S_B' = \sum_{i \in B} Y_i$,  $S_C' = \sum_{i \in C} Y_i$
    and $S' = S_A' + S_B' + S_C'$. 

    To prove the lemma we need to show that there exists $V_{\delta, K}$ such that if $\Var S' \ge V_{\delta, K}$ then $\conc(S) \le (1+\delta) \conc(S')$. 

    We note that in a special case when $A \cup B$ is empty, i.e., $C = \{1, \dots, n\}$, exact domination $S \lesspeaked S'$ follows by 
    the rearrangement inequality, Theorem~\ref{thm.HLP}.

    First consider the case when the set $C$ is empty, i.e., $A \cup B =\{1,\dots,n\}$. 
    In this case $m=n$ and $(\alpha_i)$ is strongly balanced,
    so by Lemma~\ref{lem.strongly_balanced},
    we have an exact inequality, i.e. with $\delta=0$ and $V_{\delta} = 0$.


    Now consider the case when $B$ is empty. 
    Let $V_{\frac \delta 4}'$ be the constant from  Lemma~\ref{lem.less_peaked_large} applied with $\epsilon = \frac \delta 4$.  
    Let $c_\delta$ be the constant from the right side of (\ref{eq.peakednessl1var})
    where we substitute $\frac \delta 4$ for $\epsilon$, i.e., $c_\delta = \Cpeakednesslonevardelta$.
    Let $c_{\delta}' := \max(c_\delta, V_{\frac \delta 4}')$.
    First assume that 
    \begin{equation}\label{eq.c_delta_K_prime}
       \min(\Var S_A', \Var S_C') < c_{\delta}'.
    \end{equation}
    Consider the case $\Var S_A' < c_{\delta}'$.
    If $\Var S_C' > (\frac\delta {2 \cdot \Clogconcdomination})^{-3} c_{\delta}'$
    ere $c_2 = \Clogconcdomination$ 
    then by
    Lemma~\ref{lem.logconcdomination} (applied with $Y=S_A'$, $X=S_C'$ and  $\epsilon = (\frac\delta {2 \cdot \Clogconcdomination })^3$) $\pr(S_A' + S_C' = 0) \ge (1- \frac \delta 2) \pr(S_C'=0)$.
    Thus $\conc(S) \le \conc(S_C) \le \conc(S_C') = \pr(S_C' = 0) \le (1-\frac \delta 2)^{-1} \pr(S_A' + S_C'=0) \le (1+\delta) \pr(S_A' + S_C' = 0)$. 
    Here the second inequality follows from 
    Theorem~\ref{thm.HLP},
    as noted above.
    Thus in this case the claim of the lemma holds for 
    $V_{\delta, K} \ge V_{\delta}^{\circ,1} := (1+(\frac{\delta} {2 \cdot \Clogconcdomination})^{-3})c_{\delta}'$. 
    Analogous reasoning shows that $V_{\delta} \ge V_{\delta}^{\circ,1}$ suffices in the case $\Var S_C' < c_{\delta}'$,
    in this case we use inequality $\conc(S_A) \le \conc(S_A')$ which follows by 
    Lemma~\ref{lem.strongly_balanced}.
    Now assume that (\ref{eq.c_delta_K_prime}) does not hold.
    Then by 
    Lemma~\ref{lem.less_peaked_large} 
    we have
    $S_A \lesspeaked_{\frac \delta 4} S_A'$.
    Applying 
    Lemma~\ref{lem.peakednessl1} with $X=S_A$
    and $Z=S_A'$ and $(Y_1, \dots, Y_m) = (X_{m+1}, \dots, X_n)$
    we get
    $S_A + S_C \lesspeaked_{\delta} S_A' + S_C'$ and in particular $\pr(S_A + S_C=0) \le (1+\delta) \pr(S_A' + S_C'=0)$.
 
    Now consider the case when $A$ is empty. In this case by 
    Lemma~\ref{lem.less_peaked_medium} there is a constant $V_{\frac \delta 4, K}''$
    such that if $\Var(S_B') \ge V_{\frac \delta 4, K}''$ then 
    $S_B \lesspeaked_{\frac \delta 4} S_B'$.
    Using the same argument as for the case when $B$ is empty we get that the claim holds
    for $V_{\delta,K} \ge V_{\delta,K}^{\circ,2} := (1+(\frac\delta {2 \cdot  \Clogconcdomination })^{-3}) \max(V_{\frac \delta 4, K}'', c_\delta)$.

    We have proved the claim in the case when at least one of the sets $A$, $B$ or $C$ is empty.
    And in this particular case we have shown that taking  $V_{\delta,K} \ge V_{\delta,K}^\circ := \max(V_{\delta}^{\circ,1}, V_{\delta,K}^{\circ,2})$ is enough.
    Our next step is to prove the claim in the case when the variance of any of the 3 sums $S_A'$, $S_B'$ or $S_C'$
    is much smaller than $V$.

    So suppose $\Var S_A' < f_\delta V$ for a small $f_\delta \in (0,1)$ defined below.
    By the already proved part we have that if $\Var(S_B' + S_C') > V_{\frac \delta 4, K}^\circ$,
    which is implied by $V > (1-f_\delta)^{-1} V_{\frac \delta 4, K}^\circ$, then
    \[
        \conc(S) \le \conc(S_B+S_C) \le (1 + \frac {\delta} 4) \pr(S_B' + S_C' = 0).
    \]
    As $\Var S_A' \le \frac {f_\delta} {1-f_\delta} \Var (S_B' + S_C')$ we have by Lemma~\ref{lem.logconcdomination}
    that
    \[
        \pr(S' = 0) \ge (1 -  \Clogconcdomination \left( \frac {f_\delta} {1-f_\delta} \right)^{\frac 1 3}) \pr(S_B' + S_C'=0).
    \]
    So if we choose $f_\delta$ such that 
    \[
        \Clogconcdomination \left( \frac {f_\delta} {1-f_\delta} \right)^{\frac 1 3} = \frac \delta 4
    \]
    (or $f_\delta = \Cbalancedapproxoptimalfdeltasolved$)
    then combining the previous two bounds, we have that if 
    $V > (1-f_\delta)^{-1} V_{\frac \delta 4, K}^\circ$ then
    \[
        \conc(S) \le (1 + \frac {\delta} 4) \pr(S_B' + S_C' = 0) \le  (1 + \frac {\delta} 4) (1-\frac \delta 4)^{-1} \pr(S'=0)
        \le (1+\delta) \pr(S'=0).
    \]

    We can treat the cases $\Var S_B' < f_\delta V$ and $\Var S_C' < f_\delta V$ similarly.

    So the case where $\Var S_A', \Var S_B', \Var S_C' \ge f_\delta V$ remains (recall that $f_\delta = \Theta(\delta^3)$).
    For this case we first apply Lemma~\ref{lem.less_peaked_large} to $S_A$ with $\epsilon = \frac {\delta} {80}$.
    It shows that there exists a constant $V_{\frac \delta {80}}'$ such that 
    if $\Var S_A' \ge f_\delta V \ge  V_{\frac \delta {80}}'$, or
    \begin{equation}\label{eq.balanced_approx_final_prea}
    V \ge f_\delta^{-1} V_{\frac \delta {80}}'
    \end{equation}
    then $S_A \lesspeaked_{\frac \delta {80}} S_A'$.
    When this holds we can apply
    Lemma~\ref{lem.peakednessl1} with $\epsilon = \frac {\delta} {80}$, $X = S_A$, $Z=S_A'$ and $Y_j = X_{m+j}$, $j \in \{1, \dots, |C|\}$.
    It shows that as long as
    \begin{equation}\label{eq.balanced_approx_final_a}
        V \ge  f_\delta^{-1} \Cbalancedapproxoptimalfinala
    \end{equation}
    we have
    \[
        \mu_{S_A + S_C} \lesspeaked_{\frac \delta {20}} \mu_{S_A' + S_C'}.
    \]
    By Lemma~\ref{lem.less_peaked_medium} there is $V_{\frac \delta {20}, K}''$ such that 
    if $\Var S_B' \ge f_\delta V \ge  V_{\frac \delta {20}, K}''$, or
    \begin{equation}\label{eq.balanced_approx_final_b}
        V \ge f_\delta^{-1} V_{\frac \delta {20}, K}''
    \end{equation}
    then $S_B \lesspeaked_{\frac \delta {20}} S_B'$.
    We can now apply Lemma~\ref{lem.peakednessl2} to $S_A+S_C$ and $S_B$ with $\epsilon = \frac \delta {20}$:
    if 
    \begin{equation}\label{eq.balanced_approx_final_c}
        V \ge f_\delta^{-1}  \Cbalancedapproxoptimalfinalb 
    \end{equation}
    then
    \[
        \conc(S_A + S_C + S_B) \le (1 + \delta) \conc(S_A'  + S_C' + S_B').
    \]
    Finally, we can define $V_{\delta, K}$ to be the maximum of
    $V_{\delta,K}^\circ$ and the right sides
    of equations (\ref{eq.balanced_approx_final_prea}), (\ref{eq.balanced_approx_final_a}), (\ref{eq.balanced_approx_final_b})
    and (\ref{eq.balanced_approx_final_c}).
\end{proofof}

\medskip

An integer random variable $X$ with support $\{x_1, \dots, x_N\}$ and
probabilities $\{p_1, \dots, p_N\}$ is called \emph{\#-log-concave}
if the ``squeezed version of $X$'', the
random variable with support $\{-\lfloor \frac {N-1} 2 \rfloor, \dots, \lceil \frac {N-1} 2 \rceil\}$ and probabilities
$\{p_1, \dots, p_N\}$ is log-concave, see \cite{madimanwangwoo2018}. We denote
the latter random variable by $X^\#$.
\begin{theorem}[Theorem~1.4 of~\cite{madimanwangwoo2018}]
\label{thm.madimanwangwoo}
    Let $X_1, \dots, X_n$ be independent integer \#-log-concave random variables.
    Then
    \[
    X_1 + \dots + X_n  \lesspeaked X_1^\# + \dots + X_n^\#.
    \]
\end{theorem} 

\begin{proofof}{Lemma~\ref{lem.less_peaked_large}}
    Let $S = \sum_{i=1}^n X_i$
    and $S' = \sum_{i=1}^n Y_i$. 
    It suffices to show that there exists $V_\epsilon$ such that 
    if $\Var S' \ge V_\epsilon$ then
    for each non-negative integer $k$ 
    \begin{align}\label{eq.less_peaked_large_claim}
        \sup_{U_k} \conc(S + U_k = 0) \le (1+\epsilon) \pr(S' + U_k' = 0),
    \end{align}
    where the supremum is taken over all uniform distributions on $k$ atoms in $\mathbb{Z}$ and $U_k'$ is the uniform distribution on the interval $I_k$,
    and $I_k$ is ``the maximally centered integer interval of length $k$''
    defined in Section~\ref{subsec.epsilon_less_peaked}; alternatively $U_k'$
    can be defined as $U_k' \sim U_k^+$.
    Note that since each $\alpha_i \ge \frac 1 2$ and $(\alpha_i)$ is balanced,
    it must be strongly balanced.


 
    By Lemma~\ref{lem.logconcdomination}, if 
    $\Var U_k' \le c_\epsilon \Var S'$ where 
    \begin{equation}\label{eq.c_epsilon}
        c_\epsilon := 
        \left(\frac {\epsilon} {2 \cdot \Clogconcdomination}\right)^3
    \end{equation}
    then
    \[
        \pr(S' + U_k' = 0) = \conc(S' + U_k') \ge (1 - \frac{\epsilon} 2) \pr(S'=0).
    \]
    Thus in this case
    \[
        \conc(S+U_k) \le \conc(S) \le \conc(S') \le (1-\frac \epsilon 2)^{-1} \pr(S' + U_k' = 0) \le (1+\epsilon)\pr(S'+U_k'=0).
    \]
    Here the second equation is obtained by the ``balancing lemma'', Lemma~\ref{lem.strongly_balanced},
    which also proves the claim with $\epsilon=0$ in the case $k=1$.

    We can apply a similar reasoning in the case where $V := \Var S' \le c_\epsilon \Var U_k'$. 
    
    Thus it remains to prove 
    that there is $V_\epsilon$ such that if $V \ge V_\epsilon$, $k \ge 2$ and
    \[
        c_\epsilon V  \le \Var U_k' \le c_\epsilon^{-1} V
    \]
    then (\ref{eq.less_peaked_large_claim}) holds.

     Recall that $\Var U_k' = \frac {k^2 - 1} {12}$. For $k\ge 2$ we have
    $\frac {k^2} {24} \le \Var U_k' \le \frac {k^2} {12}$. 
    Returning 
    from (\ref{eq.less_peaked_large_claim}) to the original definition of $\lesspeaked_\epsilon$
    and recalling 
    that $S'$ is symmetric and unimodal,
    we need to prove 
    that there exists a constant $V_\epsilon$ such that whenever $V \ge V_\epsilon$ we have
    \begin{equation}\label{eq.less_peaked_large_to_show_clt}
        \sup_{A \subset \mathbb{Z}, \, |A| = k} \pr(S \in A) \le (1+\epsilon) \pr(S' \in I_k)
    \end{equation}
    for each $k$ such that
    \begin{equation}\label{eq.less_peaked_large_k}
        \sqrt {12 c_\epsilon V}  \le k \le \sqrt{24 c_\epsilon^{-1} V}.
    \end{equation}
    For this case we will use Theorem~\ref{thm.madimanwangwoo} and a central limit theorem (Berry--Esseen).
    By the first theorem, since $X_1, \dots, X_n$ and $U_k$ are $\#$-log-concave
    \[
        \conc(X_1 + \dots + X_n + U_k) \le \conc(X_1^\#  + \dots + X_n^\# + U_k^\#).
    \]
    Therefore we can assume further that $(X_1, \dots, X_n)$ is a standard extremal sequence
    and $U_k = U_k'$.
    With this assumption either $X_i - b_i$ or $-X_i + b_i$ for some constant $b_i$ is distributed as $Y_i$, so $\Var S = \Var S' = V$.

    Each $X_i$ has support on two subsequent integers. By shifting if necessary we can assume
    a largest atom is at $0$ and the other atom is either $1$ or $-1$.

    Let
    \[
        m_3 = \sum \E|X_i- \E X_i|^3.
    \]
    The Berry--Esseen theorem states that if 
    \[
        m_3 \le \delta (\Var S)^{\frac 3 2}
    \]
    then
    \begin{equation}\label{eq.berry_esseen}
          |\pr(\frac {S - \E S} {\sqrt{\Var S}} \le z) - \Phi(z)| \le \delta. 
    \end{equation}
    As (with our assumption) $|X_i - \E X_i| \le 1$ we have $\E |X_i - \E X_i|^3 \le \Var X_i$, so 
    $m_3 \le V = \Var S$ and 
    \[
        \frac {m_3} {V^{\frac 3 2}} \le \frac 1 {\sqrt V} 
    \]
    so (\ref{eq.berry_esseen}) holds with any $\delta$ such that $\delta \ge \frac 1 {\sqrt V}$; we will assume this is the case. 

    With our assumptions $S$ is log-concave and unimodal, so on the left side of (\ref{eq.less_peaked_large_to_show_clt}) 
    it is enough to take the supremum over intervals $A$.
    In this case it is also easy to see that the maximum exists.
    Let $A_0 = \{a, \dots, b\}$, $b-a=k-1$, be an interval that achieves the maximum.

    Let $\eta$ be the standard normal measure. From (\ref{eq.berry_esseen}) and the symmetry and unimodality of $\eta$ we have
    \begin{align*}
        &  \pr(S \in [a,b]) \le 2 \delta + \Phi( \frac {b - \E S} {\sqrt{V}}) -  \Phi( \frac {a - 1 - \E S} {\sqrt{V}}). 
        \\ & \le \sup_x \eta([x, x +  \frac {b-a + 1} {\sqrt{V}}   ]) + 2 \delta
        \\ & \le \eta([-\frac {k} {2 \sqrt{V}}, \frac {k} {2 \sqrt{V}}]) + 2 \delta.
    \end{align*}
    Similarly 
    \begin{align}\label{eq.less_peaked_large_normal}
        & \pr(S' \in I_k) \ge \eta([- \frac k 2  V^{-\frac 1 2},  \frac k 2  V^{-\frac 1 2}]) - 2 \delta.
    \end{align}
    So
    \begin{align}
        &  \pr(S \in [a,b]) 
        \le \pr(S' \in I_k) + 
        4 \delta. \label{eq.less_peaked_large_normal_final}
    \end{align}

    The first inequality of (\ref{eq.less_peaked_large_k}) and $\eta([-x, x]) \ge 2x \frac 1 {\sqrt {2 \pi}} \exp(-x^2/2)$
    also imply that 
    \[
        \eta([- \frac k 2  V^{-\frac 1 2},  \frac k 2  V^{- \frac 1 2}]) \ge \eta([-\frac 1 2\sqrt {12 c_\epsilon}, \frac 1 2\sqrt {12 c_\epsilon}]) \ge \frac {\sqrt {12 c_\epsilon}} {\sqrt{2 \pi}}\exp(- 3 c_\epsilon / 2). 
    \]
   As $\epsilon \le 1$, by the definition (\ref{eq.c_epsilon}) we have
   $\exp(- 3 c_\epsilon / 2) \ge \frac 1 2$, so
   \[
       \eta([- \frac k 2 V^{-\frac 1 2},  \frac k 2 V^{- \frac 1 2}]) \ge  \frac {\sqrt {12 c_\epsilon}} {2 \sqrt{2 \pi}}.
   \]
   Assume that 
   \begin{equation}\label{eq.less_peaked_large_normal_delta}
       \delta \le \frac {\epsilon} {8} \cdot \frac {\sqrt {12 c_\epsilon}} {2 \sqrt{2 \pi}}.
   \end{equation}
   Then, since $(1 - \frac \epsilon 4)^{-1} \le 2$, (\ref{eq.less_peaked_large_normal}) implies
   \[
       \pr(S' \in I_k) \ge (1 - \frac \epsilon 4)  \eta([- \frac k 2  V^{-\frac 1 2},  \frac k 2  V^{-\frac 1 2}]) 
   \]
   and
   \[
       \delta \le  \frac {\epsilon} {8} \pr(S' \in I_k) (1 - \frac \epsilon 4)^{-1} \le \frac {\epsilon} 4  \pr(S' \in I_k).
   \]
   Putting this into (\ref{eq.less_peaked_large_normal_final}) completes the proof if all of the used assumptions can be satisfied.
  
   The assumption (\ref{eq.less_peaked_large_normal_delta}) and the assuption $\delta \ge \frac 1 {\sqrt V}$ can both
   be satisfied if 
   \[
       \frac 1 {\sqrt V} \le \frac {\epsilon} {8} \cdot \frac {\sqrt {12 c_\epsilon}} {2 \sqrt{2 \pi}}
   \]
   or
\[
    V = \Var S' \ge V_\epsilon :=
\Clesspeakedmedium \epsilon^{-5} \ge
    16^2 \epsilon^{-2} \frac {2 \pi} {12 c_\epsilon}
   = \left(\frac {\epsilon} {8} \cdot \frac {\sqrt {12 c_\epsilon}} {2 \sqrt{2 \pi}}\right)^{-2}.
\]
\end{proofof}

\subsection{Entry point: Proof of Theorem~\ref{thm.tse}}
\label{subsec.entry_point}


As mentioned above, 
we complete the proof of
Lemma~\ref{lem.balanced_approx_optimal}
by proving Lemma~\ref{lem.less_peaked_medium} in  
Part II (Section~\ref{sec.nontrivial}). 
In this section 
we show how Theorem~\ref{thm.tse} follows from the results stated above.

Given a sequence $x = (x_1, \dots, x_n)$ and a set of indices $I \subseteq \{1,\dots,n\}$, $I = \{i_1, \dots, i_m\}$, $i_1 < \dots < i_m$, we define a sequence $x_I  = (x_{i_1}, \dots, x_{i_m})$.


\medskip

\begin{proofof}{Theorem~\ref{thm.tse}}
    Using Lemma~\ref{lem.topt_nondec} we can assume that we have random integer variables $X_1, \dots, X_n$ with 
$\conc(X_1) = \alpha_1, \dots, \conc(X_n)=\alpha_n$.

Using our notation, it suffices to prove the following:
for each $\delta \in (0, \frac 1 4)$ there exists
$V_0 = V_0(\delta)$ such that
if $(Y_1, \dots, Y_n)$ is a standard extremal sequence of independent random variables 
for $(\alpha_1, \dots, \alpha_n) \in (0,1]^n$ 
and $\Var \sum_{i=1}^n Y_n \ge V_0$ then 
    \begin{equation}
    \tse(\alpha) \le \topt(\alpha) \le \tse(\alpha) (1 + \delta).
    \end{equation}


We will assume that $1 \ge \alpha_1 \ge \dots \ge \alpha_n > 0$.
Fix a large enough positive integer $K$, $K \ge K_\delta$, where $K_\delta$ will be specified later.
Define $\epsgrid = K^{-2}$ and $\epsilon_1 = K^{-1}$.

For each $n$ we will define
sequences  $\beta'$ and $\beta''$ of length $n$ such that $\beta' \le \alpha \le \beta''$ (the comparison is coordinate-wise). At the same time we will define a set $I =\{i_1, \dots, i_{n'}\} \subseteq \{1,\dots,n\}$ of all but at most a constant number of indices. Our $\beta'$ and $\beta''$ will be ``almost'' balanced:
the sequences $\alpha' = \beta_I'$ and  $\alpha'' = \beta_I''$  will be balanced and $|I| = n' = n - O_{\delta}(1)$
    will be satisfied. Furthermore both $\alpha'$ and $\alpha''$ will satisfy the conditions of Lemma~\ref{lem.balanced_approx_optimal} with the integer $K$.

Below we will prove that with our construction the following holds.
\begin{itemize}
    \item If
    $\Var S_n' \ge \Ctsealphaapprox$  
    for the sum $S_n'$ of a sequence of standard extremal random
    variables for $\alpha$ then with our choice of $I$
    \begin{align}\label{eq.tse_alpha_approx}
        \tse(\alpha) \ge \tse(\alpha_I) (1 - \frac \delta 5).
    \end{align}

\item If $K \ge \CtwobalancedKbound \delta^{-6}$ then\footnote{We used sympy for automatic calculation of constants, so some of them may be formatted in a slightly unusual way.} 
        \begin{align}\label{eq.two_balanced}
            \tbal(\alpha'')(1- \frac \delta 5) \le \tbal(\alpha').
        \end{align}
    \item For a sequence $\gamma \in (0,1]^r$ let $V_\gamma := \sum_{i=1}^r \Var Y_{\gamma_i}$ be the variance of the sum of a standard extremal sequence of random variables for $\gamma$. If $\Var S_n' = V_{\alpha} \ge 2K^5$ then
        \begin{align}\label{eq.varianceok}
            V_{\alpha''} \ge V_\alpha  (1 - \frac {10} K).
        \end{align}
\end{itemize}
    Using (\ref{eq.tse_alpha_approx}), Lemma~\ref{lem.tse_nondec} and (\ref{eq.two_balanced}) 
\begin{align*}
  &  \tbal(\alpha'')(1-\frac \delta 5)(1-\frac \delta 5) \le
    \tbal(\alpha')(1-\frac \delta 5) \le \tse(\alpha')(1-\frac \delta 5)
    \\ & \le \tse(\alpha_I)(1-\frac \delta 5) \le \tse(\alpha).
\end{align*}
    We can now apply Lemma~\ref{lem.balanced_approx_optimal} with $\frac \delta 5$ and $K_\delta$, defined as the smallest integer $K$ such that (\ref{eq.two_balanced}) is satisfied. This lemma shows that if $V_{\alpha''} \ge V_{\frac \delta 5, K_\delta}$ then $\topt(\alpha'') \le (1 + \frac \delta 5) \tbal(\alpha'')$.


Combining the last condition with the conditions of (\ref{eq.tse_alpha_approx}) and (\ref{eq.two_balanced}),
    and  using Lemma~\ref{lem.topt_nondec} and
 (\ref{eq.varianceok})
we finally get that if
    \[
        \Var S_n' \ge V_0(\delta) := \max\left(\frac {K_\delta} {K_\delta-10} V_{\frac \delta 5, K_\delta}, \Ctse\right)
    \]
    then
    \[
    \topt(\alpha) \le \topt(\beta'') \le \topt(\alpha'') \le \frac {(1+\frac \delta 5)} {(1- \frac \delta 5) (1-\frac \delta 5)} \tse(\alpha)\le \tse(\alpha)(1+\delta).
\]

 
It remains to define the sequences $\beta'$ and $\beta''$ and prove (\ref{eq.tse_alpha_approx}), (\ref{eq.two_balanced}) and (\ref{eq.varianceok}).  
For a start we shall assume that $K > 2$ so that $\epsgrid < \frac 1 4$.

The indices 
$j \in \{1, \dots, a\} =: A$ will be all those indices
where $\alpha_j \ge 1 - \epsilon_1$.
The indices $j \in \{a+1, \dots, a+b\} =: B$ will be those indices where
$1-\epsilon_1 > \alpha_j \ge \epsilon_1$. Finally,
the indices $j \in \{a+b+1, \dots, n\}=:C$ will be those where $\alpha_j < \epsilon_1$. Any of the sets $A$, $B$ or $C$ can be empty.
We write $S_n = S_A + S_B + S_C$ where $S_A = \sum_{i \in A} X_i$, and similarly for $S_B$
and $S_C$.

We first define a 
sequence $\beta''$, $\beta'' = (\beta_i'', i \in \{1, \dots, n\})$ with $1 \ge \beta_i'' \ge \alpha_i$.

\begin{itemize}
    \item For $i \in A$, if $i$ is odd then $\beta_i''=\alpha_i$, while if $i$ is even then $\beta_i'' = \alpha_{i-1}$.
    \item For $i \in B$, we will round $\alpha_i$ to the nearest level above on the $\epsgrid$-grid: $\beta_i'' := \lceil \frac {\alpha_i} {\epsgrid} \rceil \epsgrid$.
    \item For $i \in C$ we do the following. Let $\lceil\lceil x \rceil \rceil$ denote the smallest odd integer greater or equal to $x$.
        We set $\beta_i'' = (\lceil \lceil \alpha_i^{-1} \rceil \rceil - 2)^{-1}$. \label{p.ceilceil}
\end{itemize}

The 
sequence $(\beta_i', i \in \{1, \dots, n\})$ dominated by $(\alpha_i)$, i.e. with $0 < \beta_i' \le \alpha_i$ for each $i$, will be defined as follows:

\begin{itemize}
    \item For $i \in A$ if $i$ is odd and $i < a$ then $\beta_i'=\alpha_{i+1}$, otherwise, let $\beta_i' = \alpha_{i}$. 
    \item For $i \in B$, we will round $\alpha_i$ to the nearest level below on the $\epsgrid$-grid: $\beta_i' := \lfloor \frac {\alpha_i} {\epsgrid} \rfloor \epsgrid$.
    \item For $i \in C$ we set $\beta_i' = \lceil \lceil \alpha_i^{-1} \rceil \rceil^{-1}$. 
\end{itemize}

Let us now see how to convert $\beta'$ and $\beta''$ into balanced sequences
by dropping a constant number of elements.
If $|A|$ is odd, drop the element with the highest index from $A$, and if any level set in $B$, $\{i \in B: \beta_i'' = j \epsgrid\}$, $j \in \mathbb{Z}$, has an odd size, we likewise drop the element with the highest index from that level. 
Let $I = \{i_1, \dots, i_{n'}\} \subseteq \{1, \dots, n\}$, $i_1 < \dots < i_{n'}$ be the indices of non-dropped elements.
As there at most $\epsgrid^{-1}$ levels $\{j \epsgrid: j \in Z\}$ an element $\alpha_i$ with $i \in B$ can be rounded to,
the total number of dropped elements is at most $\epsgrid^{-1}+1$, so $n' \ge n - \epsgrid^{-1}-1$. 

\medskip

\begin{proofof}{(\ref{eq.tse_alpha_approx})}
    Let $Y_1, \dots, Y_n$ be a sequence of standard extremal random variables for $\alpha$ such that $\conc(\sum_{i\in I} Y_i) =  \tse(\alpha_I)$.

By our construction, each $\alpha_i$ for $i \in \{1,\dots,n\} \setminus I$ satisfies 
    $\alpha_i \ge \epsilon_1 = K^{-1}$.
Recall that $|\{1,\dots,n\} \setminus I| \le \epsgrid^{-1} + 1 = K^2+1$. 

Denote $S = \sum_{i=1}^n Y_i$ and $S_I' = \sum_{i \in I} Y_i$. 
By Lemma~\ref{lem.few_dropped} if $\Var S \ge V_{K^2 + 1, K, \frac \delta 5}$
then $\conc(S) (1-\frac \delta 5) \ge \conc(S_I')$.
As
\begin{align*}
V_{K^2+1, K, \frac \delta 5} = \Cfewdroppedsimpler (K^2+1) K^2 \le \Ctsealphaapprox, 
\end{align*}
it suffices that $\Var S$ is at least the right side of the above inequality.

\end{proofof}

\begin{proofof}{(\ref{eq.two_balanced})}
    Partition $\{1,\dots,n'\}$ into sets $A_I$, $B_I$ and $C_I$ depending
    on which of the sets $A$, $B$ and $C$ the $j$th element of $\alpha''$ (or $\alpha'$) originally belonged to:
    let $\bar{I} := \{i_1,i_2,\dots, i_{n'}\}$ and
    let 
    $A_I := \{j: i_j \in A\}$, $B_I := \{j: i_j \in B\}$ and $C_I := \{j: i_j \in C\}$.
    Let $a = |A_I|$ and $b=|A_I| + |B_I|$. 
    Note that by our construction $\alpha_{A_I \cup B_I}'$ and $\alpha_{A_I \cup B_I}''$ are strongly balanced. 

 
    We will use respectively Lemma~\ref{lem.balanced_continuous}, Lemma~\ref{lem.midsize_alpha_continuity}
    and Lemma~\ref{lem.balanced_continuity_large} for each part.
    We can apply each lemma to one of the parts $A_I$, $B_I$ or $C_I$ at a time, replacing
    each coefficent $\alpha_i'$ in that part to the coefficient $\alpha_i''$,
    while the other two parts are held fixed.
    Let us fix one particular permutation of these applications that leads from $\alpha'$ to $\alpha''$, for example
    \[
        \alpha' \xrightarrow{\mbox{Lemma~\ref{lem.balanced_continuous}}} \alpha^{(1)}  \xrightarrow{\mbox{Lemma~\ref{lem.midsize_alpha_continuity}}} \alpha^{(2)} \xrightarrow{\mbox{Lemma~\ref{lem.balanced_continuity_large}}} \alpha''.
    \]

Specifically, if $A_I \ne \emptyset$, by Lemma~\ref{lem.balanced_continuous} 
\[
    \frac {\tbal(\alpha_1'', \dots, \alpha_a'', \alpha_{a+1}', \dots, \alpha_{n'}')} 
    {\tbal(\alpha_1', \dots, \alpha_{n'}')} \le \prod_{i \in A_I} (1 + \Cbalancedcontinuous h_i)
\]
    where $h_i = \alpha_i'' - \alpha_i'$ and we assume $A_I = \{1, \dots, a\}$. 
    As $1 \ge \alpha_1' \ge \dots \ge \alpha_a' \ge 1-\epsilon_1$ and by the definition for $i \in A_I$
    \[
        \alpha_i'' - \alpha_i' = 
        \begin{cases} 
            &\alpha_i - \alpha_{i+1}, \mbox{if $i$ is odd}
            \\ & \alpha_{i-1} - \alpha_{i}, \mbox{if $i$ is even,}
        \end{cases}
    \]
    we have $\sum_{i \in A_I} h_i = (\alpha_1 - \alpha_2) + (\alpha_1 - \alpha_2) + (\alpha_3 - \alpha_4) + (\alpha_3 - \alpha_4) + \dots \le 2\sum_{i \in \{1,\dots,|A_I|-1\}} (\alpha_{i+1} - \alpha_{i}) = 2(\alpha_1 - \alpha_a) \le 2 \epsilon_1$.

    Since clearly $K > \Cbalancedcontinuousfourtimes$ or $\epsilon_1 < \frac 1 {\Cbalancedcontinuousfourtimes}$, we get using simple bounds
\begin{align*}
    \frac {\tbal(\alpha_1'', \dots, \alpha_a'', \alpha_{a+1}', \dots, \alpha_{n'}')}
    {\tbal(\alpha_1', \dots, \alpha_{n'}')} & \le \prod_{i \in A_I} (1 + \Cbalancedcontinuous h_i)
    \\ &\le
    e^{\Cbalancedcontinuous \sum_{i \in A_I} h_i} \le e^{\Cbalancedcontinuoustwice \epsilon_1} \le 1+ \Cbalancedcontinuousfourtimes \epsilon_1 \le (1- \Cbalancedcontinuousfourtimes \epsilon_1)^{-1}.
\end{align*}

If $B_I\ne \emptyset$, we can assume $B_I = \{a+1, \dots, b\}$. 
By Lemma~\ref{lem.midsize_alpha_continuity} as long as 
$K > \Cmidsizealphacontinuity^6$ 
\begin{align*}
    \frac {\tbal(\alpha_1'', \dots, \alpha_b'', \alpha_{b+1}', \dots, \alpha_{n'}')}
    {\tbal(\alpha_1'', \dots, \alpha_a'', \alpha_{a+1}', \dots, \alpha_{n'}')} 
    \le (1-\Cmidsizealphacontinuity K^{-\frac 1 6})^{-1} =  (1-\Cmidsizealphacontinuity \epsilon_1^{\frac 1 6})^{-1}.
\end{align*}

Finally, if $C_I \ne \emptyset$ we can assume $C_I = \{b+1, \dots, n'\}$. For $i \in C_I$ we have $\alpha_i'' \le \frac 1 {K-1}$ ($\le \frac 1 K$ if $K$ is odd).
Recal that sums of balanced sequences of standard extremal random variables are symmetric and log-concave, therefore they are unimodal. 
Applying Lemma~\ref{lem.balanced_continuity_large}, since clearly $K > \Cbalancedcontinuitylargeapp^5$ we get
\begin{align*}
    \frac {\tbal(\alpha_1'', \dots, \alpha_{n'}'')}
    {\tbal(\alpha_1'', \dots, \alpha_b'', \alpha_{b+1}', \dots, \alpha_{n'}')} 
    \le (1-  \Cbalancedcontinuitylarge 
    (K-1)^{-\frac 1 5})^{-1} 
    \le
    (1 -   \Cbalancedcontinuitylargeapp 
    \epsilon_1^{\frac 1 5})^{-1}.
\end{align*}

Multiplying the 3 inequalities, we obtain
\begin{align*}
    &\frac {\tbal(\alpha'')}
    {\tbal(\alpha')} \le 
    \frac 1 {(1- \Cbalancedcontinuousfourtimes \epsilon_1)   (1-\Cmidsizealphacontinuity \epsilon_1^{\frac 1 6})  (1 -   \Cbalancedcontinuitylargeapp 
    \epsilon_1^{\frac 1 5})}
\end{align*}
or
\begin{align*}
    &\frac {\tbal(\alpha')}
    {\tbal(\alpha'')} 
    \ge 1- \Cbalancedcontinuousfourtimes \epsilon_1 - \Cmidsizealphacontinuity \epsilon_1^{\frac 1 6} -   \Cbalancedcontinuitylargeapp 
    \epsilon_1^{\frac 1 5}
    \\ & \ge 1 - \Ctwobalanced 
    \epsilon_1^{\frac 1 6} 
    \\ & \ge 1 - \frac \delta 5.
\end{align*}
The last inequality holds, since $K \ge \CtwobalancedKbound \delta^{-6}$ implies that
$ \Ctwobalanced \epsilon_1^{\frac 1 6} \le \frac \delta 5$.
\end{proofof}

\medskip

\begin{proofof}{(\ref{eq.varianceok})}
    Let $k_i = \lfloor \alpha_i^{-1} \rfloor$.
    By Lemma~\ref{lem.variance}
    we have
    \[
      \Var Y_{\alpha_i} \ge  \frac {k_i^2 -1} {12}.
    \]
    By the properties
    shown in Lemma~\ref{lem.variance} we get that $f(x) = \Var Y_x$ for $Y_x \sim \nu_x$ satisfies
    $\partial_+ f(x) \ge -\frac {k_i (k_i+1) (k_i+2)} 6$ for any $x \in [\alpha_i,1)$
    and since $f$ is continuous and piecewise differentiable, $\Var Y_{\alpha_i + h} \ge \Var Y_{\alpha_i} -\frac {h k_i (k_i+1) (k_i+2)} 6$ for any $h \in [0,1-\alpha_i]$.
    For $i \in B$ we have $k_i \le K$, so if $k_i \ge 2$ we have 
    \begin{align*}
        & \Var Y_{\beta''_i} \ge \Var Y_{\alpha_i}  - \frac {\epsgrid k_i (k_i+1) (k_i+2)} {6} 
        \\ &\ge \Var Y_{\alpha_i} (1 - \frac {12 \epsgrid k_i (k_i+1) (k_i+2)} {6 (k_i^2 - 1)})
        \\& \ge \Var Y_{\alpha_i} (1 - \frac {2 k_i(k_i+2)} {K^2(k_i-1)})
        \ge \Var Y_{\alpha_i} (1 - \frac 8 K).
    \end{align*}
    Similarly if $i \in B$ and $k_i=1$ then by Lemma~\ref{lem.variance} $\partial_+ f(\alpha_i) \ge -1$,
 $\Var Y_{\alpha_i} \ge \frac 1 4 \epsilon_1 (1-\epsilon_1) \ge \frac 1 8 \epsilon_1$ and
    \begin{align*}
        & \Var Y_{\beta''_i} \ge \Var Y_{\alpha_i} - \epsgrid \ge  \Var Y_{\alpha_i} (1-  \frac {8\epsgrid} {\epsilon_1}) 
        = \Var Y_{\alpha_i} (1-  \frac 8 K). 
    \end{align*}
    If $i \in A$ then $\alpha_i \ge \frac 1 2$ and  writing $h_i = \beta_i'' - \alpha_i$ we have
    \begin{align*}
        & \Var Y_{\alpha_i} -  \Var Y_{\beta_i''} = \frac 1 4 (\alpha_i (1-\alpha_i) - (\alpha_i + h_i) (1-\alpha_i - h_i)) 
        \\ & 
        = h_i \frac {2\alpha_i - 1} 4 + \frac{h_i^2} 4 \le  \frac {h_i} 2.
    \end{align*}
    Therefore using the definition of $\beta''$
    \[
        \sum_{i \in A} \Var Y_{\alpha_i}  -   \sum_{i \in A} \Var Y_{\beta_i''} \le
        \frac 1 2 (\sum_{i\in A} h_i) \le \frac {\epsilon_1} 2 \le \epsilon_1.
    \]
    Finally, if $i \in C$ then for some odd $t_i$, $t_i \ge K$
    \begin{align*}
        & \frac {\Var Y_{\beta''_i}} {  \Var Y_{\alpha_i}} \ge \frac {(t_i-2)^2 - 1} {t_i^2 - 1} 
        = \frac {t_i - 3} {t_i + 1}
        = 1 - \frac 4 {t_i + 1} \ge 1 - \frac 4 K.
    \end{align*}
    Applying the above bounds to $V_{\beta''} = \sum_{i=1}^n \Var Y_{\beta''_i}$ we get
    \[
        V_{\beta''} \ge  {V_\alpha} (1 - \frac 8 K) - \epsilon_1.
    \]
    Similarly as in the proof of (\ref{eq.tse_alpha_approx}), since $|\{1,\dots,n\} \setminus I| \le K^2 + 1$
    and for $i \in I$ we have $\Var X_i \le K^2$, we get $V_{\beta''} - V_{\alpha''} \le  2K^4$. 
    Thus, as long as $2 K^4 \le \frac 1 K V_{\alpha}$ or $V_\alpha \ge 2 K^5$,
    \[
        V_{\alpha''} \ge V_{\beta''} - 2 K^4 \ge {V_\alpha}  (1 - \frac 8 K) - \epsilon_1 - 2 K^4
        \ge V_\alpha (1-\frac {10} K).
    \]

\end{proofof}

\end{proofof}

\bigskip

\begin{proofof}{Corollary~\ref{cor.general}}.
In the proof of his Lemma~3, Ushakov~\cite{ushakov} shows the following.

Let $\mathbb{H}$ be a separable Hilbert space (over the reals). 
Let $X_1, \dots, X_n$ be independent random elements taking values in $\mathbb{H}$.
Then there exist independent integer random variables $Z_1, \dots, Z_n$
such that $\conc(Z_1) = \conc(X_i)$ 
for $i \in \{1, \dots, n\}$ and 
$\conc(X_1 + \dots + X_n) \le \conc(Z_1 + \dots + Z_n)$.

We note that the statement of Lemma~3 in \cite{ushakov} 
includes a uniform bound
assumption $\conc(X_1)$, $\dots$, $\conc(X_n) \le \alpha \in (0,1]$,
however, the actually proved claim mentioned above allows to replace it with a weaker assumption
$\conc(X_i) \le \alpha_i \in (0,1]$ for each $i$. 


Combining this with Theorem~\ref{thm.tse} completes the proof.
\end{proofof}

\section{Part II: Proving Lemma~\ref{lem.less_peaked_medium}}
\label{sec.nontrivial}

In this part we prove Lemma~\ref{lem.less_peaked_medium} stated in Section~\ref{sec.combining} of Part I. Recall that the notation that is used to state the lemma is defined in Section~\ref{sec.definitions} and Section~\ref{subsec.epsilon_less_peaked}.

For a matrix $A$, $A \in \mathbb{R}^{m \times n}$ we will denote by $\sigma_1(A), \dots, \sigma_{\min(m,n)}(A)$
its singular values sorted in the 
non-increasing
order. If $m=n$, we will denote by $\lambda_1(A), \dots, \lambda_{n}(A)$
the eigenvalues of $A$, ordered 
non-increasingly by their magnitude.

By $a \pm b$ we will mean the interval $[a-b, a+b]$.
For $x \in \mathbb{R}$, $x$ rounded to the nearest integer will be denoted by $\round{x}$. More precisely, we define
$\round{x} = \lfloor x + \frac 1 2 \rfloor$. If $z \in \mathbb{R}^d$ we define $\round{z} = (\round{z_1}, \dots, \round{z_d})$.
We use bold capital letters like $\bf X$ to denote random vectors in $\mathbb{R}^d$ with $d \ge 2$, while single-dimensional random variables are denoted by regular letters like $X$.

In Sections~\ref{sec.precise}--\ref{sec.odlyzko_richmond} we introduce the main results that will be used, and in Section~\ref{subsec.lpm} we provide the proof of Lemma~\ref{lem.less_peaked_medium}.

\subsection {Approximation by discretized multivariate normal distribution}
\label{sec.precise}

Recall that the total variation distance $d_{TV}({\bf X}, {\bf Y})$
between two random vectors in $\mathbb{R}^d$ is defined as $\sup_A |\pr ({\bf X} \in A) - \pr({\bf Y} \in A)|$, where $A$
ranges over the Borel sets of $\mathbb{R}^d$. If ${\bf X}$ and ${\bf Y}$ are discrete with values in $S$, then
$d_{TV}({\bf X},{\bf Y}) = \sum_{x \in S} \frac 1 2 |\pr({\bf X} = x) - \pr({\bf Y} =x)|$.

In this section we state a result which
says that the distribution of a sum of bounded iid random vectors with values in $\mathbb{Z}^d$
tends in \textit{total variation} to what one would expect from the Central Limit Theorem: an appropriate
discretized multivariate normal random distribution.

Even though this type of CLT seems to be quite a fundamental property of
lattice random vectors and the topic has more than a solid body of
literature from various schools, an accessible theorem is surprisingly
difficult to locate.

Our first attempt to do this was using the book of Bhattacharya and Ranga Rao \cite{bhattacharyarao2010}, see Corollary~22.3 on p. 237 which is loosely based on a multivariate local limit theorem by Bikelis~\cite{bikelis1969} proved using 
the so-called
Edgeworth expansion method. While these results give extremely precise asymptotics, 
their statement is a bit technical and the edge case that concerns convergence in total variation with rate $o(1)$ is included only in  \cite{bhattacharyarao2010} without a full proof. 


Therefore we will use the following much more recent result proved using Stein's method, 
which we restate here in a bit more self-contained form and a slightly adapted notation.
\begin{theorem}[Theorem~4.2 of Barbour, Luczak and Xia~\cite{barbourluczakxia2018}]
    \label{thm.llt_stein}
    Let ${\bf Y}_i$, $1 \le i \le m$ be independent $\mathbb{Z}^d$-valued vectors with
    means $\mu_i$ and covariance matrices $\Sigma_i$. Let $\mu = \sum_{i=1}^m \mu_i$ and
    $\Sigma = \sum_{i=1}^m \Sigma_i$.
    Let $\gamma_i=\E \|{\bf Y}_i -\mu_i\|^3$.
    Define $u_i= \min_{1 \le j \le d} \{1-d_{TV}({\bf Y}_i, {\bf Y}_i+e^{(j)})\}$, where $e^{(j)} \in \mathbb{Z}^d$
    is the $j$th coordinate vector. 
    Let ${\bf S}_m = {\bf Y}_1 + \dots + {\bf Y}_m$, let ${\bf X} \sim N(\mu, \Sigma)$ and let $\round{\bf X}$ be
    $\bf X$ rounded to the nearest integer vector, that is, for any $x \in \mathbb{Z}^d$, $\pr(\round{\bf X} = x) = \pr({\bf X} \in x + [-\frac 1 2, \frac 1 2)^d)$.
    Then
    \[
        d_{TV}({\bf S}_m, \round{\bf X}) \le C d^{\frac 7 2} \ln m (L + (\frac d m)^{\frac 1 2}) \sqrt{\frac m {\tilde{s}_m}}
    \]
    where 
    \[
        L = \frac {m^{- \frac 1 2} \chi} {{\rm Tr}(\frac 2 m \Sigma)^{\frac 3 2}},
    \]
    \[
        \chi = \E \|{\boldsymbol \xi}\|^3, \quad \quad \tilde{s}_m = \sum_{i=1}^m u_i - \max_{i \in \{1, \dots, m\}} u_i,
    \]
    ${\boldsymbol \xi} = {\bf Y}_K' - {\bf Y}_K$ for 
    ${\bf Y}_i'$ an independent copy of ${\bf Y}_i$ for $1 \le i \le m$, $K$ is a uniformly random element of $\{1, \dots, m\}$ such that ${\bf Y}_1, \dots, {\bf Y}_m, {\bf Y}_1', \dots, {\bf Y}_m', K$ are independent and we assume that $\tilde{s}_m > 0$. 
    Furthermore, $C$ depends only on $\lambda_{min}(m^{-1}\Sigma)$,  $\lambda_{max}(m^{-1}\Sigma)$ and  $d^{-1}{\rm Tr}(m^{-1}\Sigma)$.
\end{theorem}

While it is not stated explicitly in \cite{barbourluczakxia2018}, we can see that the
quantities in this theorem are well defined only when $\tilde{s}_m > 0$ and the rank of $\Sigma$ is $d$.

\subsection{Projections of a smooth multivariate distribution}
\label{sec.projections}

Another result that we will need in the proof of Lemma~\ref{lem.less_peaked_medium} is the following.

\begin{lemma}\label{lem.proj_coeffs}
    For any integer $d$, $d \ge 2$ 
    and a positive $\gamma$
    there
    are positive reals $\delta_0 =\delta_0(d, \gamma)$,
    $R_1=R_1(d,\gamma)$
    and an integer 
    $L_0 = L_0(d, \gamma)$
    such that the following holds.

    Let ${\bf X}$ be a random vector in $\mathbb{R}^d$ and let $\round{{\bf X}}$ 
    be
    obtained by rounding each component of ${\bf X}$ to the nearest integer. 

    Let $q = \sup_{e: \|e\|=1} \conc(\langle e, {\bf X} \rangle, 1)$, that is $q$ is
    the supremum of the mass of ${\bf X}$ concentrated (strictly) between any two hyperplanes at distance one from each other.


    Let $R_0 >0$ and suppose that in the open ball $\{x: \|x\| < R_0\}$ the random vector ${\bf X}$ has
    a positive and differentiable density $f$ (that is $\pr({\bf X} \in A) = \int_{A} f(x) d x_1 \dots d x_d > 0$ for any Borel set $A \subseteq \{x: \|x\| < R_0\}$) such that
    \begin{equation}\label{eq.delta0_gradient}
        \sup_{ \|x\| < R_0 } \|\nabla \ln f(x) \| \le \delta_0
    \end{equation}
    Then for any vector $a \in \mathbb{R}^d$ either
    \begin{equation}\label{eq.proj_conc}
        \conc(\langle a, \round{{\bf X}} \rangle) = \max_{x \in \mathbb{Z}} \pr(\langle a, \round{{\bf X}} \rangle = x) \le \gamma q + \pr\left(\|{\bf X}\| \ge R_0 - R_1\right)
    \end{equation}
    or there is $\kappa \in \mathbb{R}$, $\kappa \ne 0$ such that $\kappa a \in \{-L_0, \dots, L_0\}^d$. 
\end{lemma}

\begin{proof}
%
    Fix $a \in \mathbb{R}^d$. Define $I =\{i: a_i \ne 0\}$ and $d':= |I|$. It suffices to consider the case $d' \ge 2$, as the second alternative in the conclusion is trivially
    true in the case $d'=0$ or $d'=1$.

    Assume $x_0 \in \mathbb{R}$ maximizes the left side of the inequality in (\ref{eq.proj_conc}). 
    We first use an argument\footnote{This argument can also be seen as a simple  inverse Littlewood--Offord result for fixed length sequences.} from \cite{borisov}
    to bound the density of solutions $x \in \mathbb{Z}^d$ to
    \begin{equation}\label{eq.dioph1}
        \langle a, x \rangle = x_0.
    \end{equation}
    Let $\epsilon$ be a positive small constant that depends just on $\gamma$ and $d$ to be specified later.
    Let $N_0$ be the smallest odd integer that $N_0 \ge 8 d$ and $N_0^{d'-2} \le \epsilon N_0^{d'-1}$ (equivalently, $N_0 \ge \frac 1 \epsilon$). 

    For $z\in\mathbb{Z}^d$ 
    define a hypercube $A_{z} = z + \{-\frac {N_0-1} 2, \dots, \frac {N_0-1} 2\}^{d}$.
    Suppose 
    $A_{z}$ has more than $\epsilon N_0^{d-1}$ solutions to
    (\ref{eq.dioph1}).

    For $x \in \mathbb{R}^d$, $k \in \{1,\dots,d\}$
    and $A = \{a_1, \dots, a_k\} \subseteq \{1, \dots, d\}$, $a_1 < \dots < a_k$ let $x_A$ denote the vector $(x_{a_1}, \dots, x_{a_k})$. For $A=\emptyset$ let $x_A$ be the empty vector.

    Consider the set 
    $S_I: \{x_I: x \in A_z \mbox{ and $x$ is a solution of (\ref{eq.dioph1})} \}$.
    We have
    $|S_I| > \epsilon N_0^{d'-1}$, otherwise the total number of solutions
    would be at most $\epsilon N_0^{d'-1} N_0^{d-d'} \le \epsilon N_0^{d-1}$.



    Fix distinct $i, j \in I$. Write $I_{ij} = I \setminus \{i, j\}$. 
    Clearly $|S_{I_{i,j}}| \le |N_0^{d'-2}|$.
    Combining this with $|S_I| > \epsilon N_0^{d'-1} \ge N_0^{d'-2}$ and using the pigeonhole principle we get that there are two solutions $x', x'' \in A_{z}$ such that $x'_{I_{ij}}=x''_{I_{ij}}$ and $x_I' \ne x_I''$.


    Putting them into (\ref{eq.dioph1}) and subtracting we get $a_i (x_i' - x_i'') + a_j (x_j' - x_j'') = 0$.
    Using also that $a_i, a_j \ne 0$ we get $x_i' - x_i'', x_j'-x_j'' \in \{-N_0+1, \dots, N_0-1\} \setminus \{0\}$.

    Fixing $i_0 \in I$ and applying this for each pair $(i_0, j)$, $j \in I$, $j \ne i_0$, we get that each coefficient $a_j$, $j \in I \setminus {i_0}$ satisfies
    $a_j = \frac {q_j} {p_j} a_{i_0}$, $p_j, q_j \in \{-N_0+1, \dots, N_0-1\} \setminus \{0\}$,
    thus each coefficient $a_j$ multiplied by
    $\kappa = \frac {\mathrm{lcm}(2, \dots, N_0-1)} {a_{i_0}}$ 
    is an integer with absolute value at most 
   $(N_0 -1)!$. Alternatively one could use 
   \[
       \mathrm{lcm}(2, \dots, N_0-1) \le N_0^{|\{p \in \{2, \dots, N_0\}: p \text{ prime}\}|}=N_0^{O(\frac {N_0} {\ln N_0})}
   \]
   to show that $L_0$ can be chosen of order $e^{O(N_0)}$.

    It remains to consider the case where 
    each hypercube $A_z$ has at most $\epsilon N_0^{d-1}$ solutions to (\ref{eq.dioph1}).

    Recall the rounding notation $\round{z}$ for $z \in \mathbb{R}^d$ means that $\round{z_i} = \lfloor \frac {z_i + 1} 2 \rfloor$, $i \in \{1, \dots, d\}$.
    Fix an orthonormal basis $e_1', \dots, e_{d-1}'$ of the linear subspace associated with the hyperplane (\ref{eq.dioph1}). Fix some solution $t_0 \in \mathbb{R}^d$ to (\ref{eq.dioph1}), for example $t_0 = \frac {a x_0} {\|a\|^2}$. Consider the tiling $T$ of the hyperplane with hypercubes of side length $\sqrt{d}$ and aligned to $e_1', \dots, e_{d-1}'$ with centers
$x_C=t_0 + \sqrt{d}\sum_{j=1}^{d-1} i_j e_j'$ where $i_1, \dots, i_{d-1} \in \mathbb{Z}$. Note that $x_C$ belongs to the hyperplane in $\mathbb{R}^d$ defined also by (\ref{eq.dioph1}) but without restricting the coordinates of $x$ to integers. 
    Define a new tiling $T'$ in $\mathbb{R}^d$ by taking the Cartesian product of each hypercube in $T$ with the 
    orthogonal
    segment 
    $s = \{x e_d': -\frac {\sqrt d} 2 \le x \le \frac {\sqrt d} 2\}$ 
    where
    $e_d'=\frac {a} {\|a\|}$.  We call the (continuous, closed) hypercubes of this tiling \emph{$T'$-cubes}. Let $\mathcal{T}$ be their union, i.e., $\mathcal{T} = \{x \in \mathbb{R}^d, \langle a, x \rangle \in [x_0-\frac {\sqrt d} 2, x_0+\frac {\sqrt{d}} 2]\}$. 

    We will now construct a map $\mathbb{Z}^d \cap \mathcal{T} \to T'$.  
    Let $x_C$ be the center of a $T'$-cube $C$. We will first map the point $\round{x_C} \in \mathbb{Z}^{d}$ to $C$ for each $C \in T'$. 
    This is well defined: the distance of two points that round to the same  $z \in \mathbb{Z}^d$ must be less than $\sqrt d$ while the centers of two different $T'$-cubes are at least $\sqrt{d}$ apart. Also note that $\|x-\round{x}\| \le \frac {\sqrt d} 2$ so $\round{x_C} \in C$.
    For any $C \in T'$ we will also map each interior point of $C$ with integer coordinates to $C$: there is no ambiguity as the $T'$-cubes have pairwise disjoint interiors.
    Finally we map any other point of  $\mathbb{Z}^d \cap \mathcal{T}$ falling on the exterior of one or more $T'$-cubes to one of these $T'$-cubes arbitrarily.


    We can partition the $T'$-cubes to internally disjoint hyper-rectangles, which we call  \emph{$T'$-boxes},  of side lengths 
    \[
        \sqrt{d} N_0, \dots, \sqrt{d} N_0,\sqrt{d}
    \]
    aligned to $e_1', \dots, e_{d-1}'$ and $\frac{a}{\|a\|}$ respectively.
    That is, each $T'$-box is formed from a hyperrectangle of $N_0^{d-1}$ $T'$-cubes which has only one ``layer'' in the direction $a$.

    Now consider the tiling of $\mathbb{Z}^d$ by $A_z$, $z = (N_0 i_1, \dots, N_0 i_d)^T$, $i_1, \dots, i_d \in \mathbb{Z}$. Let us call the elements (discrete hypercubes) of this tiling \emph{$A$-cubes}. 
    Assume first that at least one
    $T'$-box contained in
    the interior of the closed ball $B_{R_0 - \frac {\sqrt{d}} 2}$ exists,
    and consider a particular such $T'$-box $B$. 
    We denote the set of all points in $B \cap \mathbb{Z}^d$ mapped to a $T'$-cube in $B$ by $S_B$.

Since 
for 
each center $x$ of a $T'$-cube belonging to $B$ we have that $\round{x} \in S_B$
and $\round{x}$ is mapped to a unique $T'$-cube of $B$,
    \[
        |S_B| \ge N_0^{d-1}.
    \]
    Let us upper bound the number of (integer) solutions
    of (\ref{eq.dioph1}) that belong to $S_B$. The diameter of $B$ is $\sqrt{(d-1) (\sqrt{d} N_0)^2 + (\sqrt{d})^2} = \sqrt{(d(d-1) N_0^2 + d} < d N_0$. Thus a circumscribed ball around $B$ is contained in a hypercube in $\mathbb{R}^d$ of side length $< d N_0$ aligned to the standard coordinate axes. This hypercube touches at most $\Clpmhip$
    $A$-cubes. By our assumption thus $S_B$ contains at most
    \[
        \epsilon \Clpmhip N_0^{d-1} 
    \]
    integer solutions to (\ref{eq.dioph1}). Let $q_B = \pr(\round{{\bf X}} \in S_B)$
    and let $p_B = \pr(\{\round{{\bf X}} \in S_B\} \cap \{\mbox{$\round {{\bf X}}$ is a solution}\})$.
    Then $q_B > 0$ since $f$ is positive in $B$ and
    \begin{align*}
        &\frac {p_B} {q_B} \le \frac {\max_{x\in S_B}{\pr(\round{{\bf X}} = x)}} {\min_{x\in S_B}{\pr(\round{{\bf X}} = x)}} \times \frac {|\{x \in S_B\} \cap \{\mbox{$x$ is a solution}\}|} {|S_B|}
        \\ & \le \epsilon \Clpmhip \frac {\max_{x\in S_B}{\pr(\round{{\bf X}} = x)}}  {\min_{x\in S_B}{\pr(\round{{\bf X}} = x)}}  . 
    \end{align*}
    We will bound the last ratio using our 
  ``log-Lipschitz continuity'' 
    condition (\ref{eq.delta0_gradient}).
  By the mean value theorem for a function $g$, differentiable in an open set containing the path  $\{x+th: \, t \in [0,1]\}$ we have
  \begin{equation}\label{eq.mvt}
      |g(x + h) - g(x)| \le \|h\| \max_{z \in \{x+th: \, t \in [0,1]\}} \|\nabla g(z)\|.
  \end{equation}
  Applying this for $g(x)=\ln f(x)$ and using our condition we get $f(x'') \le f(x') \exp(\delta_0 \|x'' - x'\|)$
  for any points $x', x''$ in $B + [-\frac 1 2,  \frac 1 2]^d$.
  The diameter of the last set is at most $d N_0 + \sqrt{d} \le d (N_0 + 1)$.

  It follows that for any $z', z'' \in S_B$,
  by integrating over 
  $z' + [-\frac 1 2,  \frac 1 2]^d$ and $z'' + [-\frac 1 2,  \frac 1 2]^d$ respectively,
  and using  (\ref{eq.delta0_gradient}) and (\ref{eq.mvt})
  we have
  \[
      \frac {\pr(\round{{\bf X}} = z'')} {\pr(\round{{\bf X}} = z')} \le \exp(\delta_0 d (N_0+1))
  \]
  so
  \[
      \frac {p_B} {q_B} \le \epsilon \Clpmhip \exp(\delta_0 d (N_0+1)).
  \]
  Let us now set $\epsilon = \frac \gamma {e \Clpmhip \lceil 2 \sqrt d \rceil}$,
  $N_0 = \lceil \lceil \max(8d, \epsilon^{-1}) \rceil \rceil$ (here we used the notation $\lceil\lceil x \rceil \rceil$ for the least odd integer at least $x$, which is also used on p. \pageref{p.ceilceil}) and 
  $\delta_0 =  \frac {1} {d (N_0+1)}$.
  With these parameters we get $p_B \le \frac { \gamma q_B} {\lceil 2\sqrt {d} \rceil}$. 

  Let $\mathcal{B}$ be the set of all $T'$-boxes
  contained in the interior of $B_{R_0-\frac{\sqrt{d}}2}$. Let $\mathcal{S}$ be the set of integer solutions to (\ref{eq.dioph1}).
  A $T'$-box $B$ that contains a point of distance at least $R_0-\frac{\sqrt{d}}2$ from the origin has all points
  of distance at least $R_0 - {\rm diam}(B) - \frac{\sqrt{d}}2 \ge R_0 - N_0 d -\frac{\sqrt{d}}2$ from the origin.
  Note that for a 1-dimensional random variable $Y$ we have $\conc(Y, t) \le \lceil t \rceil \conc(Y,1)$.
  All the $T'$-boxes are contained in $\{x: |\langle a, x \rangle - x_0| \le \frac {\sqrt{d}} 2\}$
  so any point that rounds to an integer vector in a $T'$-box is at distance at most $\frac {\sqrt{d}} 2 + \frac {\sqrt{d}} 2  =\sqrt{d}$ from the hyperplane in $\mathbb{R}^d$ defined by (\ref{eq.dioph1}).
  Therefore, regardless if $\mathcal{B}$ is empty or not, we have
  \begin{align*}
      &    \pr(\round{{\bf X}} \in \mathcal{S}) \le \sum_{B \in \mathcal{B}} p_B +   \pr(\|[{\bf X}]\| \ge R_0 - N_0 d - \frac{\sqrt{d}}2) 
      \\ & \le \sum_{B: B \in \mathcal{B}} \frac {\gamma} {\lceil 2 \sqrt d \rceil} q_B  +   \pr(\|{\bf X}\| \ge R_0 - N_0 d-\frac{\sqrt{d}}2 - \frac {\sqrt{d}} 2) 
      \\ & \le \frac {\gamma} {\lceil 2 \sqrt d \rceil} \pr(|\langle a, {\bf X} \rangle - x_0| \le \sqrt{d})   +   \pr(\|{\bf X}\| \ge R_0 - N_0 d - \sqrt{d}) 
      \\ & \le \gamma q  +  \pr(\|{\bf X}\| \ge R_0 - N_0 d - \sqrt{d}). 
  \end{align*}
  We can set $R_1 = N_0 d + \sqrt{d}$.
  This completes the proof.
\end{proof}

\subsection {Applying inverse Littlewood--Offord}
\label{sec.inverse}

In this section we introduce a theorem from the \emph{inverse Littlewood--Offord theory} and use it to derive one of the key lemmas necessary for the proof of Lemma~\ref{lem.less_peaked_medium}.

We adapt the next definition from \cite{taovu2010}. Let $G$ be an additive group. A symmetric generalized arithmetic progression in $G$, or a symmetric GAP, is a quadruplet ${\bf A} = (A, M, g, r)$, where the \emph{rank} $\operatorname{rank}({\bf A}) = r$ is a non-negative integer, the \emph{dimensions} $M = (M_1, \dots, M_r)$ are an $r$-tuple of positive reals, the \emph{generators} (steps) $g = (g_1, \dots, g_r)$ are an $r$-tuple of elements of $G$, and $A \subseteq G$ is the set
\[
    A = \left\{ \sum_{i=1}^r j_i g_i : j_i \in [-M_i, \dots, M_i] \cap \mathbb{Z} \quad \forall i \in \{1, \dots, r\} \right\},
\]
Like \cite{nguyenvu2011, taovu2010} we may abuse notation and write $A$ for ${\bf A}$.
For any $t > 0$, also following \cite{taovu2010}, 
we define the \emph{dilate} $t {\bf A}$ of ${\bf A}$ to be the GAP ${\bf A}_t := (A_t, t M, g, r)$ formed by dilating all the dimensions by $t$. 
We say that $\bf{A}$ is \emph{proper} if all the elements $j_1 g_1 + \dots + j_r g_r$, $j_i \in \{-\lfloor M_i\rfloor, \dots,  \lfloor M_i\rfloor\}$ $\forall i$, are distinct. 
We define the \emph{volume} of $\bf A$ to be $\operatorname{vol}({\bf A}) := \prod_{i=1}^r(2\lfloor M_i \rfloor + 1)$. Note that $|A| \le \operatorname{vol}({\bf A})$, with equality if and only if $\bf A$ is proper.

The next claim follows almost directly from the definition.
For two sets $A$ and $B$ from an additive group $G$ we denote their \emph{sumset} by $A+B = \{x+y: x \in A, y \in B\}$.
\begin{prop}\label{prop.gap_union}
    Let $A_1$ and $A_2$ be two symmetric GAPs
    of rank $r_1$ and $r_2$ respectively and volume $v_1$ and $v_2$ respectively.
    Then their sumset $A_1 + A_2$ can be represented by a symmetric GAP of rank $r_1 + r_2$ and volume $v_1 v_2$.
\end{prop}

\begin{proof}
   The sumset is exactly the symmetric GAP with generators obtained by concatenating the generator tuples of $A_1$ and $A_2$
   and dimensions obtained by concatenating the dimension tuples of $A_1$ and $A_2$.
\end{proof}

\medskip

We will need the following particular version of the inverse Littlewood--Offord theorem:

\begin{theorem}[Theorem~2.1 of Nguyen and Vu~\cite{nguyenvu2011}]
    \label{thm.nguyenvu2011}
Let $\varepsilon < 1$ and $C$ be positive constants. Let
    $V = \{v_1, \dots, v_n\}$ be a multiset from a torsion-free group $G$. Let
$X_1, \dots, X_n$ be independent Rademacher random variables $(\pr(X_i=1) = \pr(X_i=-1) = \frac 1 2)$.
Assume that
\[ \rho(V):=\conc(v_1 X_1 + \dots + v_n X_n) \ge n^{-C}.\]
Then, there exists a proper symmetric GAP $A$ of rank $r = O_{C,\varepsilon}(1)$ that contains all but at most $\varepsilon n$ elements of $V$ (counting multiplicity), where
    \[|A| = O_{C,\varepsilon} (\rho(V)^{-1} n^{-\frac r 2}).\]
\end{theorem}
In this paper we only use $G=\mathbb{R}$. We
will apply the theorem even in a simpler case where $V \subset \mathbb{Z}$,
however the generators of the resulting GAP may still be in $\mathbb{R}$ in general.

The next lemma is our main application. 

\begin{lemma}\label{lem.inv_lo}
    Let $N$ be a fixed positive integer, $c > 0$ and $\epsilon > 0$.
    Then there exist $r=r(N, c, \epsilon)$ and  $v=v(N, c, \epsilon)$ such that the following holds.
    
    For any sequence $X_1, \dots, X_n$ of independent random variables on $\mathbb{Z}$
    such that $\pr(X_i = j) = \frac {a_{i,j}} N$ where $a_{i,j} \in \mathbb{Z}$, 
    $\conc(X_i) = \max_{j} \frac{a_{i,j}} N \le \frac 1 2$
    and $\conc(X_1 + \dots + X_n) \ge \frac c {\sqrt n}$
    there
    exists a symmetric GAP ${\bf A}$
    of rank at most 
    $r$ and volume at most $v$
    such that all but at most $\epsilon n$ of random variables $X_i$
    have their support contained in ${\bf A}$.
\end{lemma}

We will need the next simple lemma for the proof.

\begin{lemma}\label{lem.connected_decomposition}
    Let $\mu$ be a probability measure with support $V = \{x_1, \dots, x_k\} \subset \mathbb{R}$ and probabilities of the form $\mu(\{x_i\}) = \frac {a_i} N$, where $N, a_1, \dots, a_k \in \{1,2,\dots\}$. Suppose $N$ is even and $\conc(\mu) \le \frac 1 2$. Then $\mu$ can be represented as
    \[
        \mu = \sum_{i=1}^m w_i \mu_i; \quad w_i = \frac {a_i'} N > 0; \quad a_i' \in \{1,2,\dots\}
    \]
    where each $\mu_i$ is a uniform distribution on two distinct points in $V$
    and furthermore the graph $G = (V, E)$, where $E=\{\supp(\mu_i), i \in \{1, \dots, m\}\}$, is connected.
\end{lemma}

\begin{proof}
    We obtain the decomposition using the idea from Lemma~3.1 of Juškevičius~\cite{tj2024} and then show how to modify it so that the corresponding graph is connected. More precisely, we think of $\mu$ as $N$ units of mass $\frac 1 N$, where $a_i$ units are placed on the point $x_i$ of the real line. Let these units listed from left (smallest $x_i$) to right (largest $x_i$), ordering the units placed on the same point arbitrarily, be $o_1, \dots, o_N$. For $i \in \{1, \dots, N\}$ let $o_i$ be placed on the point $x_{j_i}$, $j_i \in \{1, \dots, k\}$.
    For $i \in \{1, \dots, N\}$ denote by $\mu_i^0$ the uniform distribution on $\{x_{j_i}, x_{j_i + \frac N 2}\}$. Clearly $x_{j_i} \ne x_{j_i +\frac N 2}$, because otherwise we would have more than $\frac N 2$ units placed on a single point which would mean $\conc(\mu) > \frac 1 2$. Also $\mu = \frac 2 N \sum_{i=1}^{\frac N 2} \mu_i^0$. Collect the weights of the distributions in $\{\mu_i^0\}$ having the same support sets. We get
    \begin{align*}
        &\mu =\sum_{i=1}^{m'} w_i' \mu_i' \, \mbox{ for some } m' \le \frac N 2, 
        \,
        w_i'= \frac {2 a_i'} N
        \mbox{ and }
        a_1',\dots,a_{m'}' \in \{1,2,\dots\} 
    \end{align*}
    where $\mu_1', \dots, \mu_{m'}'$ are now uniform 2-point distributions with \emph{distinct} supports in $V$.

    Now consider the graph $G' = (V, \{\supp(\mu_i'), i \in \{1, \dots, m'\}\})$. Let $C_0$, $\dots$, $C_{t-1}$ be its connected components and for each $j \in \{0, \dots, t-1\}$ fix an edge $e_j = \{y_j, z_j\}$ in $C_j$. For $i \in \{1, \dots, m'\}$ define $\mu_i = \mu_i'$ and weights
 \[
w_i = \begin{cases}
    \frac{2a_i' - 1}{N}, & \text{if $\supp(\mu_i')$ is one of the edges $\{e_j\}$,} \\
\frac{2a_i'}{N}, & \text{otherwise.}
\end{cases}
\]
    For $l \in \{0, \dots, t-1\}$ introduce
    \[
        \mu_{m'+l+1} = \nu_{\{y_l, z_{(l+1) \bmod t}\}}, \quad w_{m' + l + 1} = \frac 1 N
    \]
    where $\nu_A$ denotes the uniform probability measure on the set $A$.
    Let $m=m' + t$. We claim that
    \[
        \mu = \sum_{i=1}^m w_i \mu_i  
    \]
    is a decomposition as required.
    Indeed $(\sum_{i=1}^m w_i \mu_i)(\{x\}) = (\sum_{i=1}^{m'} w_i' \mu_i')(\{x\})$ for any $x$ that does not belong to any edge $e_j$.
    If $x \in e_j$ for some $j$ then the total probability mass at $x$ also does not change as weight $\frac 1 N$ 
    is moved from an existing edge in $G'$ incident to $x$ to a newly added edge connecting $x$ with a vertex in another component. 

    The weights $\{w_i, i \in \{1, \dots, m\}\}$ are as required by construction.
    As $w_i \ge \frac 1 N > 0$ for all $i$, each  distribution corresponding to an edge of $G'$ remains in the final decomposition. The edges $\supp(\mu_{m' + l + 1})$, $l\in \{0,\dots,t-1\}$ 
    make up a cycle that visits each component in $\{C_j\}$, so the resulting graph $G$ is connected. 
\end{proof}
\medskip

\begin{proofof}{Lemma~\ref{lem.inv_lo}}
    Let $S = X_1 + \dots + X_n$.
    Suppose first that $a_{i, 0} > 0$ for $i \in \{1, \dots, n\}$.
   

    For each $i \in \{1, \dots, n\}$ fix a decomposition of the distribution $\mu_i$ of $X_i$ obtained by Lemma~\ref{lem.connected_decomposition};
    let the resulting decomposition be $\mu_i = \sum_{j=1}^{k_i} \frac {t_{i,j}} {2 N} \mu_{i,j}$, $t_{i,j} \in \mathbb{Z}$, $t_{i,j} > 0$. Here, in order to apply the lemma, we represented all probabilities $\frac{a_{i,j}} N$ as $\frac{2 a_{i,j}} {2N}$ to ensure that the denominator is even.

    We call a tuple of $n$ integers $\omega$ a 
    \emph{configuration} if $\omega_i \in \{1, \dots, k_i\}$ for each $i \in \{1, \dots, n\}$. 
    Write $p_\omega = (2N)^{-n} \prod_{i=1}^n {t_{i,\omega_i}}$. 
    Denote by $S_{\omega}$ the sum of
    independent random variables $Y_1, \dots, Y_n$ where for each $i$ we have $Y_i \sim \mu_{i,\omega_i}$.

    Define a random variable $S_W$ as follows.
    \begin{enumerate}
        \item Draw a random configuration $W$ according to the distribution $\{p_\omega\}$.
        \item Given $W = \omega$, draw a value distributed as $S_{\omega}$.  
    \end{enumerate}  
    We have $S_W \sim S$.
    Denote $q(\omega) := \conc(S_\omega) =\max_x \pr(S_\omega = x)$.
    Let $G_W$ be the event that $q(W) \ge \frac c {3 \sqrt n}$. Let  $x_0 \in \mathbb{Z}$
    satisfy $\pr(S=x_0) = \conc(S)$. Then
    \begin{align}\nonumber
         \conc(S) & = \E \E(\mathbb{I}_{S_W=x_0} | W) \le \E q(W) = \E q(W) \I_{G_W} + \E q(W) \I_{\bar{G}_W}
        \\ &  
        \le \pr(G_W) \sup_\omega q(\omega) + \frac c {3 \sqrt n} \le \frac {\pr(G_W)} {\sqrt n} + \frac c {3 \sqrt n}. \label{eq.G_W}
    \end{align}
    The last inequality is obtained as follows. By the well-known result of Erd\H{o}s, see, e.g.~\cite{jk2021}, $q(\omega) \le \frac {\binom n {\lfloor \frac n 2 \rfloor}} {2^n}$. 
    A simple proof of an upper bound
    $\binom {2n} n \le \frac {4^n} {\sqrt {\pi n}}$ for any positive integer $n$ is found, e.g., in \cite{elkies}.
    Thus for even $n$ we have $q(\omega) \le \frac {\sqrt 2} {\sqrt {\pi n} } \le n^{-\frac 1 2}$.
    Using Pascal's identity we can easily see that 
    $q(\omega)  \le n^{-\frac 1 2}$ 
    also for odd $n$.
    
    So from the assumption of the Lemma $\conc(S) \ge \frac c {\sqrt n}$ and (\ref{eq.G_W}) we get
    \begin{equation}\label{eq.G_W2}
        \pr(G_W) 
        \ge \frac {2 c } 3.
    \end{equation}
    Note, see also \cite{tj2024}, that since $\conc(X_i) \le \frac 1 2$ we have $\conc(S) \le \sup_{\omega} q(\omega) \le n^{-\frac 1 2}$. Thus the assumption of the lemma $\conc(S) > \frac c {\sqrt n}$ can only be satisfied with $c \le 1$. 


    Consider the distributions $\mu_{i,j}$ in the decomposition of $\mu_i$. As the mixing weights are multiples of $\frac 1 {2N}$, there are at most $2N$ components in this decomposition.
    Furthermore, by Lemma~\ref{lem.connected_decomposition}, 
    the supports of 2-point distributions $\mu_{i,j}$ in the decomposition of $\mu_i$ form a connected graph $G_i$ on $\supp (\mu_i)$.
    For each $i$ fix a spanning tree $T_i$ of $G_i$.
    So $V(T_i) = V(G_i) = \supp(X_i)$, so by the assumption of the lemma $|V(T_i)| \le N$.

    By the assumption of the lemma $0 \in \supp(\mu_i)$ for each $i$.
    Suppose we can prove that there is a set $\mathcal{A}=\{{\bf A}_1,\dots, {\bf A}_t\}$ of symmetric GAPs
    of volume at most $v_1$, rank at most $r_1$ and integer dimensions which satisfies the following:
    for a subset $I$ of at least $1-\epsilon n$ indices $i$, for every edge $ab \in T_i$ with $i \in I$
    we have that $b-a$ belongs to some GAP in $\mathcal{A}$. (Notice that if a symmetric GAP contains $a-b$ then 
    it also contains $b-a$.)
    Then for each index $i \in I$ all the points in the support of $\mu_i$
    belong to the set 
    \begin{equation}\label{eq.sumset_gap}
        (N-1) {\bf A}_1 + (N-1) {\bf A}_2 + \dots + (N-1) {\bf A}_t.
    \end{equation}
    To see this, 
    consider the path $P = x_1\dots x_m$ on $T_i$ from $0$ to any $x \in \supp(\mu_i)$,
    so that $x_1=0$ and $x_m=x$. Recall that by the assumptions of the lemma $0 \in \supp(\mu_i)$. 
    Then $x = \sum_{j=1}^{m} (x_{j} - x_{j-1})$, i.e. $x$ is a sum of at most $N-1$ elements 
    each from a
    symmetric
    GAP in $\mathcal{A}$. Since our dimensions are integer, 
    the underlying set of the dilate $(N-1) {\bf A}_i$ 
    is the iterated sumset:
    \[
        \underbrace{A_i + \cdots + A_i}_{N-1 \text{ times}}.
    \]
    Also $(N-1) A_i$ has the same rank as $A_i$ and volume at most $(N-1)^{r_1} v_1$. 
    Using Proposition~\ref{prop.gap_union},
    the set (\ref{eq.sumset_gap}) 
    can be represented by a symmetric (but not necessarily proper) GAP $\bf A$ of rank at most $r_1 t$ and volume at most $(N-1)^{r_1t} v_1^t$.

    Let us 
    prove that such a set $\mathcal{A}$ with constant-sized $r_1, v_1$ and $t$ exists.

    We describe a sequential process to pick $t$, 
    $t \le t_0:=\frac {16 N^2} {\epsilon}$
    configurations $\omega_1, \dots, \omega_t$
    and
    GAPs ${\bf A}_1, \dots, {\bf A}_t$ with the required property.
    We use a simple ``probabilistic method'' argument.


    Let $l$ be a positive integer. Suppose we have already chosen configurations $\omega_1, \dots, \omega_{l-1}$
    and GAPs ${\bf A}_1, \dots, {\bf A}_{l-1}$. At step $l$ we need to choose the $l$th configuration $\omega_l$ and GAP ${\bf A}_l$.

    For $i \in \{1, \dots, n\}$ we will call an edge $e = ab \in T_i$ 
    \emph{open before step $l$}, if $b-a$ is not contained in any GAP in ${\bf A}_1, \dots, {\bf A}_{l-1}$. 
    Let $o_{l-1}$ be the number of indices $i \in \{1, \dots, n\}$ 
    such that $T_i$ has any open edge before step $l$.
    Of course, $o_0=n$. 
    Assume that $T_i$ has at least one open edge before step $l$. Take a random configuration $W$ according to $\{p_{\omega}\}$.
    The probability that $\supp(\mu_{i,W_i})$ is an open edge in $T_i$ before step $l$
    is at least $\min_j \frac {t_{i,j}} {2N} \ge \frac 1 {2N}$. 
    Let $Z_l = Z_l(W)$ be the random number of indices
    $i$ where an open edge of $T_i$ is sampled (i.e. $\supp(\mu_{i,W_i})$ is open in $T_i$ before step $l$). 
    Suppose $o_{l-1} \ge \epsilon n$.
    Then by a Chernoff-type bound, see, Theorem~2.3(c) of~\cite{cmcd98}, $\pr(Z_l \ge \frac {\epsilon n} {4 N}) \ge 1 - \exp(-\frac {\epsilon n} {16 N})$.
    Thus the probability that $G_W$ holds and $Z_l \ge \frac {\epsilon n} {4 N}$ is at least $\frac {2 c } 3 - \exp(-\frac {\epsilon n} {16 N})$ and,
    assuming 
    \begin{equation}\label{eq.lpm_n_more}
        n >  n_0(N,c,\epsilon) := 16 N \epsilon^{-1} \ln \frac 3 {2 c},
    \end{equation}
    this probability is positive and a configuration $\omega_l$ on which both of the events hold exists.
   
    Let us fix a particular such $\omega_l = (\omega_{l,1}, \dots, \omega_{l, n})$.
    For $i \in \{1, \dots, n\}$ denote by $\{a_i, b_i\}$, where  $a_i, b_i \in \mathbb{Z}$, $a_i=a_i(l)$, $b_i=b_i(l)$, the support of the two-point distribution $\mu_{i, \omega_{l,i}}$. Then $S_{\omega_l} = \sum_{i=1}^n \frac {a_i + b_i} 2 +  \sum_{i=1}^n \frac {b_i - a_i} 2 \tau_i$
    where $\tau_1, \dots, \tau_n$ are independent Rademacher random variables.
    Define $n' := \lceil \frac {9 n} {c^2}\rceil$. Fix a small $\epsilon' \in (0,1)$.
    Since $G_{W}$ holds on $W = \omega_l$ we have $\conc(S_{\omega_l}) = q(\omega_l) \ge \frac c {3 \sqrt n} \ge (n')^{-\frac 1 2}$.
    Pad the sequence $(\frac {b_i - a_i} 2 \tau_i, i \in \{1, \dots, n\})$ with $n'-n$ new independent Rademacher random variables, each with multiplier 0, to get a sequence of length $n'$.

    Applying Theorem~\ref{thm.nguyenvu2011} to the sum of the padded sequence with $\varepsilon = \epsilon'$, $C=\frac 1 2$ and
    $n = n'$ we get that all but at most $\epsilon' n'$ elements $(\frac {b_1 - a_1} 2, \dots, \frac {b_n-a_n} 2, 0, \dots 0)$ (equivalently, of the same sequence multiplied by 2) belong to a symmetric GAP of
    a bounded rank and
    size
    at most 
    $r_1=r_1(\epsilon')$.
    It is easy to see that such a GAP is contained in a symmetric GAP ${\bf A}_l$ of rank $r_1$, all dimensions 1, and volume
    $v_1=v_1(\epsilon') = 3^{r_1}$. (If $n$ is large enough we have even $r_1=1$ by the theorem, however we will not use this here.) 

    We choose $\epsilon'$ small enough such that $\epsilon' n' \le  \frac {\epsilon n} {8 N}$ for all $n$ (since, as we noted above, it must be $c \le 1$, using (\ref{eq.lpm_n_more}) we can take $\epsilon'=\frac {\epsilon c^2} {144N}$).
    Then we get that among  at least $\frac {\epsilon n} {4N}$ indices $i$
    with an open edge $a_i b_i$ in $T_i$ before step $l$
    we can find at least  $\frac {\epsilon n} {4N} - \epsilon' n' \ge \frac {\epsilon n} {8N}$ ones where $b_i - a_i$ belongs to 
    ${\bf A}_l$. 

    Initially there are in total at most $2 n N$ open edges over all $i \in \{1, \dots, n\}$. 
    At each step we reduce the number of total open edges by at least $\frac {\epsilon n} {8 N}$. So after some number $t$,
    $t \le t_0$, steps we must have
    that the total number of open edges before step $t+1$ is less than $\epsilon n$, which implies $o_t < \epsilon n$ as required and we can stop.

    Finally, for $n < n_0(N, c, \epsilon)$,  $\sum_{i=1}^n X_i$ always belongs to the sumset of 
    $n$ sets of size at most $N$, so it is a subset of a symmetric GAP of rank at most $r_0 = N n_0(N, c, \epsilon)$ 
    and volume at most $3^{r_0}$.
    Setting $r(N, c, \epsilon) = \max(r_0, r_1 t_0)$ and $v(N, c, \epsilon) = \max(3^{r_0}, (N-1)^{r_1 t_0} v_1^{t_0})$ completes the proof
    in the case when $a_{i, 0} > 0$ for each $i$.

    Finally, for the general case let $m_i = \min (\supp(X_i))$ apply the above proof for the sequence $(X_i - m_i)$ to get
    a gap $\bf A$ and notice that we can cover the translate $A + \sum_{i=1}^n m_i$ of the underlying set $A$ by
    adding a new generator $\sum_{i=1}^n m_i$ with dimension 1, thus increasing the rank by at most 1 and the volume at most 3 times.

\end{proofof}

\subsection{A theorem of Odlyzko and Richmond}
\label{sec.odlyzko_richmond}

Here we present the last significant result that we will need in our proof in Section~\ref{subsec.lpm}. 

Let $\mu$ be a distribution on integers with $|\supp(\mu)| \ge 2$. We call $s\in\mathbb{Z}$ its \emph{maximum span} if the support of $\mu$
is contained in $\{i_0 + i s: i \in \mathbb{Z}\}$ for some $i_0 \in \mathbb{Z}$ and there is no other such $(i_0', s')$
with $s' > s$. Note that if $A$ is the support of $\mu$ and $x \in A$ then $s = \gcd(\{z-x: z \in A \setminus\{x\}\})$.
In the case where $\mu$ is supported on a single point, define the maximum span to be $\infty$.

Answering an open question initiated by R\'enyi in 1977, Odlyzko and Richmond proved 
the following result on $n$-fold convolution of a distribution on a finite set of integers with maximum span 1.

\begin{theorem}[Theorem~2 of Odlyzko and Richmond~\cite{odlyzkorichmond}]\label{thm.odlyzko_richmond}
 If $\{p_j\}$ is a discrete distribution with $p_j = 0$ for $j < 0$ and $j > d$, while $p_0 > 0, p_d > 0$, and
    \[
        \gcd\{j: p_j \ne 0\} = 1,
    \]
    then for any $\delta > 0$ there is an $n_0^{OR} = n_0^{OR}(\delta)$ such that if $a_{k,n}$ denotes the value of the
    $n$-fold convolution $\{p_j\}^{*_n}$ at $k$ then for $n \ge n_0^{OR}$,
    \[
        a_{k,n}^2 \ge a_{k-1,n} a_{k+1,n}
    \]
    for $\delta n \le k \le (d-\delta) n$.
\end{theorem}

\subsection{Proof of Lemma~\ref{lem.less_peaked_medium}}
\label{subsec.lpm}

\begin{proofof}{Lemma~\ref{lem.less_peaked_medium}}
   For the beginning of the proof we follow the steps of the proof
   of Lemma~\ref{lem.less_peaked_large}.
    We define $S = \sum_{i=1}^n X_i$, 
    $S' = \sum_{i=1}^n Y_i$, $V=\Var S'$ 
   and,
   using an analogous argument as in the proof of Lemma~\ref{lem.less_peaked_large},
    conclude 
    that it suffices to prove that there is a constant $V_{\epsilon, K}$ such that whenever $V \ge V_{\epsilon, K}$,
   (\ref{eq.less_peaked_large_to_show_clt}) holds for each $k$ satisfying (\ref{eq.less_peaked_large_k}),
   where $c_\epsilon$ is defined in (\ref{eq.c_epsilon}).

   So let $k$ be as in (\ref{eq.less_peaked_large_k}).
    Suppose $\conc_k(S) \ge (1+\epsilon) \conc_k(S')$. Using Lemma~\ref{lem.logconcvariancenew} and recalling that $\Var U_k' = \frac {k^2 - 1} {12}$ for 
    $U_k' \sim \nu_{\frac 1 k}^+$ we get
    \begin{align*}
        \conc_k(S') & =  k \pr(S' + U_k' = 0)
        \ge \frac {k} {\sqrt{12(\Var S' + \Var U_k') + 1}}
        \\ & \ge \frac {k} {\sqrt{12 (1 + 2 c_{\epsilon}^{-1}) V + 1}}
        \ge \frac {k} {\sqrt{36 c_{\epsilon}^{-1} V  + 1}} \ge \frac k {\Cnineteen} c_\epsilon^{\frac 1 2} V^{- \frac 1 2}.
    \end{align*}
    Here we
    used
    that $c_\epsilon^{-1} \ge 1$, see (\ref{eq.c_epsilon}),
    and assumed
    without loss of generality that 
    $V \ge 1$. 
    Thus
    \begin{align}\label{eq.lpmsqn}
        k \conc(S) \ge \conc_k(S) \ge \conc_k(S') \ge \frac {k} {\Cnineteen} c_\epsilon^{\frac 1 2} V^{- \frac 1 2} \ge \frac { k \sqrt{12 c_\epsilon}} {\Cnineteen K \sqrt{n}}.
    \end{align}
    In the last inequality we used $V \le n \max_i \Var Y_i \le n\frac {K^2 - 1} {12} \le \frac {K^2 n} {12}$ which follows from the assumptions of the lemma and Lemma~\ref{lem.variance}.
Similarly,
since by the assumptions of the lemma it must be
$K \ge 2$, we have $V \ge \frac {K^2 n} {24}$. 
Note that these bounds for $V$ and (\ref{eq.less_peaked_large_k}) imply
\begin{align}\label{eq.lpm_k}
    K \sqrt{2^{-1} c_\epsilon n} \le  k \le  K \sqrt{2 c_\epsilon^{-1} n}.
\end{align}
    Using (\ref{eq.less_peaked_large_k}) and 
    the last two inequalities in (\ref{eq.lpmsqn}) 
    \begin{align}\label{eq.lpm_Q_lower}
        \conc_k(S) \ge \conc_k(S') \ge \frac{\sqrt{12} c_\epsilon} {\Cnineteen} \quad \mbox{and}\quad \conc(S) \ge  \frac{\sqrt{12 c_\epsilon}} {\Cnineteen K \sqrt n}.
    \end{align}

    As the concentration function is invariant to translations, we can assume without loss of generality that $0 \in \supp (X_i)$ for each $i$.
    We can now apply Lemma~\ref{lem.inv_lo} to $S$ with $N = K^2$, 
    $c = \frac{\sqrt{12 c_\epsilon}} {\Cnineteen K}$
    and $\epsilon = \epsilon'$, where $\epsilon' \in (0,1)$ will be specified below, to get that all but
    at most $\epsilon'$ random variables in $X_1, \dots, X_n$ have their support
    contained in a symmetric GAP $\bf A$ of rank $r$ and volume $v$, such that
    \begin{equation}\label{eq.lpm_r_v}
        r\le r_{K, \epsilon, \epsilon'} \quad \mbox{and}\quad v \le v_{K, \epsilon, \epsilon'}
    \end{equation}
    where on the right of each inequality is a constant that depends just on $K, \epsilon$ and $\epsilon'$.

    Let $J$ denote those $i$ for which the support of $X_i$ 
    is contained in
    $\bf{A}$.


    Given positive integers $l$ and $d$ where $l$ is even, define 
    $\mathcal{C}_l^d = \{-\frac {l} 2, \dots, \frac {l} 2\}^d$, 
    the discrete cube  in $\mathbb{Z}^d$ with side length $l$
    centered at $\bf 0$.

    Let $M = \max_{i \in \{1, \dots, r\}} 2 M_i$, where $M_1, \dots, M_r$ are the dimensions of the symmetric GAP $\bf A$,
    which we can assume to be integer.
    Note that $M + 1 \le \prod (2 M_i + 1) = v$, so $M$ is bounded by a constant:  $M \le v \le v_{K, \epsilon, \epsilon'}$. (Below in the proof there will be a number of bounds with $v$ and/or $r$ on the right. In all cases the expressions will be non-decreasing in $r$ and $v$ so we implicitly assume that (\ref{eq.lpm_r_v}) is applied in the last step.)

Thus we can represent each distribution $X_i$ with $i \in J$ as:
\[
    X_i = \langle a, {\bf X}_i \rangle
\]
    where ${\bf X}_i$ is a random vector with values in
    $\mathcal{C}_M^r$,
    $a$ is the vector ($r$-tuple) of generators of the GAP $\bf{A}$,
    and, without loss of generality, $a \in \mathbb{R}^r \setminus \{{\bf 0}\}$. 
    
    For each $i \in J$ and $y \in \mathbb{Z}$ if there are several atoms $x$ of ${\bf X}_i$ such that $\langle a, x\rangle = y$, then we can put all the mass to one of them.
    Similarly we can move all the mass from atoms $x$ such that $\langle a, x \rangle = 0$ to the zero vector~$\bf 0$.
    After this modification we can assume that each ${\bf X}_i$ has an ``extremal distribution'' inside $\mathcal{C}_M^r$:
    there are $\lfloor \alpha_i^{-1} \rfloor$ atoms in $\mathcal{C}_M^r$ with mass $\alpha_i$
    and one different atom in the same set with the rest of the mass $1 - \lfloor \alpha_i^{-1} \rfloor \alpha_i$
    and ${\bf 0} \in \supp({\bf X}_i)$.

    Observe that there is a finite number 
    \begin{equation}\label{eq.lpm_possible_discretized}
        N = N(\epsilon, \epsilon', K) \le (M+1)^{rK^2}
    \end{equation}
    of possible distributions on the cube 
    $\mathcal{C}_{M}^r$
    with our 
    quantized
    masses of the form $\frac j {K^2}$ (place $K^2$ masses of $K^{-2}$ on one of $(M+1)^r$ points).
    Denote these distributions $\mu_1, \dots, \mu_N$
    and suppose there are $k_i$ random vectors ${\bf X}_i$, $i \in J$ which have distribution $\mu_i$,
    so that $\sum_{i=1}^N k_i = |J| \ge (1-\epsilon') n$.

%
    We would now like to apply the multivariate local limit theorem, Theorem~\ref{thm.llt_stein} from Section~\ref{sec.precise} to $\sum_{i \in J} {\bf X}_i$.
    However, this theorem works only if $d_{TV}({\bf X}_i, {\bf X}_i + e^{(j)})$
    is bounded away from 1 for each coordinate vector $e^{(j)}$. We need to consider the possibility that the support of ${\bf X}_i$ belongs to some proper subgroup of $\mathbb{Z}^r$.

    To 
    simplify the setup,
    we will use the following observation:
    since the number $N$ of different possible distributions is constant, it is easy to group almost all random vectors
    in $\{{\bf X}_i, i \in J\}$ into constant-sized groups (blocks) so that the sum of each group has the same distribution.

    Let us make this more specific. Fix a large integer $b$, to be specified below.
    Create a single group by taking $a_i = \lfloor \frac {b k_i} {|J|}  \rfloor$ of elements with distribution $\mu_i$ for
    each $i \in \{1, \dots, N\}$.
    For each distribution $\mu_i$ we can find at least $a_i \lfloor \frac {|J|} b \rfloor$ indices $j \in J$ with ${\bf X}_j \sim \mu_i$. So we can form 
    $\lfloor \frac {|J|} b \rfloor$ of groups of indices in $J$, such that each group has $a_1$ indices
    $j$ with ${\bf X}_j \sim \mu_1$, $a_2$ indices $j$ with ${\bf X}_j \sim \mu_2$, $\dots$, and $a_N$ indices $j$ with  ${\bf X}_j \sim \mu_N$.

    Let $b'$ be the total number of elements in a single group. We have
    \begin{equation}\label{eq.lpm_b'}
        b \ge b' = \sum_{i=1}^N a_i = \sum_{i=1}^N \lfloor \frac {b k_i} {|J|}  \rfloor \ge \sum_{i=1}^N (\frac {b k_i} {|J|} - 1)  = b - N. 
    \end{equation}
    Let $J'$, $J' \subseteq J \subseteq \{1, \dots, n\}$, be the set of indices of all the $b' \lfloor \frac {|J|} b \rfloor$  grouped random variables.
    We have
    \begin{equation}\label{eq.lpm_J'}
        |J'| 
        \ge (b-N) (\frac {|J|} b - 1) \ge |J| - b - \frac {|J| N} b \ge |J| - n(\frac b n + \frac N b).
    \end{equation}

    Let $\delta' = \delta'(\epsilon, K)$, $\delta' \in (0,\frac 1 4)$ be a constant to be defined below. We will require that the constants $\epsilon'$, $b = b(\delta', M, r, N)$ and the variance $V$ satisfy
    \begin{equation}\label{eq.lpm_eps}
        \epsilon' \le \frac {\delta'} 3, \quad \frac N b \le \frac {\delta'} 3 \quad\mbox{and}\quad V \ge \frac 1 4 K^2 b / \delta'. 
    \end{equation}
    In particular, we set
    \[
        \epsilon' := \frac {\delta'} 3 \quad\mbox{and}\quad b := \left \lceil \frac {3 N} \delta \right \rceil,
    \]
    so that the first two inequalities of (\ref{eq.lpm_eps}) are satisfied.

    Recall that $V = \Var \sum_{i=1}^n Y_i \le n \frac {K^2 -1} {12}$.
    Thus for any $V_0$, $V \ge V_0$ implies
    \begin{equation}\label{eq.lpm_nV0}
        n 
        \ge \frac{12} {K^2} V
        \ge \frac{12} {K^2} V_0.
    \end{equation}
    Thus (\ref{eq.lpm_nV0}) and the last inequality in (\ref{eq.lpm_eps}) imply that $b/n \le \frac {\delta'} 3$.
    Thus by (\ref{eq.lpm_J'}) and (\ref{eq.lpm_eps}) 
    \begin{equation} \label{eq.lpm_delta'}
        |J'| \ge n - \epsilon' n - n(\frac N b + \frac b n) \ge n (1-\delta'),
    \end{equation}
    so the number $m$ of groups 
    satisfies 
    \begin{equation}\label{eq.lpm_m}
        m = \frac {|J'|}  {b'} = \lfloor \frac {|J|} {b}  \rfloor \ge \frac {n(1-\delta')} {b'} \ge \frac {n (1-\delta')} b.
    \end{equation}
    Also observe that (\ref{eq.lpm_b'}) and (\ref{eq.lpm_eps}) imply that 
    \begin{equation}\label{eq.lpm_b_half}
        b \ge b' \ge \frac b 2.
    \end{equation}



    Let $\mu = \mu_1^{*_{a_1}} * \dots * \mu_N^{*_{a_N}}$, that is, $\mu$ is the distribution
    of the sum of $b' = a_1 + \dots + a_N$ random vectors ${\bf X}_j$ with indices from any fixed group. We then have that 
    \[
        \sum_{i \in J'} \langle a, {\bf X}_i \rangle = \sum_{j=1}^{m} \langle a, {\bf Z}_j \rangle 
    \]
    where 
    ${\bf Z}_1, \dots, {\bf Z}_m \sim \mu$ are i.i.d.
    Importantly, $\mu$ is still a distribution on a constant discrete cube
    $\mathcal{C}_{b M}^r$. 


    Let $L'$ be the \emph{integer span} (the set of all possible linear combinations with integer coefficients)
    of the support of $\mu$. 
    Let $r'$ be the dimension of the linear span of $L'$. Obviously, $1 \le r' \le r$.
    We say that $\mu$ \emph{has minimal lattice} $\mathbb{Z}^r$ if $L' = \mathbb{Z}^r$.
    Suppose this is not the case, that is $L' \subsetneq \mathbb{Z}^r$.
    By the lattice basis theorem, 
    see, e.g., Theorem 2.2.2 of \cite{matousek} or Theorem~21.1 of \cite{bhattacharyarao2010},
    $L'$ is a \emph{lattice} generated by linearly independent vectors $x_1, \dots, x_{r'} \in L'$, that is
    $L' = \{B z: z \in \mathbb{Z}^{r'}\}$ for some $r \times r'$ integer matrix $B$ of rank $r'$.
    As $z \mapsto Bz$ is injective for $z \in \mathbb{Z}^{r'}$, 
    we can define a measure $\mu'$ on $\mathbb{Z}^{r'}$ by setting
    $\mu'(\{z\}) = \mu(\{Bz\})$ for each $z \in \mathbb{Z}^{r'}$.
    So we have $\mathbf{Z}_i \sim B \mathbf{Z}_i'$ where $\mathbf{Z}_1', \dots,\mathbf{Z}_m'$ are i.i.d. with distribution $\mu'$.
    


    Lemma~\ref{lem.korkinzolotarev} below shows that the basis $B$ can be chosen so that 
    $\supp(\mu')$ is contained in the discrete cube $\mathcal{C}_{M'}^{r'}$ of $\mathbb{Z}^{r'}$
    with side length 
    \begin{equation}\label{eq.lpm_M'}
        M' = 2 \left \lceil r^{r-1} \left(\frac {\sqrt{r} bM} 2\right)^{2r-1} \right \rceil.
    \end{equation}
    We fix such $B$.
    We can write
    \[
        \sum_{i \in J'} \langle a, {\bf X}_i \rangle = \sum_{i=1}^m \langle a, B {\bf Z}_i' \rangle = \langle B^T a, \sum_{i=1}^m {\bf Z}_i' \rangle
    \]
 where now $\sum_{i=1}^m {\bf Z}_i'$ is a sum of $m$ i.i.d. bounded random vectors in $\mathbb{Z}^{r'}$
 and $\mathbb{Z}^{r'}$ \emph{is} their minimal lattice.

Our next main step will be to use Theorem~\ref{thm.llt_stein} to approximate $\sum_{i=1}^m \mathbf{Z}_i'$ by a standard normal random vector. This theorem gives a non-trivial conclusion only if the total variation distance between a single random summand and the same summand shifted by any
coordinate vector is bounded away from 1. To ensure this property, we will group once more.

Let $A$ be the $r' \times s$ matrix containing all the $s$ vectors in the support of $\mu'$
as columns.
Let $e^{(i)}$, $i \in \{1, \dots, r'\}$ be the $i$th coordinate vector of $\mathbb{Z}^{r'}$.
Since the support of $\mu'$ generates $\mathbb{Z}^{r'}$, there is at least one 
solution
$x \in \mathbb{Z}^s$
to $A x = e^{(i)}$ and $A$ has rank $r'$.
Also we must have $s > r'$, otherwise, since $\bf{0}$ is one of the vectors in the support of $\mu'$, the subspace generated by the support would have dimension less than $r'$.
The main result of Borosh, Flahive, Rubin and Treybig \cite{bfrt1989}
says that in this case there exists a solution $x$ with $|x_i|$ bounded by
the maximum of the absolute values of $r' \times r'$ minors of the augmented matrix $(A | e^{(i)})$.
In our case, using the Leibniz formula, this maximum is at most $r! (M')^r$.
Thus for each $i \in \{1, \dots, r'\}$, $e^{(i)}$ belongs to the support of the sum of $D$, $D = s r! (M')^r \le r! (M'+1)^{2r}$, independent random variables with distribution $\mu'$. 

Let $\mu''$, $\mu'' = \mu'^{*_D}$, be the $D$-fold convolution of $\mu'$, let $m' = \lfloor \frac m  D \rfloor$ and let 
${\bf Z}_1'', \dots, {\bf Z}_{m'}''$ be i.i.d. with distribution $\mu''$. $\mu''$
is contained in the discrete cube $\mathcal{C}_{M''}^{r'}$, 
where
\begin{equation}\label{eq.lpm_M''}
    M'' = D M'. 
\end{equation}
Using 
(\ref{eq.lpm_m}) we have that
\[
    m' \ge \frac m D - 1 \ge \frac {n (1- \delta')} {b' D} - 1 = \frac n {b' D} (1 - \delta' - \frac {b' D} n).
\]
Thus, using (\ref{eq.lpm_b_half}), 
if we require that 
\begin{equation}\label{eq.lpm_D}
    \frac {b D} n \le \delta' \quad \mbox{ or }\quad n \ge (\delta')^{-1} b D
\end{equation}
then 
\begin{equation}\label{eq.lpm_m'}
    m' \ge \frac n {b' D} (1-2\delta')
\end{equation}
and
the size of the set $J'' \subseteq J'$ of all the indices of the original random vectors ${\bf X}_i$ that are grouped to the $m'$ newly formed random vectors satisfies $|J''| = m' b' D \ge n (1-2\delta')$.

Furthermore, since by the assumption of the lemma and our construction the minimal non-zero probability of the random variables ${\bf X}_i$ is at least $K^{-1}$, if $x$ belongs to the support of $\mu''$ then $\mu''(\{x\}) \ge K^{-b D}$.

We shall now apply Theorem~\ref{thm.llt_stein} to the sum ${\bf S}''$ defined by ${\bf S}'' = \sum_{i=1}^{m'} {\bf Z}_i'' = \sum_{i=1}^{D m'} {\bf Z}_i'$.
Let us first consider the quantity $u_1 = \min_{1 \le j \le r'} \{1-d_{TV}({\bf Z}_1'', {\bf Z}_1''+e^{(j)})\}$.
By the above assumptions and arguments for any
$\j \in \{1, \dots, r'\}$
\[
   p=\pr({\bf Z}_1'' = e^{(j)}) \ge K^{-b D}  \mbox { and }   q=\pr({\bf Z}_1''+e^{(j)} = e^{(j)}) = \pr({\bf Z}_1'' = {\bf 0}) \ge K^{-b D}.
\]
Thus

\begin{align*}
    & d_{TV}({\bf Z}_1'', {\bf Z}_1'' + e^{(j)}) = \sum_{x \in {\mathbb{Z}^{r'}}} \frac 1 2 | \pr({\bf Z}_1'' = x) -  \pr({\bf Z}_1'' + e^{(j)}=x)|
    \\ &
    \le \frac 1 2 (1 - p) + \frac 1 2 (1-q) + \frac 1 2 |p - q| = 1 + \frac {\max(p,q)  - \min(p,q) - p - q} 2
    \\ &
    = 1 - {\min(p,q)} \le 1 - K^{- b D}.
\end{align*}
Thus $u_1$ from Theorem~\ref{thm.llt_stein} satisfies $u_1 \ge K^{-b D}$,
so that $\tilde{s}_m \ge (m'-1) K^{-b D}$.
Also clearly $\chi$ from that theorem satisfies 
$\chi \le (r^{\frac 1 2} M'')^3$.

Let $\LPMSigma$ be the covariance matrix of 
${\bf Z}_1''$.
As each entry in $\LPMSigma$ 
is bounded by $(M'')^2$, we have 
\begin{equation}\label{eq.lpm_sigma_max}
    \sigma_1(\LPMSigma) = \|\LPMSigma\| \le (M'')^2 r
\end{equation}
Let us derive some lower bound 
on its smallest eigenvalue. Since the probabilities of ${\bf X}_i$
are of the form $\frac i {K^2}$ for $i \in \mathbb{Z}$, by our construction 
the probabilities $\mu''(\{x\})$ and the covariances
are of the form  $\frac i {K^{2b'D}}$ for $i \in \mathbb{Z}$. Thus $K^{2b'D} \LPMSigma$
is an integer matrix. Trivially, its column lengths are bounded by $K^{2b'D}\sqrt{r'} (M'')^2$.
As $\LPMSigma$ is symmetric and positive definite,
its eigenvalues equal singular values.
So by Lemma~\ref{lem.singular_value_lower_bound} the smallest eigen-/singular value satisfies
\begin{equation}\label{eq.lpm_lambda_min}
    \lambda_{r'}(\LPMSigma) = \sigma_{r'}(\LPMSigma) 
\ge (K^{2bD}(M'')^2 r)^{-(r-1)}.
\end{equation}


The trace of $\LPMSigma$ is equal to the sum of its eigenvalues / singular values. Thus it is 
at least $r'\lambda_{r'}(\LPMSigma)$, and, trivially, ${\rm Tr}(\LPMSigma) \le r (M'')^2$.
But since ${\bf Z}_1''$ has values in $\mathbb{Z}^{r'}$ and by the assumptions of the lemma and our construction $\conc({\bf Z}_i'') \le \conc(X_i) \le \alpha_i$, we have 
a better lower bound ${\rm Tr}(\LPMSigma) = \E \sum_{i=1}^{r'}({\bf Z}_{1,i}'')^2 \ge \pr({\bf Z}_1'' \ne {\bf 0})
\ge 1-\conc({\bf Z}_1'') \ge 1 - \max \alpha_i \ge \frac 1 2$. This implies that $L$ in Theorem~\ref{thm.llt_stein} satisfies 
$L \le (m')^{-\frac 1 2} \chi$.

Let now ${\bf Z} \sim N(m' \E{\bf Z}_1'', m' \LPMSigma)$ and recall the notation $\round{\bf Z}$ for rounding coordinatewise.
Note that if $m' \ge 2$ then $\sqrt {\frac {m'} {m'-1}} \le 2$.

Combining these estimates, Theorem~\ref{thm.llt_stein} gives for $m' \ge 2$ 
\begin{align}\label{eq.conclusion_llt}
    d_{TV}({\bf S}'', \round{{\bf Z}}) \le 2 C r^{\frac 7 2} \ln m' \frac {(M'')^3  r^{\frac 3 2}+ r^{\frac 1 2}} {(m')^{\frac 1 2}} K^{\frac {b D} 2}.
\end{align}
Here by the remark of the theorem, $C$ depends ``only'' on $\lambda_{r'}(\LPMSigma)$, $\lambda_{1}(\LPMSigma)$ 
and $(r')^{-1} {\rm Tr}(\LPMSigma)$. We have shown that all of these quantities are bounded above and away from zero by constants 
that depend only on $r$, $M''$, $K$, $b$ and $D$, but ultimately just on $K$, $\epsilon$, $\epsilon'$ and $\delta'$.
The remark could be interpreted in two ways:
{
\renewcommand{\labelenumi}{(\arabic{enumi})}
\begin{enumerate}
    \item As a function of smallest, the largest eigenvalue of $\LPMSigma$ and the trace of $(r')^{-1}\LPMSigma$, $C$ is well-behaved, for example, continuous at each positive parameter vector.
    \item The above assumption does not hold.
\end{enumerate}
}
Based on the wording of Theorem~1.1 of \cite{barbourluczakxia2018}, the interpretation (1) seems to be the correct one. In this case we get that (\ref{eq.conclusion_llt}) holds uniformly for any choice of $\{X_i\}$ satisfying the assumptions of the lemma
with a common constant that depends just on $K$, $\epsilon$, $\epsilon'$ and $\delta'$.

If, however, the correct interpretation is (2), we can still obtain the same conclusion and find a constant $C$ that works for each choice of distributions, but in a less explicit way. 
Since by (\ref{eq.lpm_possible_discretized}) there 
is a constant number $N$ of possible distributions for each ${\bf X}_i$ with
quantized
probabilities $\frac i {K^2}$, $i \in \mathbb{Z}$ in the cube $\mathcal{C}_{M}^r$ 
there is also a constant (at most $N^b$) distributions $\mu'$ that can be constructed
from a convolution of $b'$ such distributions.
Applying (\ref{eq.conclusion_llt}) to each of those distributions we can take the largest constant $C$
for given parameters $K$, $\epsilon$, $\epsilon'$ and $\delta'$.

In any case recalling $\delta' \in (0, \frac 1 4)$ 
for each $K$, $\epsilon$, $\epsilon'$, $\delta'$ we can choose 
constants 
$c''$ 
and
$V_0$, $V_0 \ge 1$, such that if $V \ge V_0$ then (a) (\ref{eq.lpm_eps}) is satisfied 
and 
(b) (\ref{eq.lpm_D}) is satisfied (using  (\ref{eq.lpm_nV0})) and so $m' \ge 2$,
and
\begin{align}\label{eq.conclusion_llt2}
    d_{TV}({\bf S}'', \round{{\bf Z}}) \le c'' n^{-\frac 1 2}\ln n.
\end{align}



 Let $A \subset \mathbb{R}$ be such that $|A| = k$ and $\conc_k(\langle B^T a, {\bf S}'' \rangle) = \pr( \langle B^T a, {\bf S}'' \rangle \in A)$. 
 Then (\ref{eq.conclusion_llt2}) 
 implies that  
 if $V \ge V_0$ then
\begin{align}\label{eq.lpm_standard_coupling}
    \conc_k(\langle B^T a, {\bf S}'' \rangle) = \pr( \langle B^T a, {\bf S}'' \rangle \in A) =  \pr( \langle B^T a, \round{\bf Z} \rangle \in A) \pm \frac {c'' \ln n} {\sqrt n} .
\end{align}
 We are now going to apply Lemma~\ref{lem.proj_coeffs} with 
 ${\bf X} = {\bf Z} - \round{\E {\bf Z}}$,
 $d=r'$ and $\gamma$ defined below. Thus 
 $\mu_0 = \E {\bf X} = \E {\bf Z} - \round{\E {\bf Z}}$ 
satisfies $\|\mu_0\| \le \frac 1 2 \sqrt{r'}$ and the distribution of $\round{\bf X}$
is just a translate of the distribution of $\round{\bf Z}$ by an integer vector $-\round{\E {\bf Z}}$.
We pick the following $\gamma$ for Lemma~\ref{lem.proj_coeffs}:
\begin{equation}
\label{eq.lpm_gamma}
    \gamma := 
    \frac {c_\epsilon^{\frac 3 2} \sqrt{3 \pi}} {\Cnineteen \sqrt{2} K \sqrt{b D}} (K^{2bD}(M'')^2 r)^{-\frac {r-1} 2}.
\end{equation}
Using $\delta' < \frac 1 4$, (\ref{eq.lpm_b_half}), (\ref{eq.lpm_m'}) and (\ref{eq.lpm_lambda_min}), we get the following bound that will be used later:
\begin{equation}\label{eq.lpm_gammaleq}
    \gamma \le \frac 1 2 \frac{\sqrt{12} c_\epsilon} {\Cnineteen K \sqrt{2 c_\epsilon^{-1}}} \frac {\sqrt{2 \pi m' \lambda_{r'}(\LPMSigma)}} {\sqrt{n}}.
\end{equation}

Since ${\bf X} \sim N(\mu_0, m'\LPMSigma)$ the density $p_{\bf X}$ of ${\bf X}$
satisfies for $x \in \mathbb{R}^{r'}$
\begin{align*}
    \|\nabla p_{\bf X}(x)\| =  \|-(m')^{-1}\LPMSigma^{-1}(x - \mu_0)\| & \le (m')^{-1}\lambda_{r'}(\LPMSigma)^{-1}\|x - \mu_0\|
    \\ & \le (\|x\| + \frac {\sqrt{r'}} 2) (m')^{-1} \lambda_{r'}(\LPMSigma)^{-1}.
\end{align*}
Thus, using (\ref{eq.lpm_lambda_min}), (\ref{eq.delta0_gradient}) is satisfied for $f=p_{\bf X}$ whenever
\begin{equation} \label{eq.xbound}
    R_0 \le  (K^{2bD} (M'')^2 r)^{-(r-1)} \delta_0 
    m' - \frac 1 2 \sqrt{r}.
\end{equation}
Using Lemma~\ref{lem.normal_length_tail} and $r' \le r$ 
we get as long as $r \le \frac {t} {16 \sigma_{1}(m' \LPMSigma)}$ 
\[
    \pr(\|{\bf X} - \mu_0\|^2 \ge t) \le \exp(-\frac {t} {4 \sigma_{1}(m' \LPMSigma)}). 
\]
Let $R = \frac 1 2 \sqrt {r} + (m')^{\frac 1 2 + \frac c 2}$ with
$c = \frac 1 2$ (any $c \in (0, 1)$ would work as well). 
Suppose $m'$ is large enough that 
\begin{equation}\label{eq.lpm_m'_chi2}
    (m')^{c} \ge  16 r^2 (M'')^2.
\end{equation} 
Then by (\ref{eq.lpm_sigma_max}) and (\ref{eq.lpm_m'_chi2})
$(m')^{1 + c} \ge 16 r \sigma_1(m' \LPMSigma)$, so using 
Lemma~\ref{lem.normal_length_tail} and (\ref{eq.lpm_sigma_max})
\begin{align}\label{eq.lpmsR}
    & \pr(\|{\bf X}\| \ge R) \le 
    \pr(\|{\bf X} - \mu_0\| \ge R - \|\mu_0\|) \le
    \pr(\|{\bf X} - \mu_0\| \ge R - \frac 1 2\sqrt{r}) \le \nonumber 
    \\  &  \pr(\|{\bf X} - \mu_0\|^2 \ge (m')^{1 + c}) \le
    \exp(-\frac {(m')^{1+c} } {4 \sigma_{1}(m' \LPMSigma)})
    \le \exp(-\frac {(m')^{c} } {4 (M'')^2 r}).
\end{align}

Let us now upper bound the quantity $q$ for our application of Lemma~\ref{lem.proj_coeffs}. Recall that $q = \sup_{e: \|e\|=1} \conc(\langle e, {\bf X} \rangle, 1)$.

The projection $\langle e, {\bf X} \rangle$ is a normal random variable.
The concentration of $\langle e, {\bf X} \rangle$ is maximized by minimizing its variance, 
and the minimum variance is equal to $\lambda_{r'}(m' \LPMSigma) = m' \lambda_{r'}(\LPMSigma)$. 
 
 Thus $q = \conc(Z, 1)$ where $Z \sim N(\zeta, m' \lambda_{r'} (\LPMSigma))$ for some $\zeta \in \mathbb{R}$, so
 \[
     q = \frac 1 {\sqrt{2 \pi m' \lambda_{r'}(\LPMSigma)}} 
     \int_{\zeta - \frac 1 2}^{\zeta + \frac 1 2} \exp^{-\frac {(t-\zeta)^2} {2 m' \lambda_{r'}(\LPMSigma)}} dt
     \le \frac 1 {\sqrt{2 \pi m' \lambda_{r'}(\LPMSigma)}}.
 \]

Recalling (\ref{eq.lpm_k}) and (\ref{eq.lpm_Q_lower}) and (\ref{eq.lpm_gammaleq}) this means 
\begin{equation}\label{eq.lpm_gqk}
    \gamma k q \le \frac 1 2 \conc_k(S').
\end{equation}
Let $R_1 = R_1(r',\gamma)$ be the constant from Lemma~\ref{lem.proj_coeffs}. 
Set $R_0 := R + R_1 = (m')^{\frac 3 4} + \frac 1 2 \sqrt r +  R_1$.
Then there is $n_1:=n_1(K, \epsilon, \epsilon', \delta')$ such that (\ref{eq.xbound}) and (\ref{eq.lpm_m'_chi2}) hold for all $n \ge n_1$.

We now apply Lemma~\ref{lem.proj_coeffs} with ${\bf X}$, $r'$, $\gamma$, $R_0$ and consider the two possible outcomes.

If we have $\conc(\langle a', \round{\bf X}\rangle) \le \gamma q + \pr(\|{\bf X}\| \ge R_0 - R_1)$ for 
$a' = B^T a$
then
by $\delta' < \frac 1 4$, (\ref{eq.lpm_k}),  (\ref{eq.lpm_b_half}), (\ref{eq.lpm_m'}),  (\ref{eq.lpmsR}) and (\ref{eq.lpm_gqk}) 
if $n \ge n_1$ then
\begin{align*}
    & \conc_k(\langle B^T a,  \round{\bf X}\rangle) \le k\conc(
\langle B^T a,  \round{\bf X}\rangle
    ) \le \gamma k q + k \pr(\|{\bf X}\| \ge R_0 - R_1) 
    \\ & \le \frac 1 2 \conc_k(S') + K \sqrt{2 c_\epsilon^{-1} n} \exp\left(-\frac {\sqrt{n} } {4 (M'')^2 r \sqrt{ 2 b D}}\right).
\end{align*}
Recalling (\ref{eq.lpm_Q_lower}) and 
(\ref{eq.lpm_standard_coupling})
we see that there is $n_2 = n_2(K, \epsilon, \epsilon', \delta') \ge n_1$ 
such that
$\conc_k(S) \le  \conc_k(\langle B^T a,  \round{\bf X}\rangle) + c'' n^{- \frac 1 2} \ln n \le  \frac 3 4 \conc_k(S')$
if $n \ge n_2$ and $V \ge V_0$,
which is a contradiction to our initial assumption that $\conc_k(S) \ge (1+\epsilon) \conc_k(S')$.

Thus if 
\begin{equation}\label{eq.nVlarge}
n \ge n_2 \mbox{ and } V \ge V_0,
\end{equation}
the first 
alternative
of Lemma~\ref{lem.proj_coeffs} cannot occur and we assume that the second 
must hold,
i.e.,
there is $\kappa \in \mathbb{R}$ and a positive integer $L_0 = L_0(r', \gamma)$ such that
\[
    \kappa B^T a \in \{-L_0, \dots, L_0\}^{r'}.
\]
So, writing $S''' := \sum_{i \in J''} \kappa X_i =  \langle \kappa B^T a, {\bf S''} \rangle$, we have
\begin{align}\label{eq.lpm_S'''}
    & \conc_k(\langle B^T a, {\bf S''} \rangle) = 
    \conc_k(S'''). 
\end{align}
Until the final paragraph of the proof we will assume (\ref{eq.nVlarge}) holds.
Recalling the definition of ${\bf S''}$, 
and that of its random summands, we note that $S'''$ can be viewed 
\begin{itemize}
    \item either as a sum of $|J''|=b' m' D$ independent random variables $\kappa X_i$, 
    \item or as a sum of $m'D$ i.i.d. random variables $\langle \kappa B^T a, \mathbf{Z}_i'\rangle$ contained in a symmetric integer interval of length at most $r L_0 M'$, 
    \item or a sum of $m'$ i.i.d. random variables $\langle \kappa B^T a, \mathbf{Z}_i''\rangle$  contained in a symetric integer interval of length at most $r L_0 M''$.
\end{itemize}
$\mathbf{Z}_i'$  and $\kappa B^T a$ are integer, so their inner product is also integer. By our assumption $0 \in \supp(X_i)$ for each $X_i$, thus if some $\kappa X_j$, $j \in J''$ was not integer, the random variables  $\langle \kappa B^T a, \mathbf{Z}_i'\rangle$ would not be integer, thus $\kappa X_j$, must be integer 
(and also an extremal random variable for $\alpha_i$) 
for all $j \in J''$. 

Thus our main achievement resulting from Section~\ref{sec.precise} and Section~\ref{sec.projections} was the following:
we started with a sum of $n$ independent extremal integer random variables and we reduced the general case to one where each integer random variable has a \emph{bounded} support, and (after grouping into constant size groups) has the same distribution.


Let $s$ be the maximum span (see Section~\ref{sec.odlyzko_richmond}) of $U+V$, where $U$ and $V$ are independent bounded integer random variables with finite maximum spans $s_U$ and $s_V$. 
We claim that $s=\gcd(s_U, s_V)$.
Without loss of generality we may assume that
the smallest value in the support of $U$ and $V$, and thus also of $U + V$,
is $0$. 
Let us first prove $s$ divides both $s_U$ and $s_V$. 
Suppose $s$ does not divide $s_U$. Let $a = j s_U$
be an element of $\supp(U)$ that is not a multiple of $s$ (if every element of the support of $U$ was a multiple of $s$ then 
$s_U \frac {s}{\gcd(s,s_U)}$ would be a larger span than $s_U$). Then both $0$ and $a$ are in the support of $U+V$,
contradicting the fact that $s$ is a span of $U+V$.
Now we can express each value in the support of $U$ and $V$ by integer multiples of $\gcd(s_U, s_V)$, so it must be 
the maximum span of $U+V$.

The  above claim implies that the maximum span $s$ is the same for both $S'''$ and each of its $m'D$ i.i.d. summands.
It also implies that each of the 
$b' D m'$
original distributions that make up $S'''$
have maximum spans divisible by $s$ (as $\alpha_i \le \frac 1 2$ there
are no distributions supported on a single point). 
Since $0 \in \supp(X_i)$ for each $i \in J''$, this means that $\frac{\kappa} {s} X_i$ is
integer, so without loss of generality we can further assume that $s=1$ (by using $\kappa'=\frac \kappa s$
in place of $\kappa$).

Thus Theorem~\ref{thm.odlyzko_richmond} can be applied to $S''' = \sum_{i=1}^{D m'} \langle \kappa B^T a, {\bf Z}_i' \rangle$ with
$n=D m'$,
$\delta = \frac {b'} 4$ and an appropriate $d$, $0 < d \le r L_0 M'$, to obtain an integer $n_0^{OR}$ such that
the conclusion of the theorem
holds whenever $D m' \ge n_0^{OR}$. Namely, after shifting and unshifting if necessary the theorem yields
\begin{equation}\label{eq.lpm_odlyzko_richmond}
    \pr(S''' = i)^2 \ge \pr(S''' = i+1) \pr(S''' = i-1)
\end{equation}
for any $i \in \mathbb{Z}$ with $\min \supp(S''') + \frac {b'} 4 D m' \le i \le \max \supp(S''') - \frac{b'} 4 D m'$.

Note, however, that
$n_0^{OR}$ may depend on the distribution of $\langle \kappa B^T a, {\bf Z}_1' \rangle$.

As noted above $\langle \kappa B^T a, {\bf Z}_1' \rangle$ is contained in a symmetric integer interval $I$ of length at most $r L_0 M'$. The same holds for each of its $b'$ random summands $\kappa X_i$, as each of them has $0$ in the support.
Furthermore, $\kappa X_i$ have 
quantized
probabilities of the form $\frac {j} {K^2}$, $j \in \mathbb{Z}$. Thus similarly as in (\ref{eq.lpm_possible_discretized}), the number of possible different distributions for $\langle \kappa B^T a, {\bf Z}_1' \rangle$ is constant, for example, it does not exceed $(r L_0 M' + 1)^{b K^2}$.

Applying Theorem~\ref{thm.odlyzko_richmond} with each such possible distribution that has $s=1$ we can pick as $n_0^*=n_0^*(\epsilon, \epsilon', K, \delta)$ the maximum of resulting $n_0^{OR}$ among the applications. 

While $n_0^*$ can be extremely large, it is still a constant that depends just on $K$, $\epsilon$, $\epsilon'$ and $\delta'$ and (\ref{eq.lpm_odlyzko_richmond}) holds whenever 
(\ref{eq.nVlarge}) holds
and $D m' \ge n_0^*$, or, see (\ref{eq.lpm_b'}), (\ref{eq.lpm_m'}),
whenever both
(\ref{eq.nVlarge})
and 
$n \ge 2 b n_0^*$ hold,
which we also assume till the end of the proof.

Using a 
Hoeffing
bound we get that $\pr(|S''' - \E S'''| \ge \frac {b' D m'} {4}) \le 2 \exp(-c'''n)$ for a positive constant $c''' = c'''(K,\epsilon,\epsilon',\delta')$
(more specifically, using (\ref{eq.lpm_b'}), (\ref{eq.lpm_m'}) we can take
$c''' =\frac {bD} {32 (r L_0 M'')^2}$).
Note that, as each of 
independent summand
$\kappa X_i$ is integer-valued and satisfies $\conc(\kappa X_i) = \conc(X_i) \le \frac 1 2$ we have
\[
    \min(\supp(S''')) + \frac {b' D m'} 2 \le \E S''' \le \max(\supp(S''')) -  \frac {b' D m'} 2.
\]
Thus
(\ref{eq.lpm_odlyzko_richmond}) holds for
each integer $i$ such that
$i \in \E S''' \pm \frac {b' D m'} {4}$. 
%
%
So the distribution of $S'''$ is unimodal when restricted to this interval.
For any $k$ 
as in (\ref{eq.lpm_k})
\begin{equation}\label{eq.lpm_chernoff}
    \conc_k(S''') \le \max_{x \in \E S''' \pm \frac {bm'D} 4} \pr(S''' \in (x-\frac k 2,x+ \frac k 2]) + 2\exp(-c'''n). 
\end{equation}
The random variables $|\langle \kappa B^T a, {\bf Z}_i' \rangle|$ are bounded by $r M' L_0$. So their third moments are bounded by $(r M' L_0)^3$. Meanwhile, as
$\kappa X_i$ are integer and
$\conc(\kappa X_i) \le \frac 1 2$ we have $\Var S''' \ge \frac 1 4 b' D m'$.
Therefore 
we can apply the Berry--Esseen theorem to bound 
\begin{align}
    & \pr(S''' \in (x-\frac k 2, x+\frac k 2]) \le \eta \left((\frac {x - \E S''' - \frac k 2} {\sqrt{\Var S'''}}, \frac {x- \E S'''  + \frac k 2  } {\sqrt{\Var S'''}}]\right) \nonumber
    \\& + \frac {2 (r M' L_0)^3 D m'} {(\frac 1 4 b' D m')^{\frac 3 2}} 
    \le   \eta \left( (\frac {-k} {2 \sqrt{\Var S'''}}, \frac {k} {2\sqrt{\Var S'''}}] \right) + \frac {c''''} {\sqrt{n}} \label{eq.lpmbe},
\end{align}
where $\eta$ is the standard normal measure and 
$c'''' = \frac {2 \cdot \Clpmbee (r M' L_0)^3 } {b}$,
again see  (\ref{eq.lpm_b'}) and (\ref{eq.lpm_m'}), and recall that $\delta' < \frac 1 4$. 
The first term is maximized by minimizing $\Var S'''$.
With our setup this is achieved precisely by taking $S''' = \sum_{i\in J''} Y_{i}$,
as by
Corollary~1.3 of~\cite{ak2025}
$Y_i$ has minimum variance
among all extremal random variables for $\alpha_i$ for each $i \in J''$.
To finish the proof, it remains to choose $\epsilon'$ and $\delta'$ so that
the effect of
\[
    \sum_{i\in \{1, \dots,n\} \setminus J''}^n Y_{i}
\]
to the sum of $k$ largest probability masses of $\sum_{i=1}^n Y_i$
is negligible.
Note that $V + \Var U_k' \ge V \ge \frac n 4$ and recall (\ref{eq.lpm_m'}) which shows that $|\{1,\dots,n\} \setminus J''| \le 2\delta n$.
Set 
\[
\delta' := \frac {\epsilon^3} {2 \cdot 70 \cdot 4^4 K^2}.
\]
This implies
\[
\frac n 4  \ge \frac {70 \cdot 4^3}{\epsilon^3} 2 \delta' n K^2.
\]
Therefore, applying Lemma~\ref{lem.few_dropped} we get
\[
    \conc_k(\sum_{i=1}^n Y_i) \ge (1-\frac {\epsilon} 4)  \conc_k(\sum_{i \in J''} Y_{i}). 
\]
 By (\ref{eq.lpm_Q_lower}) there is a constant $n_3$, $n_3 = n_3(K, \epsilon, \epsilon', \delta')$ such that
\begin{align*}
   &\conc_k(S')^{-1} (c'' n^{-\frac 1 2} \ln n + 2 \exp(-c'''n)  +  c''''n^{-\frac 1 2}) 
\\ & \le  \frac { {\Cnineteen}} {\sqrt{12} c_\epsilon} ( c'' n^{-\frac 1 2} \ln n +   2 \exp(-c'''n) + c''''n^{-\frac 1 2}) \le \frac \epsilon 4
\end{align*}
 for all $n \ge n_3$.
Thus, since $\epsilon \in (0,1)$ we get  by (\ref{eq.lpm_standard_coupling}), (\ref{eq.lpm_S'''}), (\ref{eq.lpm_chernoff}) and (\ref{eq.lpmbe})
that if $n \ge \max(n_2, n_3, 2 b n_0^*)$ and $V \ge V_0$ then
\begin{align*}
    &\conc_k(S) \le \conc_k(\sum_{i\in J''} Y_{i}) + c'' n^{-\frac 1 2} \ln n
   + 2 \exp(-c'''n) + c'''' n^{-\frac 1 2} 
    \\& \le  (1-\frac {\epsilon} 4)^{-1} \conc_k(S') + c'' n^{-\frac 1 2} \ln n + 2\exp(-c'''n) +  c'''' n^{-\frac 1 2} 
    \\ & \le \left(1 + \frac \epsilon 2 +  \frac {\epsilon} 4\right) \conc_k(S')
    \le (1+\frac{3 \epsilon} 4) \conc_k(S') 
\end{align*}
which is a contradiction to the original assumption $\conc_k(S) \ge (1+\epsilon) \conc_k(S')$. Finally, note that 
(\ref{eq.lpm_nV0}) implies $n \ge \max(n_2,n_3, 2bn_0^*)$ when 
\[
V \ge V_{\epsilon, K} := \max\left(V_0, \frac {K^2 \max(n_2,n_3, 2bn_0^*))} {12}\right).
\]
\end{proofof}

\medskip

\textbf{Acknowledgment}. 
I thank Tomas Juškevičius for proposing to work on this problem, for the motivation, and for helpful discussions.

\phantomsection

\newpage

\appendix

\section{Appendix}
\label{sec.appendix}

Here we list some auxiliary lemmas.

Note that by the Cauchy--Schwarz 
inequality we have for any $x \in \mathbb{R}^d$
\begin{equation}\label{eq.cauchy_schwarz}
    \|x\|_1 = {\sum_{i=1}^d |x_i|} \le \sqrt{d} \|x\|.
\end{equation}

We will use the next almost trivial lower bound for singular values of an integer matrix.

\begin{lemma}
    \label{lem.singular_value_lower_bound}
    Let $A$ be an $m \times n$ integer matrix of full column rank (so $n \le m$).
    Let $R > 0$.
    Suppose
    that $\|v_i\| \le R$ for each $i \in \{1, \dots, n\}$,
    where $v_i$ is the $i$th column vector of $A$.
    Then the smallest singular value of $A$ satisfies
    \[
       \sigma_n(A) \ge (\sqrt n R)^{-(n-1)}.
    \]
\end{lemma}

\begin{proof}
    A basic property of singular value decomposition is that
    \begin{align*}
        & \sigma_1(A) = \|A\| = \max_{\|z\|=1} \| A z \| = \max_{\|z\|=1}\|\sum_i z_i v_i\| 
        \\ & \le \max_i \|v_i\| \max_{\|z\|=1} \|z\|_1 \le \sqrt {n} R.
    \end{align*}
    Here the first inequality follows since $\|\sum_i z_i v_i\| \le \sum_i |z_i| \|v_i\|$ by the triangle inequality and the
    second one follows by (\ref{eq.cauchy_schwarz}). 
    So using the property $\sigma_1(A) \dots \sigma_n(A) = \sqrt{\det A^T A}$ we have
    \[
        \sigma_{n}(A) \ge \frac { \sqrt{\det A^T A}} {\sigma_1(A)^{n-1}} \ge \frac 1 {\sigma_1(A)^{n-1}} \ge (\sqrt n R)^{-(n-1)}.
    \]
\end{proof}

\begin{lemma}\label{lem.korkinzolotarev}
    Fix a positive integer $r$.
    Let $S \subseteq \mathbb{Z}^r$ be such that $\|x\| \le R$
    for all $x \in S$.

    Let $L$ be the set generated by
    all finite linear combinations of 
    vectors from $S$ with integer coefficients.

    Suppose $0 \in S$. 
    Let $r'$ be the dimension of 
    the subspace of $\mathbb{R}^r$ spanned by $S$.

    Then there is an $r \times r'$ integer matrix $A$ such that 
    \[
        S \subseteq \{A x: x \in \mathbb{Z}^{r'}, \|x\| \le c(R, r) \}
    \]
    where $c(R, r) =   R(\sqrt {\frac {r(r+3)} 4} R^2)^{r-1} \le  r^{r-1} R^{2r-1}$.
\end{lemma}
\begin{proof}
    $L$ is a lattice of rank $r'$, see, e.g. \cite{korkinbases}.
    It has a \emph{lattice basis} consisting of $r'$ vectors in $\mathbb{Z}^{r'}$
    such that $L$ is equal to the set of their linear combinations with integer coefficients.

    It is known in the literature, see, e.g., Theorem~2.1 of \cite{korkinbases},
    that $L$ has a so called \emph{Korkin--Zolotarev} lattice basis $v_1, \dots, v_{r'}$ which satisfies
    \[
        \|v_i\| \le \sqrt {\frac {i+3} 4} \lambda_i(L)^2, \quad i \in \{1, \dots, r'\},
    \]
    where $\lambda_i(L)$ is the $i$th \emph{successive minimum} of $L$, 
    which is defined as the infimum among the numbers $s$ such the closed ball $B(0, s)$ in $\mathbb{R}^{r}$ contains $i$ linearly independent vectors from $L$.
    Since our set $S$ must contain $r'$ linearly independent vectors from $L$, we have
    $\lambda_i(L) \le R$ for all $i \in \{1, \dots, r'\}$.
    So $\|v_i\| \le  \sqrt {\frac {r'+3} 4} R^2$. 

    Let $A$ be an $r \times r'$ matrix with the Korkin--Zolotarev basis vectors $v_1, \dots, v_{r'}$ as columns.
    Let
    $x \in S \subseteq L$
    and let $y \in \mathbb{Z}^{r'}$ be the 
    unique vector such that $x = A y$. 
    Since $A$ has full column rank, it has a pseudoinverse 
    $(A^T A)^{-1} A^T$.
    Thus
    \[
        y = (A^T A)^{-1} A^T x
    \]
    and
    \[
        \|y\| \le \|(A^T A)^{-1} A^T\| \|x\| \le \sigma_1((A^T A)^{-1} A^T) R
    \]
    A basic property of the pseudoinverse of the matrix $A$ is that it has singular values which are reciprocals of the singular values of $A$. Thus $\|(A^T A)^{-1} A^T\| = \sigma_{r'}^{-1}(A)$.
    Another basic fact is that $\sigma_1(A) \times \dots \times \sigma_{r'}(A) = \sqrt{\det A^T A} \ge 1$, here the inequality is because $A$ has integer values and the matrix $A$ has rank $r'$.

    By Lemma~\ref{lem.singular_value_lower_bound}
    \[
        \sigma_{r'}(A) \ge 
         (\sqrt {\frac {r'(r'+3)} 4} R^2)^{-(r'-1)}
    \]
    so
    \[
        \|y\| \le R \sigma_{r'}(A)^{-1} \le c(R, r):= R (\sqrt {\frac {r(r+3)} 4} R^2)^{r-1}.
    \]
\end{proof}

\begin{lemma}\label{lem.normal_length_tail}
    Let ${\bf X}$ be a random vector in $\mathbb{R}^d$ with normal distribution
    of mean zero and a positive definite covariance matrix $\Sigma$. Let $\sigma_1(\Sigma)$ be the largest singular value of $\Sigma$. Then for any $t > 0$ such that
    $d \le \frac t {16 \sigma_1(\Sigma)}$ we have
    \[
        \pr(\|{\bf X}\|^2 \ge t) \le \exp(-\frac {t} {4 \sigma_1(\Sigma)}).
    \]
\end{lemma}

\begin{proof}
    A Laurent and Massart bound (Lemma~1 of \cite{laurentmassart})
    says that for a standard normal random vector ${\bf Z}$ in $\mathbb{R}^d$ the corresponding
    $\chi$-squared random variable $\|{\bf Z}\|^2$ satisfies
\[
    \pr(\|{\bf Z}\|^2 - d \ge 2 \sqrt{ d t} + 2 t) \le e^{-t}.
\]
    So if $d \le \frac t 4$ we have $\pr(\|{\bf Z}\|^2 \ge 4 t) \le e^{-t}$.
    Equivalently, if $d \le \frac t {16}$ then
    \begin{equation}\label{eq.tail_chi_sq}
        \pr(\|{\bf Z}\|^2 \ge t) \le e^{- \frac t 4} 
    \end{equation}
    We have ${\bf X} \sim \Sigma^{\frac 1 2} {\bf Z}$ so
    $\|{\bf X}\|^2 = \langle {\bf X}, {\bf X} \rangle =
    {\bf Z}^T (\Sigma^{\frac 1 2})^T \Sigma^{\frac 1 2} {\bf Z} = \|{\bf Z}\|^2 (\frac {{\bf Z}} {\|{\bf Z}\|})^T \Sigma  \frac {{\bf Z}} {\|{\bf Z}\|}$.
    By the properties of singular values, the maximum of $x^T \Sigma x$ over unit length vectors $x$ is exactly
    the largest singular value $\sigma_1(\Sigma)$.
    So  $\|{\bf X}\|^2 \le \|{\bf Z}\|^2 \sigma_1(\Sigma)$. 

    So if $d \le \frac t {16 \sigma_1(\Sigma)}$, we have by  (\ref{eq.tail_chi_sq})
    \begin{align*}
        &\pr(\|{\bf X}\|^2 \ge t) \le  \pr( \|{\bf Z}\|^2 \sigma_1(\Sigma) \ge t)
        \le \exp(- \frac {t} {4 \sigma_1(\Sigma)}).
    \end{align*}
\end{proof}

\end{document}